\documentclass[leqno,english]{smfart}
\usepackage{amscd,amssymb,amsmath,subfigure}
\usepackage[all]{xy}
\usepackage[colorlinks,plainpages,backref,urlcolor=blue]{hyperref}
\usepackage{tikz}
\usetikzlibrary{calc}

\newtheorem{theorem}{Theorem}[section]
\newtheorem{corollary}[theorem]{Corollary}
\newtheorem{lemma}[theorem]{Lemma}
\newtheorem{prop}[theorem]{Proposition}

\theoremstyle{definition}
\newtheorem{definition}[theorem]{Definition}
\newtheorem{example}[theorem]{Example}
\newtheorem{remark}[theorem]{Remark}

\newtheorem{problem}[theorem]{Problem}

\newenvironment{romenum}
{\renewcommand{\theenumi}{\roman{enumi}} 
\renewcommand{\labelenumi}{\theenumi.}
\begin{enumerate}}{\end{enumerate}}

\newcommand{\N}{\mathbb{N}}
\newcommand{\Z}{\mathbb{Z}}
\newcommand{\Q}{\mathbb{Q}}
\newcommand{\R}{\mathbb{R}}
\newcommand{\C}{\mathbb{C}}
\newcommand{\F}{\mathbb{F}}
\newcommand{\T}{\mathbb{T}}
\newcommand{\CP}{\mathbb{CP}}

\renewcommand{\P}{\mathbb{P}}

\newcommand{\LL}{\mathcal{L}}
\newcommand{\sV}{\mathsf{V}}
\newcommand{\sE}{\mathsf{E}}

\renewcommand{\k}{\Bbbk}
\newcommand{\RR}{\mathcal{R}}
\newcommand{\VV}{\mathcal{V}}
\newcommand{\V}{\mathcal{V}}
\newcommand{\A}{{\mathcal{A}}}
\newcommand{\B}{{\mathcal{B}}}

\newcommand{\cT}{\mathsf{T}}

\newcommand{\XX}{{\mathcal X}}
\newcommand{\cM}{{\mathcal{M}}}

\DeclareMathOperator{\rank}{rank}
\DeclareMathOperator{\gr}{gr}
\DeclareMathOperator{\im}{im}

\DeclareMathOperator{\id}{id}
\DeclareMathOperator{\ab}{{ab}}

\DeclareMathOperator{\ch}{char}

\renewcommand{\sl}{\mathfrak {sl}}

\DeclareMathOperator{\Hom}{{Hom}}

\DeclareMathOperator{\Poin}{{Poin}}
\DeclareMathOperator{\spn}{span}

\DeclareMathOperator{\const}{const}

\DeclareMathOperator{\depth}{depth}

\DeclareMathOperator{\TC}{TC}

\DeclareMathOperator{\orb}{orb}

\newcommand{\PL}{\scriptscriptstyle{\rm PL}}
\newcommand{\ii}{\mathrm{i}}
\newcommand{\bo}{{\mathbf 1}}

\newcommand{\Uc}{{\overline{U}}}
\newcommand{\bdU}{{\partial{\Uc}}}
\newcommand{\Fc}{{\overline{F}}}
\newcommand{\bdF}{{\partial{\Fc}}}
\newcommand{\G}{\Gamma}

\newcommand{\qam}{{Q_m(\A)}}
\newcommand{\fma}{{F_m(\A)}}

\newcommand{\surj}{\twoheadrightarrow}
\newcommand{\inj}{\hookrightarrow}

\newcommand{\abs}[1]{\left| #1 \right|}

\def\angl#1{{\langle #1\rangle}}
\def\set#1{{\{ #1\}}}
   
\newcommand{\norm}[1]{|\!|#1|\!|}
\newcommand{\cv}{\check{{\mathcal{V}}}_1}
\newcommand{\nuc}{{\stackrel{\circ}{\nu}}}

\newcommand{\co}{\!:\!}

\definecolor{dkgreen}{RGB}{0,100,0}
\definecolor{dkbrown}{RGB}{139,69,19}

\begin{document}

\title[Hyperplane arrangements and Milnor fibrations]{%
Hyperplane arrangements and Milnor fibrations} 

\alttitle{Arrangements d'hyperplans et fibr\'{e}s de Milnor}

\author[Alexander~I.~Suciu]{Alexander~I.~Suciu}
\address{Department of Mathematics,
Northeastern University,
Boston, MA 02115, USA}
\email{a.suciu@neu.edu}
\urladdr{http://www.northeastern.edu/suciu/}
\thanks{Partially supported by NSF grant DMS--1010298}



\begin{abstract}
There are several topological spaces associated to a 
complex hyperplane arrangement:  the complement 
and its boundary manifold, as well as the Milnor fiber 
and its own boundary.  All these spaces are related in 
various ways, primarily by a set of interlocking fibrations. 
We use cohomology with coefficients in rank $1$ local 
systems on the complement of the arrangement to gain 
information on the homology of the other three spaces, 
and on the monodromy operators of the various fibrations.  
\end{abstract}

\begin{altabstract}
\'{E}tant donn\'{e} un arrangement d'hyperplans, il y a 
plusieurs espaces topologiques qu'on puisse lui associer: 
le compl\'{e}mentaire et sa vari\'{e}t\'{e} de bord, ainsi que 
la fibre de Milnor et son bord.  Tous ces espaces sont 
reli\'{e}s, en premier lieu par des fibrations.  On utilise 
la cohomologie avec coefficients dans les syst\`{e}mes 
locaux de rang $1$ sur le compl\'{e}mentaire d'un arrangement 
d'hyperplans pour \'{e}tudier l'homologie des trois autres 
espaces, et les op\'{e}rateurs de monodromie des fibrations 
associ\'{e}es. 
\end{altabstract}
\maketitle

\tableofcontents

\section{Introduction}
\label{sect:intro}

\subsection{Arrangements of hyperplanes} 
\label{subsec:intro1}

This paper is mostly an expository survey, centered on the 
topology of complements of hyperplane arrangements, their Milnor 
fibrations, and their boundary structures.  The presentation is loosely 
based on a set of notes for a mini-course given at the conference 
``Arrangements in Pyr\'{e}n\'{e}es", held in Pau, France in 
June 2012.  Although we expanded the 
scope of those notes, and provided many more details 
and explanations, we made every effort to maintain the 
original spirit of the lectures, which was to give a 
brisk, self-contained introduction to the subject, 
and provide motivation for further study.

An arrangement is a finite collection of hyperplanes in 
a finite-dimensional, complex vector space. There are 
various ways to understand the topology of such an object.  
In this paper, we describe several topological 
spaces associated to a hyperplane arrangement, 
all connected to each other by means of inclusions, bundle 
maps, or covering projections. Associated to these spaces, 
there is a plethora of topological invariants 
of an algebraic nature:  fundamental group and lower 
central series, Betti numbers and torsion coefficients, cohomology 
ring and Massey products, characteristic and resonance 
varieties, and so on.  One of the main goals of the subject 
is to decide whether a given invariant is combinatorially 
determined, and, if so, to express it explicitly in terms of 
the intersection lattice of the arrangement.   

As the title indicates, the focus of the paper is on the 
Milnor fibration of the complement of a hyperplane arrangement.  
We use cohomology with coefficients in rank $1$ local 
systems on the complement to compute the homology 
groups of the Milnor fiber, with coefficients in a field 
not dividing the number of hyperplanes, and determine the 
characteristic polynomial of the  monodromy operator 
acting on these homology groups. 
In the process, we also show how to compute various homological 
invariants of the boundary of the complement and the 
boundary of  the Milnor fiber.

\subsection{The complement} 
\label{subsec:intro2}

Given an arrangement $\A$ in $\C^{d+1}$, the most 
direct approach is to study the complement, 
$M=\C^{d+1} \setminus \bigcup_{H\in\A} H$. 
The fundamental group of this space can be 
computed algorithmically, using the braid monodromy 
associated to a generic plane section.  The cohomology 
ring of the complement was computed by Brieskorn in \cite{Br}.  
His result shows that $M$ is a formal space; 
in particular, all rational Massey products vanish. 
In \cite{OS}, Orlik and Solomon gave a simple 
combinatorial description of the algebra 
$H^*(M,\Z)$:  it is the quotient of the exterior 
algebra on classes dual to the meridians, modulo 
a certain ideal determined by the intersection lattice, $L(\A)$.
A comprehensive treatment of this topic can be 
found in \cite{OT}.

Starting in the mid-to-late 1990s, the subject underwent a considerable 
shift of emphasis towards the study of the cohomology jump 
loci of the complement.  These loci come in two basic flavors: 
the characteristic varieties, which are the jump loci for cohomology 
with coefficients in complex, rank one local systems, and the 
resonance varieties, which are the jump loci for the homology 
of the cochain complexes arising from multiplication classes in $H^1(M,\C)$.  

Since the complement of the arrangement 
is a smooth, quasi-projective variety, its characteristic 
varieties are finite unions of torsion-translates of algebraic subtori 
of the character group $\Hom(\pi_1(M),\C^*)$.  
Since the complement is a formal space, 
the resonance varieties of $M$ coincide 
with the tangent cone at the origin to the corresponding 
characteristic varieties, and thus are finite unions 
of rationally defined linear subspaces of $H^1(M,\C)$.  

For the degree-$1$ resonance varieties, these subspaces 
were described combinatorially by Libgober and Yuzvinsky \cite{LY00}, 
and later by Falk and Yuzvinsky \cite{FY07}, solely 
in terms of multinets on sub-arrangements of $\A$.  In general, 
though, the degree-$1$ characteristic varieties of an arrangement 
may contain components which do not pass through the origin, 
and it is still an open problem whether such components are 
combinatorially determined.

The characteristic and resonance varieties of an arrangement complement  
may also be defined for (algebraically closed) fields of positive 
characteristic.  The nature of these varieties is less predictable, 
though.  For instance, as noted in \cite{MS00, S01, Fa07}, both 
the tangent cone formula and the linearity of the resonance 
components fail in this setting.  Furthermore, as shown 
in \cite{Ma07}, non-vanishing Massey triple products may appear 
in positive characteristic. 

\subsection{The Milnor fibration} 
\label{subsec:intro3}

A more refined topological invariant of an arrangement $\A$ is 
the Milnor fibration of its complement.  For each hyperplane 
$H\in \A$, choose a linear form $f_H$ whose kernel is $H$. 
The polynomial $Q=\prod_{H\in\A} f_H$, then, is homogeneous 
of degree $n=\abs{\A}$. As shown by Milnor \cite{Mi} 
in a more general context, the restriction $Q\colon  M\to \C^*$ 
is a smooth fibration. The typical fiber, $F=Q^{-1}(1)$, is called 
the Milnor fiber of the arrangement.  We give a topological description 
of this fibration, building on previous joint work with Cohen and 
Denham \cite{CS95, CDS03, DeS12}.

It turns out that $F$ is a regular, cyclic $n$-fold cover of the 
projectivized complement, $U=\P(M)$. The classifying homomorphism 
for this cover, $\delta\colon \pi_1(U)\to \Z_n$, takes each meridian 
generator to $1$.  Embedding $\Z_n$ into $\C^*$ by 
sending $1$ to a primitive $n$-th root of unity, we may 
view $\delta$ as a character on $\pi_1(U)$. The relative 
position of this character with respect to the characteristic 
varieties of $U$ determines the Betti numbers of the 
Milnor fiber $F$, as well as the characteristic polynomial 
of the algebraic monodromy.  

More generally, given multiplicities $m_H\ge 1$ for each hyperplane 
$H\in A$, we consider the Milnor fibration of the multi-arrangement 
$(\A,m)$, defined by the homogeneous polynomial 
$Q(\A,m)=\prod_{H\in \A} f_H^{m_H}$. The Milnor fiber 
of this polynomial is now a regular, $N$-fold cyclic cover 
of the complement, where $N$ is the sum of the multiplicities. 
Not too surprisingly, the Betti numbers and other topological invariants 
of the Milnor fiber $F(\A,m)$ vary with the choice of $m$. 

Although the complement of an arrangement and its Milnor 
fiber share some common features (for instance, they are 
both smooth, quasi-projective varieties), there are some 
striking differences between the two.  For one, the 
homology groups of the Milnor fiber need not be torsion-free, 
as recent examples from \cite{DeS12} show.  For another, the 
Milnor fiber may have non-vanishing Massey products, 
and thus be non-formal, as examples from \cite{Zu} show. 
Finally, as noted in \cite{DP11}, the cohomology jump loci 
of the Milnor fiber may differ from those of the complement. 
All these novel qualitative features are related to the nature 
of multinets supported by the arrangement in 
question, or one of its sub-arrangements. 

\subsection{Boundary structures} 
\label{subsec:intro4}

Both the projectivized complement  and the Milnor fiber 
of an arrangement are non-compact manifolds (without boundary).   
Replacing these spaces by their compact versions allows us 
to study the behavior of the complement and the Milnor fibration 
as we approach the boundary, thereby revealing otherwise  
hidden phenomena. 
 
Cutting off a regular neighborhood of the arrangement $\A$ 
yields a compact manifold with boundary, $\Uc$, onto which $U$  
deform-retracts.  The boundary manifold of the arrangement, then, 
is the smooth, compact, orientable $(2d-1)$-dimensional manifold $\bdU$. 
The homology groups are torsion-free, and the ring 
$H^*(\bdU,\Z)$ is functorially determined by $H^*(U,\Z)$, 
via a ``doubling" construction. 

Especially interesting is the case $d=2$, when the manifold $\bdU$ 
is a graph-manifold, in the sense of Waldhausen. 
Following the approach from \cite{CS06, CS08}, we describe 
the fundamental group, the cohomology ring, and the cohomology 
jump loci of $\bdU$ in terms of the underlying graph structure, 
which in turn can be read off the intersection lattice of $\A$.  
Yet significant differences 
with the complement exist.  For one, the boundary manifold  
is never formal, unless $\A$ is a pencil or a near-pencil. 
For another, the resonance varieties of $\bdU$ may have non-linear 
components.

Intersecting now the Milnor fiber with a ball in $\C^{d+1}$ centered at 
the origin yields a compact manifold with boundary, $\Fc$, onto which 
$F$ deform-retracts.  The boundary of the Milnor fiber, then, is the 
smooth, compact, orientable $(2d-1)$-dimensional manifold $\bdF$. 
We observe here that $\bdF$ is a regular, cyclic $n$-fold cover of $\bdU$, 
where $n=\abs{\A}$, and identify a classifying homomorphism for this cover.  

When $d=2$, the boundary of $\Fc$ is a $3$-dimensional 
graph-manifold. Using work of N\'{e}methi and Szil\'{a}rd \cite{NS}, 
we give a formula for the characteristic polynomial of the monodromy
operator acting on $H_1(\bdF,\C)$, purely in terms of the M\"{o}bius 
function of $L_{\le 2}(\A)$.  We also show that the manifold $\bdF$ 
is never formal, unless $\A$ is a pencil or a near-pencil, and point 
out that $H_1(\bdF,\Z)$ typically has non-trivial $n$-torsion. 

\subsection{Organization of the paper}
\label{subsec:intro5}
The paper is divided roughly into three parts, following the 
approach outlined so far in this introduction. 

The first part deals with the complement of a hyperplane arrangement $\A$. 
In \S\ref{sect:ma1}, we discuss the combinatorics of $\A$, as it relates to the 
cohomology ring and resonance varieties of the complement, with special 
emphasis on multinets. In \S\ref{sect:ma2}, we study the topology of $M(\A)$, 
as reflected in the fundamental group and characteristic varieties, 
with special emphasis on orbifold fibrations and translated 
subtori in those varieties. 

The second part covers the Milnor fibration defined by 
a multi-arrangement. In \S\ref{sect:mf arr} we set up 
in all detail the regular covers into which the 
corresponding Milnor fiber fits.  In \S\ref{sect:mf bis},  
we study the homology of the Milnor fiber and the 
monodromy action, and discuss the formality 
properties and the cohomology jump loci of the Milnor fiber.

The third part deals with the boundary structures 
associated to an arrangement. In \S\ref{sect:boundary} 
we discuss the boundary manifold of an arrangement, 
while in \S\ref{sect:bdry mf} we discuss the boundary 
of the Milnor fiber. 

Sprinkled throughout the text there are about two dozen 
open problems. Most of these problems have been raised 
before; some are well-known, but some appear here for the 
first time.

We collect the necessary background material in three appendices 
at the end. Appendix \ref{sect:cvs} serves as a quick introduction to the 
cohomology jump loci of a space, and some of their properties.   
Appendix \ref{sect:covers} deals with the homology groups and 
jump loci of finite, regular abelian covers of a space.  
Finally, Appendix \ref{sect:formal} discusses formality 
properties of spaces, especially as they pertain to finite 
covers and jump loci.

\subsection{Acknowledgements}
\label{subsec:ack}
I wish to thank the organizers of the Arrangements in Pyr\'{e}n\'{e}es
conference for giving me the opportunity to lecture in such a  
wonderful setting, and for their support and hospitality. 
I also wish to thank the University of Sydney for its support
and hospitality while most of this paper was written up.

\section{The complement of an arrangement. I.}
\label{sect:ma1}

In this section we describe the cohomology ring and the 
resonance varieties of the  
complement of a complex hyperplane arrangement.

\subsection{Hyperplane arrangements}
\label{subsec:hyp arr}

An {\em arrangement of hyperplanes}\/ is a finite set $\A$ of 
codimension-$1$ linear subspaces in a finite-dimensional, 
complex vector space $\C^{d+1}$.  
Throughout, we will assume that the arrangement is central, 
that is, all the hyperplanes pass through the origin.  
The combinatorics of the arrangement is encoded in its 
{\em intersection lattice}, $L(\A)$; this is the poset of all 
intersections of $\A$, ordered by reverse inclusion, and 
ranked by codimension.   The arrangement is said to 
be essential if the intersection of all flats in $L(\A)$ 
is $\{0\}$. 

For each hyperplane $H\in \A$, let $f_H\colon \C^{d+1} \to \C$ 
be a linear form with kernel $H$. The product 
\begin{equation}
\label{eq:qa}
Q(\A)=\prod_{H\in \A} f_H,
\end{equation}
then, is a defining polynomial for the arrangement, 
unique up to a (non-zero) constant factor. 
Notice that $Q(\A)$ is a homogeneous polynomial 
of degree equal to $\abs{\A}$, the cardinality of the set $\A$.

On occasion, we will allow multiplicities on the hyperplanes. 
A {\em multi-arrangement}\/ is a pair $(\A,m)$, where $\A$ is a 
hyperplane arrangement, and $m\colon \A\to \Z$ is a function 
with $m_H\geq1$ for each $H\in\A$. We may also assign 
multiplicities to subspaces $X\in L(\A)$, by letting 
$m_X=\sum_{H\leq X}m_H$.
A defining polynomial for the multi-arrangement $(\A,m)$ 
is the homogeneous polynomial 
\begin{equation}
\label{eq:qam}
\qam=\prod_{H\in \A} f_H^{m_H}.
\end{equation}

\subsection{The complement}
\label{subsec:ma}

The main topological invariant associated to an arrangement $\A$ 
is its complement, $M(\A)=\C^{d+1}\setminus \bigcup_{H\in \A} H$.  
This is a smooth, quasi-projective variety, with the homotopy type 
of a connected, finite CW-complex of dimension $d+1$.   

\begin{example}
\label{ex:boolean}
The Boolean arrangement $\B_n$ consists of the coordinate 
hyperplanes $H_i=\set{z_i=0}$ in $\C^{n}$. The intersection 
lattice is the Boolean lattice of subsets of $\set{0,1}^{n}$, 
ordered by reverse inclusion, while the complement is the 
complex algebraic torus $(\C^*)^{n}$.
\end{example} 

\begin{example}
\label{ex:config}
The reflection arrangement of type ${\rm A}_{n-1}$, also known 
as the braid arrangement, consists of the diagonal hyperplanes 
$H_{ij}=\set{z_i-z_j=0}$ in  $\C^{n}$. The intersection 
lattice is the lattice of partitions of $[n]=\set{1,\dots,n}$, 
ordered by refinement.  The complement is the configuration 
space $F(\C,n)$ of $n$ ordered points in $\C$.  
In the early 1960s, Fox and Neuwirth showed that 
$\pi_1(F(\C,n))=P_{n}$, 
the pure braid group on $n$ strings, while Neuwirth and 
Fadell showed that $F(\C,n)$ is aspherical. 
\end{example} 

The group $\C^*$ acts freely on $\C^{d+1}\setminus \set{0}$ via 
$\zeta\cdot (z_0,\dots,z_{d})=(\zeta z_0,\dots, \zeta z_{d})$. 
The orbit space is the complex projective space of dimension $d$, 
while the orbit map, $\pi\colon \C^{d+1}\setminus \set{0} \to \CP^{d}$,  
$z \mapsto [z]$, 
is the Hopf fibration. The set  $\P(\A)=\set{\pi(H)\colon H\in \A}$ is an 
arrangement of codimension $1$ projective subspaces in $\CP^{d}$; 
its complement, $U(\A)$, coincides with the quotient 
$\P(M(\A))=M(\A)/\C^*$. 

The Hopf map restricts to a bundle map, 
$\pi(\A)\colon M(\A)\to U(\A)$, with fiber $\C^{*}$. 
Fixing a hyperplane $H\in \A$, we see that $\pi(\A)$ is 
also the restriction of the bundle map 
$\pi\colon \C^{d+1}\setminus H\to \CP^{d} 
\setminus \pi(H) \cong \C^{d}$.  Clearly, this latter bundle 
is trivial; hence, we have a diffeomorphism  
\begin{equation}
\label{eq:mu}
M(\A) \cong U(\A)\times \C^*.
\end{equation}

\begin{example}
\label{ex:pencil}
Let $\mathcal{P}_n$ be the arrangement of $n+1$ lines in 
$\C^2$ defined by the polynomial $Q=x^{n+1}-y^{n+1}$.  
Then $\P(\mathcal{P}_n)$ consists of $n+1$ 
points in $\CP^1$. Thus, $U(\mathcal{P}_n)\cong 
\C\setminus \{\text{$n$ points}\}$, and so $M(\mathcal{P}_n)$ is 
homotopy equivalent  to $S^1 \times \bigvee^{n} S^1$. 
\end{example}

\subsection{Cohomology ring}
\label{subsec:os alg}

The (integral) cohomology ring of a hyperplane arrangement 
complement was computed by Brieskorn in \cite{Br}, 
building on pioneering work of Arnol'd on the cohomology 
ring of the pure braid group.   In \cite{OS}, Orlik and 
Solomon gave a simple description of this ring, solely 
in terms of the intersection lattice of the arrangement.  

Let $\A$ be an arrangement, with complement $M=M(\A)$. 
Fix a linear order on $\A$, and let $E$ be the exterior 
algebra over $\Z$ with generators $\set{e_H \mid H\in \A}$ 
in degree $1$.  Next, define a differential $\partial \colon E\to E$ 
of degree $-1$, starting from $\partial(1)=0$ 
and $\partial(e_H)=1$, and extending $\partial$ to 
a linear operator on $E$, using the graded Leibniz rule 
$\partial (ab)=\partial(a) b + (-1)^{\deg a} a \partial (b)$ 
for homogeneous elements $a$ and $b$. 
Finally, let $I$ be the ideal of $E$ generated by $\partial e_{\B}$, 
where $\B$ runs through all sub-arrangements of $\A$ which 
are not in general position, 
and $e_\B=\prod_{H\in \B} e_H$.  Then $H^*(M,\Z)$ is 
isomorphic, as a graded $\Z$-algebra, to the quotient ring $A=E/I$.  

Under this isomorphism, the basis $\{e_H\}$ of
$A^1$ is dual to the basis of $H_1(M,\Z)$
given by the meridians $\{x_H\}$ around the hyperplanes, 
oriented compatibly with the complex orientations on $\C^{d+1}$ 
and the hyperplanes.   Furthermore, all the homology groups of $M$ 
are torsion-free; the generating function for their ranks 
is given by
\begin{equation}
\label{eq:poin}
\Poin(M(\A),t)=\sum_{X\in L(\A)} \mu(X) (-t)^{\rank(X)}, 
\end{equation}
where $\mu\colon L(\A) \to \Z$ is the M\"{o}bius function 
of the intersection lattice, given inductively by $\mu(\C^{d+1})=1$ 
and $\mu(X)=-\sum_{Y\supsetneq X} \mu(Y)$.  For details on all this, 
we refer to \cite{OT}.

Recall that a finite CW-complex is formal if its rational 
cohomology ring is quasi-isomorphic to its Sullivan's 
algebra of polynomial differential forms (more details 
can be found in Appendix \ref{sect:formal}).
It follows from Brieskorn's work that the complement 
of a hyperplane arrangement $\A$ is formal, in a very 
strong sense.  Indeed, for each hyperplane $H\in \A$, the $1$-form 
$\omega_H= \frac{1}{2\pi \ii} d \log f_H$ on $\C^{d+1}$ restricts 
to a $1$-form on $M(\A)$.  Let $\mathcal{D}$ be the subalgebra 
of the de~Rham algebra $\Omega^*_{\rm dR}(M(\A))$ generated over 
$\R$ by these $1$-forms.  Then, as shown in \cite{Br}, 
the correspondence $\omega_H \mapsto [\omega_H]$ 
induces an isomorphism $\mathcal{D} \to H^*(M(\A),\R)$, 
and this readily implies the formality of $M(\A)$.   

Similar considerations apply to the homology and cohomology 
of the projectivized complement, $U(\A)$.  In particular, if we let 
$n=\abs{\A}$ be the number of hyperplanes, then 
$H_1(M(\A),\Z)=\Z^n$, with canonical basis the set 
$\set{x_H : H\in \A}$, and 
\begin{equation}
\label{eq:h1u}
H_1(U(\A),\Z)=\Z^n \slash \big(\sum_{H\in\A} x_H\big)\cong \Z^{n-1}.
\end{equation}
We will denote by $\overline{x}_H$ the image of 
$x_H$ in $H_1(U(\A),\Z)$.

\subsection{Resonance varieties}
\label{subsec:res var}

Let $\k$ be an algebraically closed field.   The above-mentioned 
isomorphisms allow us to identify  $H^1(M(\A),\k)$ with the affine 
space $\k^n$, and $H^1(U(\A),\k)$ with the affine space
\begin{equation}
\label{eq:affine}
\mathbb{A}_\k(\A)=\{ x \in \k^n \mid x_1 + \cdots + x_n=0\}\cong \k^{n-1}.
\end{equation}

Let $\RR^q_s(M(\A),\k)$ be the resonance varieties of the 
arrangement complement.  From the general theory reviewed 
in \S\ref{subsec:rv}, we know that each of these sets  
is a homogeneous subvariety of $\k^n$, and that 
$\RR^q_s(M(\A),\k)\subseteq \RR^q_1(M(\A),\k)$, for all $s\ge 1$.  
The diffeomorphism \eqref{eq:mu}, together with the 
product formula \eqref{eq:resprod} yields an identification 
\begin{equation}
\label{eq:rqma}
\RR^q_1(M(\A),\k)\cong \RR^q_1(U(\A),\k)\cup 
\RR^{q-1}_1(U(\A),\k).
\end{equation}  
Thus, we may view the resonance varieties of $\A$ 
as lying in the affine space $\mathbb{A}_\k(\A)$.  

If $\B\subset \A$ is a proper sub-arrangement, the inclusion 
$M(\A) \inj M(\B)$ induces an epimorphism on fundamental 
groups.  By Proposition \ref{prop:epi cv}, the induced 
homomorphism in cohomology restricts to an embedding 
$\RR^1_s(\B,\k) \inj \RR^1_s(\A,\k)$.  The irreducible 
components of $\RR^1_s(\A,\k)$ that  lie in the image 
of such an embedding are called {\em non-essential}; 
the remaining components are called {\em essential}. 

The description of the Orlik--Solomon algebra given in 
\S\ref{subsec:os alg} makes it clear that the resonance 
varieties $\RR^q_s(M(\A),\k)$ depend only on the 
intersection lattice, $L(\A)$, and on the characteristic 
of the field $\k$.  A basic problem in the subject is to 
find concrete formulas making this dependence explicit. 
We will briefly discuss the positive characteristic case 
in \S\ref{subsec:res p}, but for now we will concentrate 
on the case when $\ch(\k)=0$ and $q=1$.

The complex resonance varieties $\RR^q_1(M(\A),\C)$ 
were first defined and studied by Falk in \cite{Fa97}.  Soon 
after, Cohen--Suciu \cite{CS99}, Libgober \cite{Li01}, and 
Libgober--Yuzvinsky \cite{LY00} showed that the varieties 
$\RR_s(\A)=\RR^1_s(M(\A),\C)$ consist of linear subspaces 
of the vector space $\mathbb{A}(\A)=\mathbb{A}_{\C}(\A)$, 
and analyzed the nature of these subspaces. Let 
us summarize those results.   

\begin{theorem}
\label{thm:res arr}
For a hyperplane arrangement $\A$, the following hold. 
\begin{enumerate}
\item \label{r1}
 Each irreducible component of $\RR_1(\A)$ 
is either $\{0\}$, or a linear subspace of $\mathbb{A}(\A)$ 
of dimension at least $2$. 
\item \label{r2} 
Two distinct components of $\RR_1(\A)$ meet 
only at $0$.
\item \label{r3}
$\RR_s(\A)$ is either $\{0\}$, or the union of all components 
of $\RR_1(\A)$ of dimension greater than $s$.
\end{enumerate}
\end{theorem}

\subsection{Multinets}
\label{subsec:multinets}
An elegant method, due to Falk and Yuzvinsky \cite{FY07}, 
describes explicitly the linear subspaces comprising $\RR_1(\A)$.   
Before proceeding, we need to introduce the relevant 
combinatorial notion, following the treatment from \cite{FY07}.  

\begin{definition}
\label{def:multinet}
A {\em multinet}\/ $\mathcal{M}$ on an arrangement $\A$ consists 
of the following data:
\begin{romenum}
\item An integer $k\ge 3$, and a partition of $\A$ 
into $k$ subsets, say, $\A_1,\ldots,\A_k$.
\item An assignment of multiplicities on the hyperplanes, 
$m\colon \A\to \N$.
\item A subset $\XX\subseteq L_2(\A)$, called the base locus.
\end{romenum}
Moreover, the following conditions must be satisfied:
\begin{enumerate}
\item  \label{m1} 
There is an integer $\ell$ such that $\sum_{H\in\A_i} m_H=\ell$, 
for all $i\in [k]$.
\item  \label{m2} 
For any two hyperplanes $H$ and $H'$ in different classes, the flat 
$H\cap H'$ belongs to $\XX$.
\item  \label{m3} 
For each $X\in\XX$, the sum 
$n_X:=\sum_{H\in\A_i\colon H\supset X} m_H$ is independent of $i$.
\item  \label{m4} 
For each $i\in [k]$, the space 
$\big(\bigcup_{H\in \A_i} H\big) \setminus \XX$ is connected. 
\end{enumerate}
\end{definition}

We say that a multinet  as above has $k$ classes and weight $\ell$, 
and refer to it as a $(k,\ell)$-multinet.  Without essential loss of generality, 
we may assume that $\gcd\{m_H \mid H\in \A\}=1$.  If all the multiplicities 
are equal to $1$, the multinet is said to be {\em reduced}. 
If, furthermore, every flat in $\XX$  is contained in precisely one 
hyperplane from each class, the multinet is called 
a {\em $(k,\ell)$-net}. 

For any multinet on $\A$, the base locus $\XX$ 
is determined by the partition $(\A_1,\dots, \A_k)$.  Indeed, 
for each $i\ne j$, we have that 
\begin{equation}
\label{eq:base}
\XX = \{ H \cap H' \mid H\in \A_i,\, H'\in \A_j \}.
\end{equation}
The next lemma (also from \cite{FY07}) is an easy consequence 
of this observation, and Definition \ref{def:multinet}.  
For completeness, we include a proof. 

\begin{lemma}[\cite{FY07}] 
\label{lem:fysums}
For any $(k,\ell)$-multinet, the following identities hold:
\begin{enumerate}
\item $\sum_{H\in \A} m_H=k\ell$. 
\item $\sum_{X\in \XX:  H\supset X} n_X= \ell$, for any $H\in \A$. 
\item $\sum_{X\in \XX} n_X^2 = \ell^2$.
\end{enumerate}
\end{lemma}

\begin{proof}
Clearly, 
\begin{equation}
\label{eq:fysum1}
\sum_{H\in \A} m_H = \sum_{i=1}^{k} \sum_{H\in \A_i} m_H=k\ell.
\end{equation}

To verify the second identity, let $i\in [k]$ be such that $H\in \A_i$; then, 
for any $j\ne i$, we have 
\begin{equation}
\label{eq:fysum2}
\ell=\sum_{H'\in \A_j} m_{H'} = \sum _{X\in \XX:  H\supset X}  
\sum_{H'\in \A_j: H'\supset X} m_{H'} = \sum _{X\in \XX:  H\supset X}  n_X.
\end{equation}

Finally, for any $i\ne j$, 
\begin{align}
\label{eq:fysum3}
\ell^2&=\sum_{H\in \A_i} m_{H} \sum_{H'\in \A_j} m_{H'} = 
\sum_{H\in \A_i, H'\in \A_j} m_H m_{H'} \\
&=\sum_{H\in \A_i, H'\in \A_j}  \sum_{X\in \XX: H\supset X, H'\supset X} m_H m_{H'} 
= \sum_{H\in \A_i} \sum_{X\in \XX: H\supset X} m_H n_X =   \sum_{X\in \XX} n_X^2,\notag
\end{align}
and this completes the proof.
\end{proof}

\begin{example}
\label{ex:local component}
Let $X\in L_2(\A)$, and assume that the 
sub-arrangement $\A_{X}=\{H \in \A\mid H \supset X\}$ has 
size at least $3$. We may then form a net on $\A_{X}$ by 
assigning to each hyperplane the multiplicity $1$, putting 
one hyperplane in each class, and setting $\XX=\{X\}$.
\end{example}

\begin{figure}
\centering
\begin{tikzpicture}[scale=0.78]
\draw[style=thick,densely dashed,color=blue] (-0.5,3) -- (2.5,-3);
\draw[style=thick,densely dotted,color=red]  (0.5,3) -- (-2.5,-3);
\draw[style=thick,color=dkgreen] (-3,-2) -- (3,-2);
\draw[style=thick,densely dotted,color=red]  (3,-2.68) -- (-2,0.68);
\draw[style=thick,densely dashed,color=blue] (-3,-2.68) -- (2,0.68);
\draw[style=thick,color=dkgreen] (0,-3.1) -- (0,3.1);
\node at (-2,-2) {$\bullet$};
\node at (2,-2) {$\bullet$};
\node at (0,2) {$\bullet$};
\node at (0,-0.7) {$\bullet$};
\node at (-2.5,0.4) {$x+z$};
\node at (2.5,0.4) {$x-y$};
\node at (0,-3.5) {$y+z$};
\node at (-3.7,-2) {$y-z$};
\node at (-2.5,-3.3) {$x-z$};
\node at (2.5,-3.3) {$x+y$};
\end{tikzpicture}
\caption{A $(3,2)$-net on the ${\rm A}_3$ arrangement}
\label{fig:braid}
\end{figure}
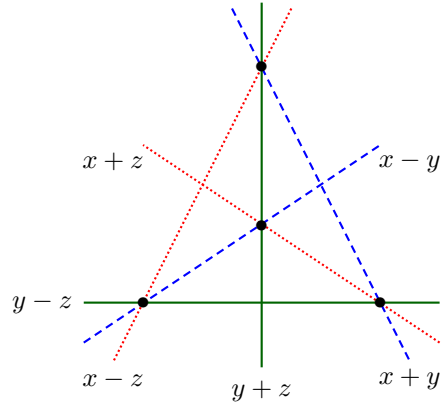

\begin{example}
\label{ex:braid arr}
Let $\A$ be a generic $3$-section of the reflection arrangement 
of type ${\rm A}_3$, defined by the polynomial 
$Q(\A)=(x+y)(x-y)(x+z)(x-z)(y+z)(y-z)$.   
Figure \ref{fig:braid} shows a plane section of $\A$, with the 
hyperplanes labeled by their defining linear forms. 
Ordering the hyperplanes as the factors of $Q(\A)$, 
the flats in $L_2(\A)$ may be 
labeled as $136$, $145$, $235$, and $246$.  
The $(3,2)$-net on $\A$ depicted in the picture 
corresponds to the partition $(12|34|56)$; 
the base locus $\XX$ consists of the triple 
points indicated by dark circles. 
\end{example}

\begin{example}
\label{ex:B3}
Let $\A$ be the reflection arrangement of type ${\rm B}_3$, 
defined by the polynomial $Q(\A)=xyz(x^2-y^2)(x^2-z^2)(y^2-z^2)$.   
Figure \ref{fig:b3 arr} shows a plane section of $\A$, together with a 
$(3,4)$-multinet on it.
\end{example}

\begin{figure}
\centering
\begin{tikzpicture}[scale=0.9]
\draw[style=thick,color=dkgreen] (0,0) circle (3.1);
\node at (-2.4,0.25) {$y^2$};
\node at (0,-2.5) {$x^2$};
\node at (3.35,0.5) {$z^2$};
\clip (0,0) circle (2.9);
\draw[style=thick,densely dashed,color=blue] (-1,-2.1) -- (-1,2.5);
\draw[style=thick,densely dotted,color=red] (0,-2.2) -- (0,2.5);
\draw[style=thick,densely dashed,color=blue] (1,-2.1) -- (1,2.5);
\draw[style=thick,densely dotted,color=red] (-2.5,-1) -- (2.5,-1);
\draw[style=thick,densely dashed,color=blue] (-2.5,0) -- (2.5,0);
\draw[style=thick,densely dotted,color=red] (-2.5,1) -- (2.5,1);
\draw[style=thick,color=dkgreen]  (-2,-2) -- (2,2);
\draw[style=thick,color=dkgreen](-2,2) -- (2,-2);
\node at (-2.1,-0.75) {$y-z$};
\node at (-2.1,1.2) {$y+z$};
\node at (-1,-2.35) {$x-z$};
\node at (1,-2.35) {$x+z$};
\node at (-2,-1.4) {$x-y$};
\node at (2,-1.4) {$x+y$};
\end{tikzpicture}
\caption{A $(3,4)$-multinet on the ${\rm B}_3$ arrangement}
\label{fig:b3 arr}
\end{figure}
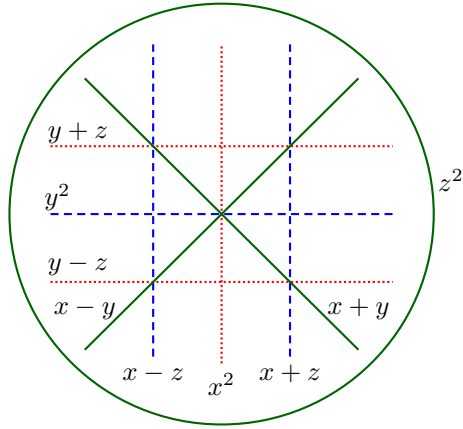
 
The next theorem, which combines results of Pereira--Yuzvinsky 
\cite{PY} and Yuzvinsky \cite{Yu09}, summarizes what is known 
about the existence of non-trivial multinets on arrangements. 

\begin{theorem} 
\label{thm:pyy}
Let $(\A_1,\dots,\A_k)$ be a multinet on $\A$, with multiplicity 
vector $m$ and base locus $\XX$. 
\begin{enumerate}
\item \label{multi1}
If $\abs{\XX}>1$, then $k=3$ or $4$.    
\item \label{multi2}
If there is a hyperplane $H\in \A$ such that $m_H>1$, 
then $k=3$. 
\end{enumerate}
\end{theorem}

Although infinite families of multinets with $k=3$ are known, 
only one multinet with $k=4$ is known to exist: the $(4,3)$-net 
on the Hessian arrangement, which we will discuss in 
Example \ref{ex:hesse}.  In fact, it is a conjecture of Yuzvinsky 
\cite{Yu12} that the only $(4,\ell)$-multinet is the Hessian 
$(4,3)$-net.

\subsection{Resonance and multinets}
\label{subsec:res net}

We now return to the (degree $1$) complex resonance 
varieties of a hyperplane arrangement $\A$. Recall that 
$\RR_1(\A)=\RR^1_1(M(\A),\C)$ is a subvariety of the 
affine space $\mathbb{A}(\A)=\{ x\in \C^{\abs{\A}} 
\mid \sum_{H\in \A} x_H=0\}$, consisting of linear 
subspaces meeting only at the origin.  Let us describe 
these subspaces in a concrete way.

Given a multinet $\cM$ on $\A$, with parts 
$(\A_1,\dots ,\A_k)$ and multiplicity vector $m$, 
set $u_i=\sum_{H\in \A_i} m_H e_H$ for each $1\le i\le k$, 
and  put 
\begin{equation}
\label{eq:pm}
P_\cM=\spn \set{u_2-u_1,\dots , u_k-u_1}.  
\end{equation}
By construction, $P_\cM$ is a linear subspace of $\mathbb{A}(\A)$.  
As shown in \cite[Theorem 2.4]{FY07}, this subspace 
lies inside $\RR_1(\A)$ and has dimension $k-1$. 

Now suppose there is a sub-arrangement $\B\subset \A$ 
which supports a multinet $\cM$ with $k$ parts.   By the above, 
the linear space $P_{\cM}$ lies inside $\RR_1(\B)$.
By the discussion from \S\ref{subsec:res var}, 
the inclusion $M(\A)\inj M(\B)$ induces an embedding 
$\RR_1(\B) \inj \RR_1(\A)$.  Thus, $P_{\cM}$ is a 
linear subspace of $\RR_1(\A)$, of dimension $k-1$.
Conversely, it is shown in \cite[Theorem 2.5]{FY07} 
that all (non-zero) irreducible components of $\RR_1(\A)$ 
arise in this fashion.  

Summarizing the above discussion, we have the following 
description of the first resonance variety of an arrangement.

\begin{theorem}[\cite{FY07}]
\label{thm:res fy}
The positive-dimensional, irreducible components of 
$\RR_1(\A)$ are in one-to-one correspondence 
with the multinets on sub-arrangements of $\A$, and so 
\[
\RR_1(\A) = \set{0}\cup \bigcup_{\B \subset \A} 
\bigcup_{\text{$\cM$ a multinet on $\B$}} P_{\cM}.
\]
\end{theorem}
Using now Theorem \ref{thm:res arr}, part \ref{r3}, we 
find that 
\begin{equation}
\label{eq:rsa}
\RR_s(\A) = \set{0}\cup \bigcup_{\B \subset \A} 
\bigcup_{\stackrel{\text{$\cM$ a multinet on $\B$}}
{\text{with at least $s+2$ parts}}} P_{\cM}.
\end{equation}

\subsection{Local and non-local components}
\label{subsec:r1a}

The simplest components of $\RR_1(\A)$ arise in 
the following fashion.  Let $X$ be a rank $2$ flat lying at the 
intersection of at least $3$ hyperplanes.   Recall from 
Example \ref{ex:local component} that $X$ determines 
a net on the sub-arrangement $\A_X$ consisting of those 
hyperplanes in $\A$ that contain $X$. 
The corresponding component of $\RR_1(\A)$,
\begin{equation}
\label{eq:loc comp}
P_X=\Big\{ x\in \mathbb{A}(\A) \,\big|\,  
 \text{$\sum_{H \supset X} x_H = 0$ and $x_H=0$ 
if $H \not\supset X$} \big. \Big\},
\end{equation}
has dimension $\abs{\A_X}-1$, and is called a {\em local}\/ 
component.

If $\abs{\A}\le 5$, then all components of $\RR_1(\A)$ are local.  
For $\abs{\A}\ge 6$, though, the resonance variety $\RR_1(\A)$ 
may have non-local components. It follows from 
Theorems \ref{thm:pyy} and \ref{thm:res fy} that 
any such component must have dimension either 
$2$ or $3$. 

\begin{example}
\label{ex:braid}  
Let $\A$ be the braid arrangement from Example \ref{ex:braid arr}. 
The variety $\RR_{1}(\A)\subset \C^6$ has $4$ local components, 
corresponding to the flats $136, 145, 235, 246$, and one non-local 
component, corresponding to the net $\cM$ indicated in Figure \ref{fig:braid}:
\[
\small{
\begin{aligned}
& P_{136}= \{ x_1 + x_3 + x_6=x_2=x_4=x_5=0 \} ,\ 
P_{145}=\{ x_1 + x_4 + x_5=x_2=x_3=x_6=0 \} ,  \\
& P_{235}= \{ x_2 + x_3 + x_5= x_1=x_4=x_6=0 \}, \
 P_{246}= \{ x_2 + x_4 + x_6=x_1=x_3=x_5=0 \} ,\\
& P_{\cM}=\{ x_1+x_3+x_6=x_1- x_2=x_3-x_4=x_5-x_6=0 \} .
\end{aligned}
}
\]
Since all these components are $2$-dimensional, $\RR_2(\A)=\{0\}$. 
\end{example}

\begin{example}
\label{ex:B3 bis}
Let $\A$ be the ${\rm B}_3$-arrangement from Example \ref{ex:B3}.  
Ordering the hyperplanes as the factors of the defining polynomial, 
the multinet $\cM$ indicated in Figure \ref{fig:b3 arr} has associated  
partition $(167|289|345)$. 

The variety $\RR_1(\A)\subset \C^9$
has $7$ local components, corresponding to $4$ triple points and $3$ 
quadruple points, $11$ components corresponding to braid sub-arrangements, 
and one essential, $2$-dimensional component, corresponding to the 
above multinet, 
\[
P_{\cM}=\set{x_1=x_6=x_7, \, 
x_2=x_8=x_9, \,
x_3=x_4=x_5,\, 
x_1+x_2+x_3=0}.
\]
\end{example}

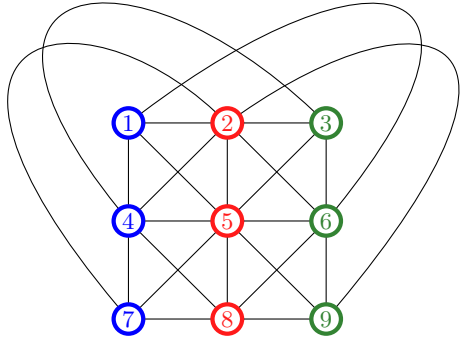
\begin{figure}
\centering
\begin{tikzpicture}[scale=0.65]
\path[use as bounding box] (-2,-0.5) rectangle (6,6.5); 
\path (0,4) node[draw, ultra thick, shape=circle, 
inner sep=1.3pt, outer sep=0.7pt,  color=blue] (v1) {\small{1}};
\path (2,4) node[draw, ultra thick, shape=circle, 
inner sep=1.3pt, outer sep=0.7pt, color=red!90!white] (v2) {\small{2}};
\path (4,4) node[draw, ultra thick, shape=circle, 
inner sep=1.3pt, outer sep=0.7pt,  color=dkgreen!80!white] (v3) {\small{3}};
\path (0,2) node[draw, ultra thick, shape=circle, 
inner sep=1.3pt, outer sep=0.7pt,  color=blue] (v4) {\small{4}};
\path (2,2) node[draw, ultra thick, shape=circle, 
inner sep=1.3pt, outer sep=0.7pt,  color=red!90!white] (v5) {\small{5}};
\path (4,2) node[draw, ultra thick, shape=circle, 
inner sep=1.3pt, outer sep=0.7pt, color=dkgreen!80!white] (v6) {\small{6}};
\path (0,0) node[draw, ultra thick, shape=circle, 
inner sep=1.3pt, outer sep=0.7pt, color=blue] (v7) {\small{7}};
\path (2,0) node[draw, ultra thick, shape=circle, 
inner sep=1.3pt, outer sep=0.7pt, color=red!90!white] (v8) {\small{8}};
\path (4,0) node[draw, ultra thick, shape=circle, 
inner sep=1.3pt, outer sep=0.7pt,  color=dkgreen!80!white] (v9) {\small{9}};
\draw (v1) -- (v2) -- (v3); \draw (v4) -- (v5) -- (v6); \draw (v7) -- (v8) -- (v9);
\draw (v1) -- (v4) -- (v7);  \draw (v2) -- (v5) -- (v8); \draw (v3) -- (v6) -- (v9);
\draw (v2) -- (v4) -- (v8) -- (v6) -- (v2);
\draw (v1) -- (v5) -- (v9); \draw (v3) -- (v5) -- (v7); 
\draw (v7) .. controls (-5,6) and (-1,6.8) .. (v2);
\draw (v4) .. controls (-4,7) and (0,7.8) .. (v3);
\draw (v1) .. controls (5,7.8) and (8,7) .. (v6);
\draw (v2) .. controls (6,6.8) and (9,6) .. (v9);
\end{tikzpicture}
\caption{The Ceva($3$) matroid,  
together with a $(3,3)$-net}
\label{fig:ceva3}
\end{figure}

\begin{example}
\label{ex:ceva3}
Let $\A$ be the Ceva($3$) arrangement, also known as the monomial arrangement 
of type $\A(3,3,3)$, associated to the complex reflection group $G(3,3,3)$, and 
defined by the polynomial $Q(\A)=(x^3-y^3)(y^3-z^3)(x^3-z^3)$. 
Figure \ref{fig:ceva3} shows the corresponding matroid, 
denoted ${\rm AG}(2,3)$ in \cite{Ox}.  

The resonance variety $\RR_1(\A)\subset \C^9$  has  
$12$ local components, corresponding to the triple points, 
and $4$ essential components corresponding to the 
$(3,3)$-nets defined by the partitions 
$(123 | 456 | 789)$, $(147 | 258 | 369)$, $(159 | 267 | 348)$, 
and $(168 | 249 | 357)$.  The first of these nets is depicted in 
Figure \ref{fig:ceva3}.  
\end{example}

The next example describes the only arrangement 
for which non-local components of dimension $3$ 
are known to occur.

\begin{example}
\label{ex:hesse}
Let $\A$ be the Hessian arrangement in $\C^3$, 
also known as the monomial arrangement of type 
$\A(3,1,3)$ associated to the complex reflexion group 
of type $G(3,1,3)$, and 
defined by the polynomial 
$Q(\A)=xyx(x^3-y^3)(x^3-z^3)(y^3-z^3)$.    
The projective configuration consists of $12$ projective 
lines meeting at $9$ quadruple points, and $12$ double points; 
each line contains $3$ quadruple points and $2$ double points.  
Figure~\ref{fig:hessian} shows the corresponding matroid, 
which is obtained from the projective plane over $\Z_3$ 
(also known as the ${\rm PG}(2,3)$ matroid, cf.~\cite{Ox}) 
by deleting a point.  The resonance 
variety $\RR_1(\A)$ has $9$ local components, 
and a single, $3$-dimensional, 
essential component, corresponding to the $(4,3)$-net  
depicted in Figure~\ref{fig:hessian}.
\end{example}

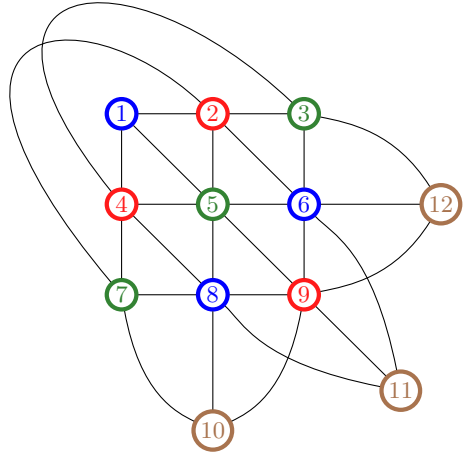
\begin{figure}
\centering
\begin{tikzpicture}[scale=0.6]
\path[use as bounding box] (-1.5,-3.5) rectangle (6,6.5); 
\path (0,4) node[draw, ultra thick, shape=circle, 
inner sep=1.3pt, outer sep=0.7pt,  color=blue] (v1) {\small{1}};
\path (2,4) node[draw, ultra thick, shape=circle, 
inner sep=1.3pt, outer sep=0.7pt, color=red!90!white] (v2) {\small{2}};
\path (4,4) node[draw, ultra thick, shape=circle, 
inner sep=1.3pt, outer sep=0.7pt,  color=dkgreen!80!white] (v3) {\small{3}};
\path (0,2) node[draw, ultra thick, shape=circle, 
inner sep=1.3pt, outer sep=0.7pt,  color=red!90!white] (v4) {\small{4}};
\path (2,2) node[draw, ultra thick, shape=circle, 
inner sep=1.3pt, outer sep=0.7pt,  color=dkgreen!80!white] (v5) {\small{5}};
\path (4,2) node[draw, ultra thick, shape=circle, 
inner sep=1.3pt, outer sep=0.7pt, color=blue] (v6) {\small{6}};
\path (0,0) node[draw, ultra thick, shape=circle, 
inner sep=1.3pt, outer sep=0.7pt, color=dkgreen!80!white] (v7) {\small{7}};
\path (2,0) node[draw, ultra thick, shape=circle, 
inner sep=1.3pt, outer sep=0.7pt, color=blue] (v8) {\small{8}};
\path (4,0) node[draw, ultra thick, shape=circle, 
inner sep=1.3pt, outer sep=0.7pt,  color=red!90!white] (v9) {\small{9}};
\draw (v1) -- (v2) -- (v3); \draw (v4) -- (v5) -- (v6); \draw (v7) -- (v8) -- (v9);
\draw (v1) -- (v4) -- (v7);  \draw (v2) -- (v5) -- (v8); \draw (v3) -- (v6) -- (v9);
\draw (v4) -- (v8); \draw (v2) -- (v6);
\draw (v1) -- (v5) -- (v9); 
\draw (v7) .. controls (-5,6) and (-1,6.8) .. (v2);
\draw (v4) .. controls (-4,7) and (0,7.8) .. (v3);
\path (2,-3) node[draw, ultra thick, shape=circle, 
inner sep=1.3pt, outer sep=0.7pt,  color=dkbrown!75!white] (v10) {\small{10}};
\path (6.12,-2.12) node[draw, ultra thick, shape=circle, 
inner sep=1.3pt, outer sep=0.7pt,  color=dkbrown!75!white] (v11) {\small{11}};
\path (7,2) node[draw, ultra thick, shape=circle, 
inner sep=1.3pt, outer sep=0.7pt,  color=dkbrown!75!white] (v12) {\small{12}};
\draw (v8) .. controls (2.65,-0.6) and (3,-1.5) .. (v11);
\draw (v6) .. controls (4.6,1.35) and (5.5,1) .. (v11);
\draw (v7) .. controls (0.4,-2.2) and (1.2,-2.6) .. (v10);
\draw (v9) .. controls (3.6,-2.2) and (2.8,-2.6) .. (v10);
\draw (v8) -- (v10);  \draw (v9) -- (v11);  \draw (v6) -- (v12); 
\draw (v3)  .. controls (5,3.8) and (6,3.6) .. (v12); 
\draw (v9) .. controls (5,0.2) and (6,0.4) ..  (v12); 
\end{tikzpicture}
\caption{The Hessian matroid, together with 
a $(4,3)$-net}
\label{fig:hessian}
\end{figure}

\subsection{Resonance in positive characteristic}
\label{subsec:res p}
 
The varieties $\RR^1_s(M(\A),\k)$ over fields $\k$ of 
characteristic $p>0$ were first defined and studied by Matei 
and Suciu in \cite{MS00}, and further investigated by Falk in \cite{Fa07}. 
As illustrated in the next few examples, all three    
properties from Theorem \ref{thm:res arr}, which hold 
over $\k=\C$, fail in this setting. 

\begin{example}
\label{ex:nonfano}
Let $\A$ be the realization of the non-Fano plane, 
obtained from the $\operatorname{B}_3$ arrangement 
from Figure \ref{fig:b3 arr} by deleting the planes $x=0$ 
and $y=0$, and order the hyperplanes as the factors of 
$Q(\A)=z(x+y)(x-y)(x+z)(x-z)(y+z)(y-z)$.  The variety 
$\RR_{1}(\A)$ has $6$ local components corresponding 
to triple points, and $3$ components arising from  braid 
sub-arrangements, but no essential components.

Now let $\k=\overline{\F_2}$. Then, as noted in \cite{MS00}, the 
variety $\RR^1_1(M(\A),\k) \subset \mathbb{A}_\k(\A) $ has a   
single non-local component, defined by the equations  
$x_1+x_4+x_5=x_1+x_6+x_7=x_2+x_5+x_6=x_3+x_5+x_7=0$, 
while $\RR^1_2(M(\A),\k)$ has a single, $1$-dimensional component, 
defined by the equations $x_4=x_5=x_6=x_7$ and $x_1=x_2=x_3=0$. 
Thus, property \ref{r3} fails.

It should be noted that the analogue of the tangent cone formula 
\eqref{eq:tc} also fails in this case.  Indeed, a computation from 
\cite{S01} shows that $\VV^1_2(M(\A),\k) =\{1\}$; thus, 
the tangent cone to this variety is $\{0\}$, which is 
is strictly included in $\RR^1_2(M(\A),\k)$.
\end{example} 

\begin{example}
\label{ex:db3}
Let $\A$ be the arrangement obtained by deleting 
the plane $z=0$ from the $\operatorname{B}_3$ arrangement, 
and let $\k=\overline{\F_2}$.  Then, as shown in  \cite{Fa07}, 
the variety $\RR^1_1(M(\A),\k)$ has two distinct components 
which intersect outside $0$.  Thus, property \ref{r2} fails.
\end{example}

\begin{example}
\label{ex:hesse again}
Let $\A$ be the Hessian arrangement from Example \ref{ex:hesse}, 
and let $\k=\overline{\F_3}$.  Then, as shown in  \cite{Fa07}, 
the variety $\RR^1_1(M(\A),\k)$ has a component which is 
isomorphic to an irreducible cubic hypersurface in $\k^5$.  
Thus, property \ref{r1} fails.
\end{example}

\begin{remark}
\label{rem:p-massey}
In \cite{Ma07}, Matei uses the resonance varieties $\RR^1_1(M(\A),\k)$ 
over a field $\k$ of characteristic an odd prime $p$
to detect non-vanishing Massey products in the cohomology 
ring $H^*(M(\A), \k)$ of certain arrangements $\A$.  
The simplest such example is the Hessian arrangement 
from Example \ref{ex:hesse}, for which $p=3$.
\end{remark}

These examples and the above remark raise a natural question.

\begin{problem}
\label{quest:p-formal}
Given an arrangement $\A$, find the primes $p$ (if any) for which 
$M(\A)$ is not $\k$-formal, over a field $\k$ of characteristic $p$.
\end{problem}

\section{The complement of an arrangement. II.}
\label{sect:ma2}

In this section we describe the fundamental group and the 
characteristic varieties of the complement of a complex 
hyperplane arrangement.

\subsection{Fundamental group}
\label{subsec:pi1arr}
As usual, let $\A$ be a hyperplane arrangement in $\C^{d+1}$. 
The complement $M(\A)$ is a path-connected space.  Thus, we 
may fix a basepoint $x_0\in M(\A)$, and consider the fundamental 
group $\pi_1(M(\A),x_0)$. 

For each hyperplane $H\in \A$, pick a meridian curve about $H$, 
oriented compatibly with the complex orientations on $\C^{d+1}$ and $H$, 
and let $x_H$ denote the based homotopy class of this curve, 
joined to the basepoint by a path in $M(\A)$.  By the van Kampen theorem, 
then, the group $\pi_1(M(\A),x_0)$ is generated by the set $\set{x_H : H\in \A}$.
An explicit presentation for this group can be obtained 
via the braid monodromy algorithm from \cite{CS97}. Let us 
briefly describe this algorithm.

Let $\A'$ be a generic two-dimensional section of $\A$.   
By the Lefschetz-type theorem of Hamm and L\^{e}, 
the inclusion $M(\A')\to M(\A)$ between the respective 
complements induces an isomorphism on fundamental 
groups.   Thus, for the purpose of computing fundamental 
groups, we may as well replace $\A$ by $\A'$.

So let $\A =\{\ell_{1},\dots ,\ell_{n}\}$ be an arrangement of  affine lines
in $\C^2$.  Let $v_1, \dots ,v_s$ be the intersection points of the lines. 
If $v = \ell_{i_1}\cap \dots \cap \ell_{i_r}$ is one such intersection point, write 
$I=(i_1,\dots,i_{r  })$ for the corresponding increasingly ordered tuple, 
and let $A_I$ be the braid in the pure braid group $P_n$ which 
performs a full twist on the strands corresponding to $I$, leaving 
the other strands fixed.   There are then braids $\delta_1,\dots, \delta_s$ 
in the full braid group $B_n$ such that the arrangement group 
has presentation
\begin{equation}
\label{eq:bmpres}
\pi_1(M(\A))=\langle x_1,\dots ,x_n \mid  \delta_q^{-1} A_{I_{s}} 
\delta_{q}(x_i)=x_i
\ \text{for } i \in I_q\setminus \max(I_q)  \text{ and } q\in [s]
\rangle.  
\end{equation}
 
Clearly, this is a commutator-relators presentation.  Furthermore, 
the presentation is minimal, in that the number of generators equals 
$b_1(M(\A))$, and the number of relators equals $b_2(M(\A))$.

If $\A$ is the complexification of a real arrangement, 
the conjugating braids $\delta_{q}$ may be obtained by a procedure 
that goes back to \cite{Hi93}.  Each vertex set $I_{q}$  gives rise to a partition
$[n] = I'_{q}\cup I_{q} \cup I''_{q}$ into lower, middle, and upper indices.  
Let $J_{q}=\{i \in I''_{q} \mid \min I_{q}<i<\max I_{q}\}$.  
Then $\delta_{q}$ is the subword of the full twist $A_{[n]}$ 
given by 
\begin{equation}
\label{eq:conjbraid}
\delta_{q} = \prod_{i\in I_{q}} \prod_{j\in J_{q}} A_{ji}.
\end{equation}

In the general case, the braids $\delta_q$ can be 
read off a ``braided wiring diagram"  associated to the 
arrangement, see \cite{CS97} for further details and references.

\begin{remark}
\label{rem:ryb}
In work from the early 1990s that appeared in \cite{Ryb}, 
Rybnikov constructed a pair of arrangements, $\A^+$ and $\A^-$, 
both realizations of the same matroid, for which 
$\pi_1(M(\A^+))\not\cong \pi_1(M(\A^-))$.
Thus, the fundamental group of an arrangement 
complement is {\em not}\/ 
determined by the intersection lattice, in general.
\end{remark}

\begin{remark}
\label{rem:lcs}
Nevertheless, if $G=\pi_1(M(\A))$ is an arrangement 
group, the ranks of the associated graded Lie algebra, 
$\phi_k=\rank ( \gr_k(G) )$, are combinatorially determined, 
due to the formality of $M(\A)$.  Moreover, as shown 
in \cite{PS-imrn}, the ranks of the Chen Lie algebra, 
$\theta_k=\rank ( \gr_k(G/G'') )$, are also combinatorially 
determined. As observed in \cite{S01}, there are interesting 
connections between the Chen ranks of the group $G$ 
and the dimensions of the components of the resonance 
variety $\RR_1(\A)$.  For a recent survey of these topics, 
we refer to \cite{De10}.
\end{remark}

\subsection{The universal abelian cover}
\label{subsec:abelcov}
 
Returning to the general situation, let $\A$ be an essential 
arrangement in $\C^{d+1}$.  Fix an ordering of the hyperplanes, 
say, $\A=\{H_1,\dots ,H_n\}$, and choose linear forms 
$f_i\colon \C^{d+1}\to \C$ with $\ker(f_i)=H_i$.  Assembling 
these forms together, we obtain an injective linear map
\begin{equation}
\label{eq:iota}
\iota(\A)\colon \C^{d+1}\to\C^n, \quad 
z\mapsto (f_1(z), \dots ,f_n(z)). 
\end{equation}

Now let $\B_n$ be the Boolean arrangement in $\C^n$, and identify 
$M(\B_n)$ with $(\C^*)^n$.  Recall that $z\in H_i$ if and only if 
$f_i(z)=0$; thus, the map $\iota(\A)$ restricts to an inclusion 
$\iota(\A)\colon M(\A)\inj M(\B_n)$.  Consequently, we may view  
the complement $M(\A)$ as a linear slice of the complex $n$-torus:
\begin{equation}
\label{eq:slice}
M(\A)=\iota(\A)(\C^{d+1})\cap (\C^*)^n.
\end{equation}
 
Clearly, the map $\iota(\A)\colon M(\A)\to(\C^*)^n$ 
is equivariant with respect to the diagonal action of $\C^*$ 
on both source and target.  Thus, $\iota(\A)$ descends to a map 
$\overline{\iota}(\A)\colon M(\A)/\C^*\inj (\C^*)^n/\C^*$, which 
defines an embedding $\overline{\iota}(\A)\colon U(\A) \inj(\C^*)^{n-1}$. 

\begin{lemma}[\cite{DeS12}]
\label{lem:H1ofA}
Let $\A=\{H_1,\dots ,H_n\}$ be a hyperplane arrangement.  Then,
\begin{enumerate}
\item \label{iota1}
The inclusion $\iota(\A)\colon M(\A)\inj (\C^*)^n$ is 
a classifying map for the universal abelian cover $M(\A)^{\ab}\to M(\A)$.
\item \label{iota2}
The inclusion $\overline{\iota}(\A)\colon U(\A)\inj (\C^*)^{n-1}$ 
is a classifying map for the universal abelian cover $U(\A)^{\ab}\to U(\A)$.
\end{enumerate}
\end{lemma}

\begin{proof}
We start with the first statement.  Set $\iota=\iota(\A)$. 
Since $(\C^*)^n$ is a classifying 
space for the group $\pi_1(M(\A))_{\ab}=\Z^n$, 
it is enough to to check that the induced homomorphism, 
$\iota_{\sharp}\colon \pi_1(M(\A)) \to \Z^n$,  
coincides with the abelianization map 
$\ab\colon \pi_1(M(\A)) \to \pi_1(M(\A))_{\ab}$. 
By naturality of the Hurewicz homomorphism, we 
have a commuting diagram
\begin{equation}
\label{eq:hur}
\xymatrix{
\pi_1(M(\A)) \ar^{\iota_{\sharp}}[r]  \ar@{->>}^{\ab}[d]
& \pi_1(M(\B_n))  \ar@{=}[d] \\
H_1(M(\A),\Z) \ar^{\iota_{*}}[r]& H_1(M(\B_n),\Z)\, .
}
\end{equation}

Plainly, the homomorphism $\iota_*$ 
takes each meridian $x_H$ in $H_1(M(\A),\Z)=\Z^n$ 
to the corresponding meridian in $H_1(M(\B_n),\Z)=\Z^n$.  
Hence, in these meridian bases, the map $\iota_*\colon \Z^n \to \Z^n$ 
is the identity map, and we are done.

For the second statement, recall from \eqref{eq:h1u} 
that $H_1(U(\A),\Z)=\Z^n \slash \big(\sum_{H\in\A} x_H\big)$.
An argument as above shows that 
$\overline{\iota}_{\sharp}(\A)\colon\pi_1(U(\A))\to \Z^n/(1,\dots,1)$ 
is the abelianization map, and this ends the proof.
\end{proof}

\subsection{Characteristic varieties}
\label{subsec:cv arr}
Let $G=\pi_1(M(\A))$ and $\overline{G}=\pi_1(U(\A))$ 
be the fundamental groups of the complement and 
the projectivized complement, respectively. 
Let $\k$ be an algebraically closed field.  
The isomorphism $H_1(M(\A),\Z) \cong \Z^n$ allows us to 
identify the character group $\Hom(G,\k^*)$ 
with the algebraic torus $(\k^{*})^n$. 
Similarly, the isomorphism $H_1(U(\A),\Z) \cong \Z^n/(1,\dots, 1)$ 
allows us to identify the character group $\Hom(\overline{G},\k^*)$ 
with the algebraic torus 
\begin{equation}
\label{eq:torus}
\mathbb{T}_\k(\A)=\{t \in (\k^{*})^n \mid t_1\cdots t_n=1\}\cong (\k^{*})^{n-1}.
\end{equation}

Let $\VV^q_s(M(\A),\k)$ be the characteristic varieties of the 
arrangement complement.  From the general theory reviewed 
in \S\ref{subsec:char var}, we know that each set $\VV^q_s(M(\A),\k)$ 
is a Zariski closed subset of the algebraic torus $(\k^{*})^n$. 
The diffeomorphism \eqref{eq:mu}, together with the 
product formula \eqref{eq:cvprod} yields an identification 
$\VV^q_1(M(\A),\k)\cong \VV^q_1(U(\A),\k)\cup 
\VV^{q-1}_1(U(\A),\k)$.  In particular, we 
may view the characteristic varieties of $\A$ 
as lying in the torus \eqref{eq:torus}.

\begin{example}
\label{ex:circle}
Let $\A$ be the arrangement in $\C$ consisting of the 
single point $0$. Then $M(\A)=\C^*$ is homotopy equivalent  
to a circle $S^1$. Under the identifications  
$\pi_1(S^1,1)=\Z$ and $\Z\Z=\Z[t^{\pm 1}]$,  
the equivariant chain complex of the universal cover 
$\widetilde{S^1}=\R$ takes the form 
$0 \to C_1 \xrightarrow{\partial_1} C_0 \to 0$, 
where $C_1=C_0=\Z\Z$ and $\partial_1(1)=t-1$.  

Given a character $\rho\in \Hom(\Z,\k^*)=\k^{*}$, 
we tensor this chain complex with the  
local system $\k_\rho$, and obtain a new chain 
complex, $0 \to \k \xrightarrow{\rho -1 } \k \to 0$.  
This chain complex is exact, except for $\rho=1$, 
when $H_0=H_1=\k$.   The upshot  is that 
$\VV^0_1(S^1,\k)=\VV^1_1(S^1,\k)=\{1\}$ 
and $\VV^i_s(S^1,\k)  = \emptyset$, otherwise.
\end{example}

\begin{example}
\label{ex:cv boolean}
More generally, let $\B_n$ be the Boolean arrangement 
in $\C^n$, with complement $M(\B_n)=(\C^*)^n$.
Proceeding as in the previous example, we see that 
\begin{equation}
\label{eq:cv nil}
\VV^q_s(M(\B_n),\k)=
\begin{cases}
\{1\}& \text{if $s\le \binom{n}{q}$},\\[2pt]
 \emptyset & \text{otherwise}.
\end{cases}
\end{equation} 
\end{example}

\begin{example}
\label{ex:cv free}
Let $\mathcal{P}_n$ be a pencil of $n+1\ge 2$ lines in 
$\C^2$ as in Example \ref{ex:pencil}.  
Recall that $U(\mathcal{P}_n)\cong \C\setminus \set{\text{$n$ points}}$, 
and thus $\pi_1(U(\mathcal{P}_n))\cong F_{n}$. Identifying 
$\T_{\k}(\mathcal{P}_n)$ with 
$(\k^*)^{n}$, we find that 
\begin{equation}
\label{eq:cv free}
\V^q_s(M(\mathcal{P}_n),\k)=\begin{cases}
(\k^*)^{n} &\text{if $q=1$ and $s< n$},\\
\set{1} & \text{if $q=1$ and $s=n$ or $q=0$ and $s=1$},\\
\emptyset & \text{otherwise}.
\end{cases}
\end{equation}
\end{example}

Although an explicit formula for the characteristic varieties 
$\VV^q_s(M(\A),\k)$ is far from known in general, a 
structural result is known in the case when $\k=\C$.    
A theorem of Arapura \cite{Ar}, as strengthened by 
Dimca \cite{Di07}, Libgober \cite{Li09}, 
Artal-Bartolo, Cogolludo and Matei \cite{ACM}, 
Budur and Wang \cite{BW}, and others, 
states the following.

\begin{theorem}
\label{thm:arapura}
The characteristic varieties $\VV^q_s(M(\A),\C)$ 
of an arrangement complement are finite unions of 
torsion-translates of algebraic subtori in $\mathbb{T}_{\C}(\A)$.
\end{theorem}  

\subsection{Orbifold fibrations}
\label{subsec:orbi}

We now describe in more detail the characteristic varieties 
of $M(\A)$ in degree $q=1$.  Since these 
varieties depend only on $\pi_1(M(\A))$, we may as well 
assume $\A$ is a (central) arrangement in $\C^3$. 

A key point here is the naturality property enjoyed 
by these varieties. 
Suppose $f\colon U(\A) \to U(\B)$ is a map 
between (projectivized) arrangement complements,  
and that the induced homomorphism, 
$f_{\sharp} \colon \pi_1(U(\A)) \to \pi_1(U(\B))$, 
is surjective.  Then, by Proposition \ref{prop:epi cv}, the 
corresponding monomorphism 
between character groups restricts to an embedding 
$\VV^1_s(\B,\k) \inj \VV^1_s(\A,\k)$.  
In particular, if $\B\subset \A$ is a sub-arrangement, the 
inclusion $U(\A) \inj U(\B)$ induces an epimorphism on 
fundamental groups, and thus defines an embedding 
$\VV^1_s(\B,\k) \inj \VV^1_s(\A,\k)$.  

Building on Arapura's work, Dimca \cite{Di07} 
and Artal Bartolo, Cogolludo and Matei \cite{ACM} showed 
that the varieties $\VV_s(\A) = \VV^1_s(M(\A),\C)$ are 
unions of torsion-translated subtori inside the algebraic torus 
$\mathbb{T}(\A)=\mathbb{T}_{\C}(\A)$.  
The isolated (torsion) points in these characteristic 
varieties are still poorly understood, but the  
positive-dimensional components in $\VV_1(\A)$ 
can be described in very concrete terms. 

A (genus $0$) {\em orbifold fibration}\/ is a surjective holomorphic 
map $f\colon U(\A)\to U(\mathcal{P}_r)$, for some $r\ge 1$, 
that has connected generic fiber, and extends to 
a map between the respective compactifications, 
$\overline{f}\colon \CP^2\to \CP^1$, which is 
also a surjective, holomorphic map with connected 
generic fibers. 

The map $f$ is a locally trivial bundle 
map, away from a finite (possibly empty) set of points 
$\set{q_1,\dots, q_s}$ inside $U(\mathcal{P}_r)$;  
let $\mu_1, \dots , \mu_s$  ($\mu_i\ge 2$) denote the 
multiplicities of the respective fibers.  Let $f_{\sharp} \colon 
\pi_1(U(\A))\to \pi_1(U(\mathcal{P}_r))$ be the induced 
homomorphism on fundamental groups. Since the generic 
fiber of $f$ is connected, this homomorphism is surjective. 
Moreover,  $f_{\sharp}$ lifts to a (surjective) homomorphism 
$ f_{\sharp} \colon \pi_1(U(\A))\surj \pi$, where 
\begin{equation}
\label{eq:piorb}
\pi:=\pi_1^{\orb} ( U(\mathcal{P}_r),\mu) = 
F_{r} * \Z_{\mu_1} * \cdots * \Z_{\mu_s}
\end{equation}
is the orbifold fundamental group of the base. Thus, 
$ f_{\sharp}$ induces an injective morphism, 
$\widehat{f_{\sharp}} \colon \widehat{\pi} \inj  
\widehat{\pi_1(U(\A))}$, between character groups.
Note that  $\widehat{\pi}= \widehat{\pi}^{\circ} \times \widehat{A}$  
where $\widehat{\pi}^{\circ}\cong (\C^*)^{r}$ is the identity component  
and $\widehat{A} \cong A:=\Z_{\mu_1}\oplus \cdots \oplus \Z_{\mu_s}$.

By Proposition \ref{prop:epi cv}, the above morphism restricts 
to an inclusion $\widehat{f_{\sharp}} \colon \VV_1(\pi)\inj 
\VV_1(\A)$.  A computation as in Example \ref{ex:circle} 
shows that
\begin{equation}
\label{eq:v1piorb}
\V_1(\pi)=\begin{cases}
\widehat{\pi} & 
\text{if $r>1$},
\\[2pt]
\big( \widehat{\pi}\setminus 
\widehat{\pi}^{\circ}\big) \cup  \{\bo \} 
& \text{if $r=1$ and $s\ge 1$}.
\end{cases}
\end{equation}
Note that each irreducible component of $\VV_1(\pi)$ has 
dimension $r$.

The next theorem says that {\em all}\/ positive-dimensional 
components of $\V_1(\A)$ arise by pullback along 
such orbifold fibrations, which we call {\em large}\/ in the 
first case, and {\em small}\/ in the second case. 

\begin{theorem}[\cite{Ar, Di07, ACM}]
\label{thm:v1a}
The first characteristic variety of an arrangement $\A$ is given by
\[
\V_1(\A)=\bigcup_{\text{$f$ large}} \im(\widehat{f_{\sharp}}) \cup 
\bigcup_{\text{$f$ small}} \big( \im(\widehat{f_{\sharp}}) \setminus 
\im(\widehat{f_{\sharp}})^{\circ} \big) \cup Z, 
\]
where the unions are over the equivalence classes of pencils 
$f\colon U(\A) \to U(\mathcal{P}_r)$ of the types indicated, 
and $Z$ is a finite set of torsion characters. 
\end{theorem}

\subsection{Multinets and pencils}
\label{subsec:pencils}

As shown by Falk and Yuzvinsky in \cite{FY07}, the 
irreducible components of $\VV_1(\A)$ passing through 
the origin can be described in terms of multinets on 
the intersection lattice of $\A$.   Indeed, 
let $Q(\A)=\prod_{H\in \A} f_H$ be a defining polynomial for $\A$. 
Given a multinet $\cM$ on $\A$, with parts 
$(\A_1,\dots ,\A_k)$ and multiplicity vector $m$, write 
\begin{equation}
\label{eq:factors}
g_i=\prod_{H\in\A_i}f_H^{m_H},
\end{equation}
for $1\leq i\leq k$, so that the defining polynomial for $(\A,m)$ 
factors as $Q_m(\A)=g_1\cdots g_k$.  Next, define a rational map 
$f_m\colon \C^3\to\P^1$ by 
\begin{equation}
\label{eq:FYpencil}
f_m(x)=[g_1(x):g_2(x)].
\end{equation}

By definition of multinets, the degrees of the polynomials 
$g_i$ are independent of $i$; hence, the map $f_m$ factors 
through a rational map $\CP^2 \dashrightarrow \CP^1$.   
As shown in \cite{FY07}, there is a set 
$S=\set{[a_1\co b_1], \dots , [a_k\co b_k]} \subset \CP^1$
such that each of the polynomials  \eqref{eq:factors} 
can be written as $g_i=a_ig_2-b_ig_1$, and, furthermore, 
the image of $f_m\colon U(\A)\to\CP^1$ misses $S$.   Identify 
$\CP^1\setminus S=U(\mathcal{P}_{k-1})$.  
The restriction 
\begin{equation}
\label{eq:m pencil}
f_m\colon U(\A) \to U(\mathcal{P}_{k-1}),
\end{equation}
then, is a large pencil which gives rise by pullback 
to a $(k-1)$-dimensional component of $\VV_1(\A)$ 
which passes through the origin of $\mathbb{T}(\A)$.

\begin{example}
\label{ex:cv braid}
Let $\A$ be the braid arrangement from Examples \ref{ex:braid arr} 
and \ref{ex:braid}.  The characteristic variety $\VV_1(\A)$ has $4$ 
local components of dimension $2$, corresponding to the $4$ triple points.  
Additionally, the $(3,2)$-net $(\A,m)$ depicted in Figure \ref{fig:braid}
defines a rational map, $f_m\colon \CP^2 \dashrightarrow \CP^1$, 
$(x,y,z) \mapsto (x^2-y^2,x^2-z^2)$.  This map restricts to 
a pencil $f_m\colon U(\A)\to \CP^1 \setminus \{ (1,0), (0,1), (1,1) \}$, 
which yields another $2$-dimensional component, 
$T=\{ t\in (\C^*)^6 \mid t_1 t_3 t_6=t_1 t_2^{-1}=
t_3 t_4^{-1}=t_5t_6^{-1}=1 \}$.  
\end{example}

\subsection{Translated tori}
\label{subsec:tt}

In general, the characteristic variety $\VV_1(\A)$ also has  
irreducible components not passing through the origin. 
Following \cite{DeS12}, we describe a combinatorial 
construction which, under certain assumptions, produces 
$1$-dimensional translated subtori in $\VV_1(\A)$.

Fix a hyperplane $H\in \A$, and let   
$\A'=\A\setminus \{H\}$ be the {\em deletion}\/ 
of $\A$ with respect to $H$.  
A {\em pointed multinet}\/  on $\A$ is a multinet  
$\cM=((\A_1,\dots, \A_k), m, \XX)$, together with a distinguished 
hyperplane $H\in \A$ for which $m_H>1$, and $m_H \mid n_X$ 
for each flat $X\in \XX$ such that $X\ge H$.

\begin{prop}[\cite{DeS12}]
\label{prop:del}
Suppose $\A$ admits a pointed multinet, and $\A'$ is obtained 
from $\A$ by deleting the distinguished hyperplane $H$.  
Then $\VV_1(\A')$ has a component which is a $1$-dimensional 
subtorus of $\mathbb{T}(\A')$, translated by a character of 
order $m_H$.
\end{prop}

\begin{proof}
Without loss of generality, we may assume that $H\in\A_1$.  
Consider the regular map given by \eqref{eq:FYpencil}, 
$f_m\colon M(\A)\to \CP^1$.   Since $f_H$ does not divide 
$g_2$, we may extend $f_m$ to a regular map 
$\overline{f}_m\colon M(\A')\to\CP^1$.  
By construction, $f_H \mid g_1$, and so 
$\im(\overline{f}_m) \setminus \im(f_m)=\set{[0\co 1]}$.  
Furthermore,  $\im(\overline{f}_m)$ equals 
$U(\mathcal{P}_1)=\CP^1\setminus \set{[1\co 0],[a_3\co b_3]}$. 
Passing to the projective complement, and taking 
the corestriction of $\overline{f}_m$ to its image 
yields an orbifold fibration, 
$\overline{f}_m\colon U(\A') \to U(\mathcal{P}_1)$. 

By hypothesis, the degrees of the restrictions 
of $g_1$ and $g_2$ to the hyperplane $H$
are both divisible by $m_H$.  Thus, the fiber of 
$\overline{f}_m$ over $[0\co 1]$ has multiplicity 
$m_H$, and so $\overline{f}_m$ 
is a small pencil. The desired conclusion now follows 
from Theorem \ref{thm:v1a}.
\end{proof}

\begin{example}
\label{ex:deleted B3} 
Let $\A$ be the ${\rm B}_3$ arrangement 
from Examples \ref{ex:B3} and \ref{ex:B3 bis}, and let 
$\A'$ be the arrangement obtained by deleting the hyperplane 
$z=0$, as in Example \ref{ex:db3}.  
As noted in \cite{S02}, the characteristic variety 
$\V_1(\A')\subset (\C^{*})^8$ has $7$ local components, 
corresponding to $6$ triple points and one quadruple point, 
$5$ components corresponding to braid sub-arrangements, 
and a component of the form $\rho T$, where 
$\rho=(1,1,-1,-1,-1,-1,1,1)$ and 
$T=\{(t^2,t^{-2},1,1,t^{-1},t^{-1},t,t) \mid t\in \C^{*}\}$.  

To explain where this translated torus comes from, 
let $m$ be the multinet on $\A$ depicted in Figure \ref{fig:b3 arr},  
and let $Q_m(\A)=g_1g_2g_3$ be the corresponding 
factorization of the defining polynomial of $(\A,m)$. 
We then obtain a  (large) pencil $f_m\colon U(\A)\to \CP^1 
\setminus \{ [1\co 0], [0\co 1], [1\co 1] \}$, given by
$[x,y,z]\mapsto [z^2(x^2-y^2):y^2(x^2-z^2)]$. 

Extending $f_m$ to the complement of $\A'$ 
by allowing $z=0$ yields an orbifold fibration,   
$f_m\colon U(\A')\to\P^1\setminus \{ [1\co 0], [1\co 1] \}$.  
Note that $f_m([x,y,0])=[0\co y^2x^2]$, 
and so the fiber over $[0\co 1]$ has multiplicity $2$; 
thus, $f_m$ is a small pencil. 
The orbifold fundamental group of the base 
of the pencil is $\Z*\Z_2$, and the pullback 
of $\VV_1(\Z*\Z_2)=\C^*\times \{-1\}$ along $f_m$  
yields the translated torus $\rho T \subset \VV_1(\A')$.
\end{example}

\begin{problem}
\label{prob:trans}
Do all positive-dimensional translated tori in the first 
characteristic variety of an arrangement arise in the 
manner described in Proposition \ref{prop:del}?
\end{problem} 

Finally, there are also arrangements $\A$ with 
isolated torsion points in the characteristic variety 
$\VV_1(\A)$. Here is an example, also from \cite{S02}.

\begin{example}
\label{ex:grunbaum}
Let $\A$ be the arrangement in $\C^3$ defined by the polynomial 
$Q(\A)=xyz(y-x)(y+x)(2y-z)(y-x-z)(y-x+z)(y+x+z)(y+x-z)$.  
Then $\VV_1(\A)$ has $10$ local components, corresponding to 
$7$ triple and $3$ quadruple points, $17$ components 
corresponding to braid sub-arrangements, $1$ component 
corresponding to a $\operatorname{B}_3$ sub-arrangement, 
$3$ translated components corresponding to deleted 
$\operatorname{B}_3$ sub-arrangements, and $2$ isolated 
points of order $6$.  
\end{example}

\begin{problem}
\label{prob:tors v1}
Find a concrete description of the $0$-dimensional components of the 
first characteristic variety $\VV_1(\A)$ of an arrangement $\A$. 
Are all such components determined by the intersection lattice $L(\A)$?
\end{problem} 

Answering in the affirmative Problems \ref{prob:trans} 
and \ref{prob:tors v1} would lead to a solution 
(at least for $q=s=1$) of the following well-known 
problem, which is central to the theory of hyperplane 
arrangements.

\begin{problem}
\label{prob:cv comb}
Given a hyperplane arrangement $\A$, are the characteristic varieties 
$\VV^q_s(M(\A),\C)$ determined by the intersection lattice $L(\A)$?
\end{problem} 

\section{The Milnor fibration of an arrangement. I.}
\label{sect:mf arr}

In this section, we introduce our main object of study:  
the Milnor fibration attached to a multi-arrangement.  
In the process, we describe several covering spaces 
related to the Milnor fiber.

\subsection{The Milnor fibration}
\label{subsec:mf}

Let $\A$ be an arrangement of hyperplanes in $\C^{d+1}$.   
Recall we associated to each multiplicity vector $m\in \N^{\abs{\A}}$ 
a homogeneous polynomial 
\begin{equation}
\label{eq:qam bis}
\qam=\prod_{H\in \A} f_H^{m_H}
\end{equation}
of degree $N=\sum_{H\in \A} m_H$. Note that $\qam$ is a 
proper power if and only if $\gcd(m)>1$, where 
$\gcd(m)=\gcd(m_H\colon H\in \A)$.

As before, let $M(\A)$ be the complement of the arrangement. 
The polynomial map $\qam\colon \C^{d+1} \to \C$ restricts 
to a map $\qam\colon M(\A) \to \C^{*}$.  As shown by 
J.~Milnor \cite{Mi} in a much more general context, $\qam$ 
is the projection map of a smooth, locally trivial bundle, 
known as the {\em (global) Milnor fibration}\/ of the 
multi-arrangement $(\A,m)$.  The typical fiber of this fibration, 
\begin{equation}
\label{eq:mfam}
\fma=\qam^{-1}(1)
\end{equation}
is called the {\em Milnor fiber}\/ of the multi-arrangement, 
while the Milnor fibration itself is written as 
\begin{equation}
\label{eq:mfib}
\xymatrixcolsep{30pt}
\xymatrix{
\fma\ar[r]  & M(\A) 
\ar^(.55){\qam}[r] & \C^*. 
}
\end{equation}

Clearly, the Milnor fiber is a smooth manifold of dimension $2d$. 
In fact, $\fma$ is a Stein domain of complex dimension $d$, 
and thus has the homotopy type of a finite CW-complex 
of dimension $d$. 

In the case when all the multiplicities 
$m_H$ are equal to $1$, the polynomial 
$Q(\A)=\qam$ is the usual defining polynomial for the 
arrangement, and has degree $n=\abs{\A}$.  Moreover, 
$F(\A)=\fma$ is the usual Milnor fiber of $\A$.

\begin{remark}
\label{rem:mdep}
Although the polynomials $\qam$ depend on the choice of 
multiplicities, they all have the same zero set; thus,  
they all share the same complement, namely, $M(\A)$.  On the 
other hand, the Milnor fibers $\fma$ do depend on the 
various choices of $m$, and not just on $\A$. The next example 
illustrates this point.
\end{remark}

\begin{example}
\label{ex:point}
Let $\A$ be the arrangement in $\C$ consisting of the single 
subspace $H=\set{0}$, and assign a multiplicity $m\in \N$ 
to that point. Then $\qam=z^m$, and so $\fma$ is the 
set of $m$-roots of unity. 
\end{example}

Of course, the Milnor fibers in the previous example 
are connected only when $m=1$.  For an arbitrary 
arrangement $\A$, the number of connected 
components of $\fma$ equals $\gcd(m)$.  

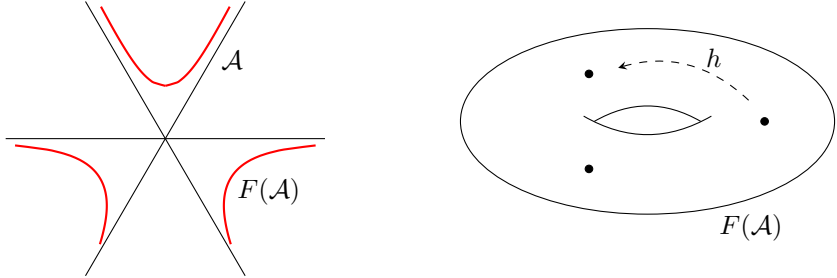
\begin{figure}
\[
\begin{tikzpicture}[baseline=(current bounding box.center),scale=0.7]  
\clip (0,0) circle (3);
\foreach \a in {0, 60,...,359}
      \draw(0, 0) -- (\a:8);
\draw[style=thick,color=red, scale=1,
domain=1:2.5,smooth,variable=\y] 
plot ({ 0.5*sqrt (\y^3-1)/sqrt(\y) },{\y});
\draw[style=thick,color=red, scale=1,
domain=1:2.5,smooth,variable=\y] 
plot ({ -0.5*sqrt (\y^3-1)/sqrt(\y) },{\y});
\draw[style=thick,color=red, scale=1.6,
domain=-1.25:-0.08,smooth,variable=\y] 
plot ({ 0.5*sqrt (1-\y^3)/sqrt(-\y) },{\y});
\draw[style=thick,color=red, scale=1.6,
domain=-1.25:-0.08,smooth,variable=\y] 
plot ({ -0.5*sqrt (1-\y^3)/sqrt(-\y) },{\y});
 \node at (0.7,1.5) [label=right:$\A$] {};
 \node at (1,-1) [label=right:$F(\A)$] {};
\end{tikzpicture}
\hspace*{0.7in}
\begin{tikzpicture}[baseline=(current bounding box.center),scale=0.7]   
        \draw (-1,0) to[bend left] (1,0);
        \draw (-1.2,.1) to[bend right] (1.2,.1);
        \draw[rotate=0] (0,0) ellipse (100pt and 50pt); 
        \node[circle,draw,inner sep=1pt,fill=black] at (2.2,0)  {};
        \node[circle,draw,inner sep=1pt,fill=black] at (-1.1,0.9)  {};
        \node[circle,draw,inner sep=1pt,fill=black] at (-1.1,-0.9)  {};
         \draw[->,>=stealth,shorten >=1pt,dashed] (1.9, 0.4) to[bend right] (-0.6,1);
         \node at (0.75,1.2) [label=right:$h$] {};
         \node at (1,-2) [label=right:$F(\A)$] {};
\end{tikzpicture}
\]
\caption{Milnor fiber and monodromy for $Q(\A)=y(y^2-4x^2)$} 
\label{fig:mf pencil} 
\end{figure}

\subsection{The geometric monodromy}
\label{subsec:mf mono}

Consider the Milnor fibration $M(\A)\to \C^*$ defined by the polynomial 
$Q=Q_m(\A)$.  For each  $\theta\in [0,1]$, let $F_{\theta}$ be the 
fiber over the point $e^{2\pi \ii \theta}\in \C^*$, so that $F_0=F_1=F_m(\A)$. 

For each $z\in M(\A)$, the path $\gamma_{\theta,z}\colon [0,1] \to \C^*$, 
$t\mapsto e^{2\pi \ii t \theta}$ lifts to a path 
$\tilde\gamma_{\theta,z}\colon [0,1] \to M(\A)$, 
$t\mapsto e^{2\pi \ii t \theta/N} z$, satisfying $\tilde\gamma_{\theta,z}(0)=z$. 
Notice that $Q(\tilde\gamma_{\theta,z}(1))= e^{2\pi \ii  \theta} Q(z)$;  
thus, if $z\in F_{0}$, then $\tilde\gamma_{\theta,z}(1)\in F_{\theta}$. 

By definition, the {\em monodromy}\/ of the Milnor fibration 
is the diffeomorphism $h\colon F_0\to F_1$ given by 
$h(z)=\tilde\gamma_{1,z}(1)$.  In view of the above 
discussion, this diffeomorphism can be written as 
\begin{equation}
\label{eq:hfam}
h\colon \fma\to \fma, \quad 
z \mapsto e^{2\pi \ii/N} z.
\end{equation} 

Clearly, the map $h$ has order $N$.  Furthermore, note 
that the complement $M(\A)$ is homotopy equivalent 
to the mapping torus of $h$:
\begin{equation}
\label{eq:map torus}
M(\A) \simeq \fma \times [0,1]/(z,0) \sim  (h(z),1).
\end{equation}

\begin{example} 
\label{ex:mf boolean}
Let $\B_n$ be the Boolean arrangement in $\C^{n}$, and 
identify the complement $M(\B_n)$ with the algebraic torus 
$(\C^*)^n$.  Given a multiplicity vector $m$, the  
map $Q_m(\B_n)\colon (\C^*)^n \to \C^*$, 
$z\mapsto z_1^{m_1}\cdots z_{n}^{m_n}$ is 
a morphism of algebraic groups. Hence, the 
Milnor fiber 
\begin{equation}
\label{eq:fbnm}
F_m(\B_n)= \ker (Q_m(\B_n))
\end{equation}
is an algebraic subgroup, realized as 
the disjoint union of $\gcd(m)$ copies of $(\C^*)^{n-1}$.  
The monodromy automorphism,  
$h\colon F_m(\B_n)\to F_m(\B_n)$, permutes those 
copies in a circular fashion. 

If $\gcd(m)=1$, the algebraic subgroup $F_m(\B_n)$ is, 
in fact, an algebraic torus isomorphic to $(\C^*)^{n-1}$.   
Moreover, the monodromy automorphism is isotopic 
to the identity, through the isotopy 
$h_t(z) = e^{2\pi \ii t/N} z$.  Thus, the bundle 
$F_m(\B_n) \to M(\B_n) \to \C^*$ is trivial in this case.
\end{example}  

\subsection{Comparing two Milnor fibrations}
\label{subsec:mf boolean}

The previous example is, in some sense, a classifying object 
for Milnor fibrations of arrangements.  

To make this statement more precise, let $\A=\{H_1,\dots ,H_n\}$ 
be an (ordered) essential arrangement of hyperplanes 
in $\C^{d+1}$, with defining linear forms $f_1,\dots , f_n$. 
Recall we defined in \S\ref{subsec:abelcov} a linear 
map $\iota(\A)\colon \C^{d+1}\to\C^n$  by 
$\iota(\A)(z)=(f_1(z), \dots ,f_n(z))$.  Restricting this map 
to the complement yields an embedding of  
$M(\A)$ into  $M(\B_n)=(\C^*)^n$.  

Since $\A$ and $\B_n$ have the same number of hyperplanes, 
a multiplicity vector for one of them yields a multiplicity 
vector for the other.   The next lemma shows that the 
corresponding Milnor fibrations are compatible. 

\begin{lemma}
\label{lem:iota mab}
For each $m\in \N^n$, the map $\iota(\A)\colon M(\A)\inj M(\B_n)$ 
is compatible with the Milnor fibrations $\qam\colon M(\A)\to \C^*$ 
and $Q_m(\B_n)\colon M(\B_n)\to \C^*$.  
\end{lemma}

\begin{proof}
Write $\qam=\prod_{i=1}^n f_i^{m_i}$ and 
$Q_m(\B_n)=\prod_{i=1}^n w_i^{m_i}$.  Clearly, then,  
the following equality holds:       
\begin{equation}
\label{eq:qmbn}
\qam=Q_m(\B_n)\circ \iota(\A).
\end{equation}

Thus, the map $\iota(\A)$ restricts to 
an inclusion $\iota_m(\A)\colon \fma\inj F_m(\B_n)$,  
which fits into the commuting diagram 
\begin{equation}
\xymatrixcolsep{36pt}
\xymatrixrowsep{26pt}
\label{eq:iota milnor}
\xymatrix{
\fma \ar[r]  \ar^{\iota_m(\A)}[d] 
& M(\A)  \ar^(.55){\qam}[r] \ar^{\iota(\A)}[d] 
& \C^*  \ar@{=}[d] \\
F_m(\B_n) \ar[r] 
& M(\B_n) \ar^(.55){Q_m(\B_n)}[r]& \C^*.
}
\end{equation}
In other words, $\iota(\A)$ is a bundle map, as claimed.
\end{proof}

As a consequence of this   
lemma, we may expresses the Milnor fiber of $(\A,m)$ 
as the intersection of two familiar objects.  

\begin{corollary}
\label{cor:fam bnm}
The Milnor fiber $\fma$ is obtained by intersecting the 
complement $M(\A)$, viewed as a subvariety of the algebraic 
torus $M(\B_n)=(\C^*)^n$ via the inclusion $\iota(\A)$, 
with $F_m(\B_n)\cong \coprod_{\gcd(m)} (\C^*)^{n-1}$, 
viewed as an algebraic subgroup of $(\C^*)^n$: 
\[
\fma = M(\A) \cap F_m(\B_n).
\]
\end{corollary}

\subsection{Induced homomorphisms on $\pi_1$}
\label{subsec:induced}

Using the above comparison between the two Milnor 
fibrations, we identify now the homomorphism induced 
on fundamental groups by the projection map 
$Q_m(\A)\colon M(\A)\to \C^*$. As usual, denote by 
$\set{x_H \mid H\in \A}$ the standard generating 
set for $\pi_1(M(\A))$.

\begin{prop}
\label{prop:qm pi1}
The induced homomorphism 
$Q_m(\A)_{\sharp} \colon \pi_1(M(\A))\to \pi_1(\C^*)=\Z$  
is given by  $x_H \mapsto m_H$. 
\end{prop}

\begin{proof}
From \eqref{eq:qmbn}, we know that $\qam=Q_m(\B_n)\circ \iota(\A)$. 
By Lemma \ref{lem:H1ofA}, the homomorphism 
$\iota(\A)_{\sharp}\colon \pi_1(M(\A))\to \pi_1((\C^*)^n)$ may be 
identified with the abelianization map, $\ab\colon \pi_1(M(\A))\to \Z^n$. 

On the other hand, the homomorphism induced in $\pi_1$ by the 
algebraic morphism $Q_m(\B_n)\colon (\C^*)^n \to \C^*$, 
$w\mapsto w_1^{m_1}\cdots w_n^{m_n}$ may be identified with 
the linear map $m\colon \Z^n \to \Z$, $v\mapsto \sum_{i=1}^n m_iv_i$. 
The conclusion follows.
\end{proof}

Using this proposition, we can derive in a 
novel way a well-known lower-bound on the first Betti 
number of the Milnor fiber of a (multi-) arrangement. 

\begin{corollary}
\label{corollary:qpi}
Suppose $\gcd(m)=1$. Then $b_1(\fma)\ge n-1$, where $n=\abs{\A}$. 
\end{corollary}

\begin{proof}
From the assumption, we know that the Milnor fibers  
$\fma$ and $F_m(\B_n)$ are connected.  Applying 
the $\pi_1$ functor to diagram \eqref{eq:iota milnor}, 
and using the homotopy long exact sequences associated to 
the respective Milnor fibrations, as well as Lemma \ref{lem:H1ofA} 
and Proposition \ref{prop:qm pi1}, we obtain the following 
commuting diagram with exact rows:
\begin{equation}
\label{eq:ladder}
\xymatrixcolsep{18pt}
\xymatrixrowsep{12pt}
\xymatrix{
1\ar[r]&\pi_1(\fma)\ar[rr] \ar@{->>}^{\iota_m(\A)_{\sharp}}[dd] 
\ar@{..>>}_(.4){}@/_28pt/[ddd] 
&& \pi_1(M(\A)) \ar@{->>}^{\iota(\A)_{\sharp}}[dd]
\ar^(.62){\qam_{\sharp}}[rr] \ar@{..>>}_(.38){\ab}@/_28pt/[ddd] 
&& \Z  \ar[r]  \ar@{=}[dd]  & 1 \,\phantom{.}
\\  \\
1\ar[r]&\pi_1(F(\B_n))\ar[rr]  \ar@{=}[d] && \pi_1(M(\B_n))   
\ar^(.62){Q_m(\B_n)_{\sharp}}[rr]   \ar@{=}[d]  && \Z  \ar[r] & 1 \, . \\ 
& \Z^{n-1} && \Z^n \ar@{->>}_{m}[urr]
}
\end{equation}

This diagram yields an epimorphism $\pi_1(\fma) \surj \Z^{n-1}$,  
indicated by a dotted arrow in \eqref{eq:ladder}.  This completes the proof.
\end{proof}

We will come back to this topic in \S\ref{sect:mf bis}, 
where we will use a different approach  
to provide improved estimates on the Betti numbers of 
the Milnor fiber.

\subsection{The Milnor fiber as an infinite cyclic cover}
\label{subsec:zcov}

A standard construction turns any continuous 
map into a fibration, at least up to homotopy, 
see e.g.~\cite[\S 4.3]{Ha}.    In the case of  
an inclusion $F\inj E$, the construction is particularly 
simple: let $Y$ be the space of paths in $E$ 
starting in $F$.  Then the inclusion 
$F\inj  Y$, $a\mapsto \const_{a}$ is a homotopy 
equivalence, while the map $g\colon Y \to E$, 
$\gamma\mapsto \gamma(1)$ is a fibration.  

Now suppose $F \to E \xrightarrow{p} B$ is a fibration, 
and all three spaces are path-connected CW-complexes.  
Pick a basepoint $e_0\in E$, and let $Z:=g^{-1}(e_0)$ be 
the homotopy fiber of the inclusion $F\inj E$. 
Let $\Omega B$ be the space of loops in $B$, based at 
$b_0=p(e_0)$.  Then $\Omega B$ has the homotopy 
type of a CW-complex, and we have a map 
$f\colon \Omega B \to Z$, which sends a loop at $b_0$  
to its lift at $e_0$, traversed in the opposite direction. 
Comparing the long exact sequences of homotopy 
groups of the two fibrations, $F \to E \to B$ and 
$Z\to Y\to E$, we conclude that $f$ is a 
homotopy equivalence.

In our situation, consider the Milnor fibration \eqref{eq:mfam}
associated to $(\A,m)$.  The homotopy fiber of the 
inclusion $\fma \to M(\A)$, then, is homotopic to 
$\Omega \C^* \simeq \Z$. 
 
\begin{prop}
\label{prop:mf cyclic} 
The homotopy fibration defined by the inclusion 
$\fma \inj M(\A)$ is equivalent to the regular, infinite 
cyclic cover classified by the homomorphism 
$\pi_1(M(\A))\surj \Z$, $x_H \mapsto m_H$.
\end{prop}

\begin{proof}
The exponential map $\exp\colon \C\to \C^*$,  
$\zeta\mapsto e^{2\pi \ii \zeta}$ is a regular $\Z$-cover.  
Pulling back this cover along the map 
$Q=\qam\colon M(\A)\to \C^*$, we obtain the diagram
\begin{equation}
\label{eq:zcov}
\xymatrix{
& & \Z \ar[r] \ar[d] & \Z \ar[d] \\
& F \ar^{\simeq}[r] \ar^{\simeq}[d] & 
E \ar^{q}[r] \ar^{p}[d] \ar@{..>>}_{\phi}[dl] & \C \ar^{\exp}[d] \\
\Z \ar[r]  & \fma \ar^{j}[r]  & M(\A)  \ar^(.55){Q}[r]
& \C^* 
}
\end{equation}
where $F$ is the homotopy fiber of $q$, 
the homotopy equivalence $F \to \fma$ is the restriction 
of $p$ to (homotopy) fibers, and $\Omega \C^*\simeq \Z$ 
is the homotopy fiber of $j$. 

By construction, the pullback cover, $E\to M(\A)$, is 
a regular $\Z$-cover, classified by the homomorphism 
$Q_{\sharp}\colon \pi_1(M(\A))\to \Z$.  Furthermore, 
the homotopy equivalence 
$\phi\colon E\xrightarrow{\,\simeq\,}\fma$ defined 
by diagram \eqref{eq:zcov} 
satisfies $j\circ \phi \simeq p$. Thus, the cover $E\to M(\A)$ 
may be identified with the homotopy fibration associated to 
the inclusion $\fma \to M(\A)$.  The desired conclusion now 
follows from Proposition \ref{prop:qm pi1}.
\end{proof}

\begin{remark}
\label{rem:df}
In general, an infinite cyclic cover of a finite CW-complex 
need not have finite Betti numbers, let alone have the homotopy 
type of a finite CW-complex. For instance, if $X=S^1\vee S^1$ 
is a wedge of two circles, and $Y\to X$ is an arbitrary $\Z$-cover, 
then $b_1(Y)=\infty$.  In our situation, though, the geometry 
of the situation conspires to insure that all the $\Z$-covers 
corresponding to Milnor fibrations of arrangements have 
finite cell decompositions.  For more on the geometric 
and homological finiteness properties of (infinite) abelian 
covers, we refer to \cite{Su-pisa, Su12, SYZ-annali}.
\end{remark}

\subsection{The Milnor fiber as a finite cyclic cover}
\label{subsec:mf cover}

As before, fix a multiplicity vector $m$ on $\A$ 
with $\gcd(m)=1$, and set $N=\sum_{H\in \A} m_H$.  
The monodromy automorphism $h\colon \fma\to \fma$, 
given by $h(z)=e^{2\pi \ii/N} z$,  generates a cyclic group $\Z_N$.    
As noted in \cite{OR93}, this group acts freely on $\fma$, 
and the quotient space may be identified with the projective  
complement, $U(\A)$. We  thus have a regular, $N$-fold 
cyclic cover, $\fma \to U(\A)$.   As shown in \cite{CS95, CDS03}, 
this cover can be described in purely group-theoretic terms.  

We give here a self-contained, full-detail proof of this 
basic result.  The proof uses a somewhat different 
approach, which draws in part from \cite{Su-conm,DeS12}.  
To start with, recall from \S\ref{subsec:ma} that the Hopf fibration 
$\pi \colon \C^{d+1} \setminus \set{0} \to \CP^{d}$ 
restricts to a (trivializable) bundle map, 
$\pi(\A)\colon M(\A)\to U(\A)$, with fiber $\C^*$.  

\begin{theorem}
\label{thm:mf} 
The map $\pi(\A)\colon M(\A)\to U(\A)$ restricts 
to a regular, $\Z_N$-cover $\pi_m(\A) \colon \fma\to U(\A)$. 
Furthermore, this cover is classified by the homomorphism 
\[
\delta_m\colon \pi_1(U(\A))\surj \Z_N, \quad 
\overline{x}_H \mapsto m_H \bmod N.
\]
\end{theorem}

\begin{proof} 
For simplicity, we will drop the arrangement $\A$ from the 
notation.  By homogeneity of the polynomial $Q_m$, 
we have that $Q_m(wz)=w^N Q_m(z)$, for every $z\in M$ 
and $w\in \C^{*}$. Thus, the restriction of $Q_m$ to a fiber of 
$\pi$ may be identified with the covering projection 
$q\colon \C^{*}\to \C^{*}$, $q(w)=w^N$.  

Now, if both $z$ and $wz$ belong to $F_m$, then 
$Q_m(z)=Q_m(wz)=1$, and so $w^N=1$. Thus, the 
restriction $\pi_m=\left. \pi \right|_{F_m} \colon F_m\to U$ 
is the orbit map of the free action of the geometric 
monodromy on $F_m$.   
Hence, $\pi_m$ is a regular $\Z_N$-cover, as claimed.

To determine the classifying homomorphism $\delta_m \colon 
\pi_1(U)\to \Z_N$ for this cover, we will make some auxiliary 
constructions. We refer to  diagram \eqref{eq:hopf milnor} 
for the various spaces and maps introduced along the way. 
\begin{equation}
\label{eq:hopf milnor}
\xymatrixcolsep{36pt}
\xymatrixrowsep{30pt}
\xymatrix{
Y_m \ar^{p_2}[r] \ar^{\psi}[d] \ar^{p_1}[dr] 
& \C^* \ar[d] \ar^{q}[dr] \\
F_m\ar[r] \ar_(.5){\pi_m}[dr] 
& M \ar^{\pi}[d] \ar^{Q_m}[r]  & \C^*\, .\\
& U}
\end{equation}

Let $Y_m=\set{(z,w)\in M\times \C^{*}\mid Q_m(z)=w^N}$, 
and denote by $p_1\colon Y_m\to M$ and $p_2\colon Y_m\to \C^{*}$ 
the restrictions of the coordinate projections.   Clearly, 
\begin{equation}
\label{eq:qmp1}
Q_m\circ p_1=q\circ p_2.
\end{equation}

Thus, $p_1$ is a regular $\Z_N$-cover, obtained as the 
pullback of $q$ along $Q_m$. Of course, the cover $q$ is 
classified by the canonical projection $\pi_1(\C^*)=\Z\to \Z_N$.  
Hence, by Proposition \ref{prop:qm pi1}, the cover $p_1$ is 
classified by the homomorphism 
\begin{equation}
\label{eq:del1}
\pi_1(M)\to \Z_N, \quad  
x_H\mapsto m_H \bmod N.
\end{equation}

We now proceed to the other side of 
diagram \eqref{eq:hopf milnor}. 
Given an element $(z,w)\in Y_m$, note that 
$Q_m(w^{-1}z)=w^{-N}Q_m(z)=1$.  Hence, we may define a map 
$\psi\colon Y_m\to F_m$ by $\psi(z,w)=w^{-1}z$.  Plainly, 
\begin{equation}
\label{eq:qmpsi1}
\pi_m\circ \psi = \pi\circ p_1.
\end{equation}
 
Therefore, the pullback of 
$\pi_m$ along $\pi$ coincides with the cover $p_1$.
Now, by Lemma \ref{lem:H1ofA}, the homomorphism 
$\pi_{\sharp} \colon \pi_1(M)\to \pi_1U)$ 
is given by $x_H \mapsto \overline{x}_H$.   
Thus, the cover $p_1$ is 
classified by the homomorphism 
\begin{equation}
\label{eq:del2}
\pi_1(M)\to \Z_N, \quad 
x_H\mapsto \delta_m(\overline{x}_H).
\end{equation}

Comparing the answers obtained in \eqref{eq:del1} and 
\eqref{eq:del2}, we find that 
$\delta_m(\overline{x}_H)=m_H \bmod N$, 
and we are done.
\end{proof}

\section{The Milnor fibration of an arrangement. II.}
\label{sect:mf bis}

We continue our discussion of the Milnor fibration with 
a study of its homology groups and monodromy operators. 

\subsection{Homology of the Milnor fiber}
\label{subsec:mf homology}
Using the interpretation of  the Milnor fiber of a multi-arrangement 
as a finite cyclic cover of the projectivized complement, we may 
compute the homology groups of the Milnor fiber and the 
characteristic polynomial of the algebraic monodromy in each 
degree, at least over a field of characteristic not dividing the 
order of the cover.

\begin{theorem}
\label{thm:hmf}
Let $(\A,m)$ be a multi-arrangement. Set 
$N=\sum_{H\in \A} m_H$, and let $\delta_m\colon 
\pi_1(U(\A))\to \Z_N$ be the homomorphism given by 
$x_H\mapsto m_H \bmod N$. If $\k$ is a coefficient field of 
characteristic not dividing $N$, then
\begin{enumerate}
\item 
The homology groups of the Milnor fiber $\fma$ are given by
\begin{equation}
\label{eq:eko bis}
\dim_\k H_q(\fma,\k)=\sum_{s\ge 1}\abs{\V^q_s (U(\A),\k)\cap 
\im(\widehat{\delta_{m}})}.
\end{equation}
\item The characteristic polynomial of the algebraic 
monodromy of the Milnor fibration, 
$h_*\colon H_q(\fma,\k) \to H_q(\fma,\k)$, is given by
\begin{equation}
\label{eq:charpoly bis}
\Delta^{\k}_{h,q}(t) =  \prod_{s\ge 1} 
\prod_{\substack{\zeta\in \k^* \colon \zeta^N=1, \\
\zeta^m \in  \V^q_s(U(\A),\k)}}  (t-\zeta).
\end{equation}
\end{enumerate}
\end{theorem}

\begin{proof}
By Theorem \ref{thm:mf}, the restriction of the Hopf fibration 
to the Milnor fiber, $\pi_m(\A)\colon F_m(\A)\to U(\A)$, 
is a regular, $\Z_N$-cover classified by the homomorphism 
$\delta_m$.  The desired conclusions follow from 
Theorem \ref{thm:betti cover}.
\end{proof}

\begin{remark}
\label{rem:1evalue}
The $1$-eigenspace of the linear transformation $h_*$ acting 
on the $\k$-vector space $H_q(\fma,\k)$ is isomorphic to $H_q(U(\A),\k)$, 
which has dimension $b_q(U(\A))$, independent of $\k$. 
In particular, $h_*$ is the identity on $H_q(\fma,\k)$  if and 
only if $\Delta^{\k}_{h,q}(t)=(t-1)^{b_q(U(\A))}$.  The other 
eigenvalues of $h_*$, and the dimensions of their eigenspaces, 
depend in principle on both the multiplicity vector $m$ and the 
characteristic of $\k$.
\end{remark}

\begin{remark}
\label{rem:primitive}
In the case when all multiplicities $m_H$ are equal to $1$, 
we can say more: in the product  \eqref{eq:charpoly bis}, 
only non-primitive roots of $1$ appear, provided $\k=\C$ 
and $q<d$, see \cite[Corollary 2.2]{CS95}.  For more 
refined versions of this result, we refer to 
\cite[Proposition 2.1]{Li02b} and \cite[Theorem 3.13]{MP09}. 
\end{remark}

\begin{example}
\label{ex:mf pencil}
Let $\mathcal{P}_n$ be the pencil of $n+1$ lines through the 
origin of $\C^2$ from Examples  \ref{ex:pencil} and \ref{ex:cv free}. 
We know that the Milnor fiber $F_m(\mathcal{P}_n)$ is an $N$-fold 
cover of $U(\mathcal{P}_n)= \C\setminus \{\text{$n$ points}\}$, 
where $N$ is the sum of the multiplicities.  We also know 
that $F_m(\mathcal{P}_n)$ is a Riemann surface with 
$N$ punctures.  A standard Euler characteristic argument now 
shows that the genus of this surface is $g=\frac{N(n-2)}{2}+1$; thus,  
\begin{equation}
\label{eq:b1pn}
\dim_{\k} H_1 (F_m(\mathcal{P}_n),\k)=
N(n-1)+1,
\end{equation}
for all fields $\k$.  

Alternatively, we know from \eqref{eq:cv free} 
that $\VV^1_1(U(\mathcal{P}_n),\k)=\cdots = 
\VV^1_{n-1}(U(\mathcal{P}_n),\k)=(\k^*)^n$ and 
$\VV^1_n(U(\mathcal{P}_n),\k)= \{1\}$.  Thus, 
we may recover the calculation from \eqref{eq:b1pn} 
by means of formula \eqref{eq:eko bis}, at least in the case 
when $\ch(\k)\nmid N$. 
\end{example}

In view of Theorem \ref{thm:v1a}, the above 
result yields a rather explicit formula for 
the first Betti number of the Milnor fiber $F(\A)$, 
and for the characteristic polynomial $\Delta=\Delta^{\C}_{h,1}$ 
of its algebraic monodromy, 
$h_*\colon H_1(F(\A),\C)\to  H_1(F(\A),\C)$. 
Let $\Phi_r$ be the $r$-th cyclotomic polynomial, 
and let  $\varphi(r)$ be its degree.

\begin{corollary}
\label{cor:b1 mf}
Let $\A$ be an arrangement of $n$ hyperplanes, and 
let $F(\A)$ be its Milnor fiber. Then,
\begin{equation}
\label{eq:charpoly}
\Delta(t) = (t-1)^{n-1} \cdot \prod_{1\ne r | n} \Phi_r(t)^{\depth(\delta^{n/r})}, 
\end{equation}
where $\delta\colon \pi_1(U(\A))\to \C^{*}$ is the ``diagonal" 
character, given by $\delta(x_H) = e^{2\pi\ii/n}$, for all $H\in \A$, 
and $\depth(\rho):=\max\{s \mid \rho\in \VV^1_s(U(\A))\}$. 
In particular, 
\begin{equation}
\label{eq:betti1 mf}
b_1(F(\A))= n-1 + \sum_{1\ne r | n} \varphi(r) 
\depth \big(\delta^{n/r}\big), 
\end{equation}
\end{corollary}

\begin{remark}
\label{rem:trivial}
In formula \eqref{eq:betti1 mf}, only the essential components of 
$\VV_1(\A)$ may contribute to the sum.  Indeed, if a component 
lies on a (proper) coordinate subtorus $C\subset \mathbb{T}(\A)$, 
then the diagonal subtorus, $D=\{(t,\dots ,t) \mid t\in \C^*\}$, 
intersects $C$ only at the origin. In particular, local components, 
or, more generally, components arising from multinets supported 
on proper sub-arrangements of $\A$, do {\em not}\/ produce 
jumps in the first Betti number of $F(\A)$.  

On the other hand, if $(\A,m)$ is a multi-arrangement 
with some $m_H>1$, then even the non-essential 
components in the characteristic varieties of $U(\A)$ may produce 
jumps in the Betti numbers of $F_m(\A)$, according 
to formula \eqref{eq:eko bis}.
\end{remark}

\begin{remark}
\label{rem:cdo}
In \cite[Theorem 13]{CDO03}, Cohen, Dimca, and Orlik 
give some nice combinatorial upper bounds on the exponents 
of the cyclotomic polynomials appearing in formula 
\eqref{eq:charpoly}.  For more on this, we refer to \cite[\S 6.4]{Di04}.
\end{remark}

\subsection{Discussion an examples}
\label{subsec:discuss hmf}

We now illustrate with a few examples the range of 
applicability of Theorem \ref{thm:hmf} (see \cite{CS95} 
for more computations of this sort.)

\begin{example}
\label{ex:mf braid arr}
Let $\A$ be the braid arrangement from Examples \ref{ex:braid arr}, 
\ref{ex:braid}, and \ref{ex:cv braid}. 
We know that $\VV_1(\A)$ has a single essential 
component, namely, the subtorus 
$T=\{ t\in (\C^*)^6 \mid t_1 t_3 t_6=t_1 t_2^{-1}=
t_3 t_4^{-1}=t_5t_6^{-1}=1 \}$. Clearly, $\delta^2\in T$, yet 
$\delta\notin T$; hence, $\Delta(t)=(t-1)^5(t^2+t+1)$.  
\end{example}

\begin{example}
\label{ex:cv mf b3}
Let $\A$ be the reflection arrangement of type $\operatorname{B}_3$ 
from Examples \ref{ex:B3} and \ref{ex:B3 bis}. We know 
that $\VV_1(\A)$ has an essential component, corresponding 
to the $(3,4)$-multinet depicted in Figure \ref{fig:b3 arr}.  
It is readily verified that this is the only such component, and 
that the diagonal subtorus  intersects this component only 
at the origin.  Hence, $\Delta(t)=(t-1)^8$. 
\end{example}

\begin{example}
\label{ex:ceva bis}
Let $\A$ be the Ceva($3$) arrangement from Example \ref{ex:ceva3}. 
The $(3,3)$-net depicted in Figure \ref{fig:ceva3} defines a rational map, 
$\CP^2 \dashrightarrow \CP^1$, $(x,y,z) \mapsto (x^3-y^3,y^3-z^3)$,  
which restricts to a pencil 
$U(\A)\to \CP^1 \setminus \{ (1,0), (0,1), (1,-1) \}$.  
Let $T$ be the essential $2$-dimensional component 
of $\VV_1(\A)$ obtained by pullback along this pencil. 
It is readily verified 
that the subgroup generated by the diagonal character $\delta$ 
intersects $\VV_2(\A)$ in two points, both lying on $T$, and 
both of order $3$.  Hence, $\Delta(t)=(1-t)^8(1+t+t^2)^2$.  
\end{example}

As the next example shows, the first Betti number 
of the Milnor fiber of an arrangement in $\C^3$ 
does not only depend on the number and type of 
multiple points of $\P(\A)$, but also on their relative position. 

\begin{example} 
\label{ex:pappus}
Consider the arrangements $\A_1$ and 
$\A_2$ defined by the polynomials 
\begin{align*}
Q(\A_1)&=xyz(x-y)(y-z)(x-y-z)(2x+y+z)(2x+y-z)(2x-5y+z),\\
Q(\A_2)&=xyz(x+y)(y+z)(x+3z)(x+2y+z)(x+2y+3z)(2x+3y+3z).
\end{align*}
The arrangement $\A_1$ is a realization of the Pappus 
configuration $(9_3)_1$, while $\A_2$ is a realization 
of the configuration $(9_3)_2$. Each projective configurations 
has $9$ double points and $9$ triple points; additionally, 
the first configuration supports a $(3,3)$-net, but the 
second one supports no essential multinet.  Applying 
formula \eqref{eq:charpoly}, we find that $\Delta_1(t) =
(t-1)^8(t^2+t+1)$ and $\Delta_2(t)=(t-1)^8$, as in \cite{CS95}. 
\end{example}

Such examples, and many others from the literature 
(see for instance \cite{MP09, BDS, Li12, Di11, Yo} for some 
recent progress in this direction) raise the following 
well-known question, dating from the 1980s. 

\begin{problem}
\label{quest:mf comb}
Given an arrangement $\A$, are the Betti numbers of the 
Milnor fiber $F(\A)$ and the characteristic polynomial 
of the algebraic monodromy determined by the intersection 
lattice $L(\A)$?
\end{problem}

\subsection{Torsion in the homology of the Milnor fiber}
\label{subsec:tors mf}

A long-standing question, raised by Randell and Dimca--N\'emethi 
among others, asks whether the Milnor fiber of an arrangement 
can have non-trivial torsion in homology.  Examples from 
\cite{CDS03} first showed that $H_1(F_m(\A),\Z)$ may have torsion, 
for suitable multi-arrangements $(\A,m)$.  In \cite{DeS12}, these 
examples were recast in a more general framework, leading to 
examples of high-dimensional hyperplane arrangements $\A$ 
for which $H_q(F(\A),\Z)$ has torsion, for some $q>1$. Let us 
sketch here this circle of ideas. 

\begin{theorem}[\cite{DeS12}]
\label{thm:milnor multi}
Suppose $\A$ admits a pointed multinet, with distinguished 
hyperplane $H$ and multiplicity $m$.  Let $p$ be a prime 
dividing $m_H$. There is then a choice of multiplicities $m'$ 
on the deletion $\A' =\A\setminus \{H\}$ such that 
$H_1(F_{m'}(\A'),\Z)$ has non-zero $p$-torsion.
\end{theorem}

\begin{proof} 
(Sketch) 
By Proposition \ref{prop:del}, the variety $\VV_1(\A')$ has a 
component of the form $\rho T$, where $T$ is a $1$-dimensional 
subtorus, and $\rho$ is a torsion character, of order 
divisible by $p$.  Using this fact, and some further machinery, 
it is shown in \cite[Theorem 6.1]{DeS12}, that, 
for all sufficiently large integers $r>1$ not divisible by $p$, 
there exists a regular, $r$-fold cyclic cover $Y\to U(\A')$ 
such that $H_1(Y,\Z)$ has $p$-torsion.  

On the other hand, 
it is also shown in  \cite[Proposition 6.7]{DeS12} that any finite 
cyclic cover of an arrangement complement is dominated by 
a Milnor fiber corresponding to a suitable choice of multiplicities. 
Hence, there exists a choice of multiplicities $m'$ for which the 
cover $\pi_{m'}(\A')\colon F_{m'}(\A')\to U(\A')$ factors through $Y$, 
and $p$ does not divide $N' :=\sum_{H\in\A'}m'_H$.  
A standard transfer argument then shows that $H_1(F_{m'}(\A'),\Z)$ 
also has $p$-torsion.
\end{proof}

\begin{example}
\label{ex:mf deleted B3}
Let $\A'$ be the deletion of the $\operatorname{B}_3$ 
arrangement from Example \ref{ex:deleted B3}.  
Proceeding as above, let $r=3$ and consider the 
$3$-fold cover $Y\to U(\A')$ classified by the 
vector $\chi=(2,1,0,0,2,2,1,1)\in (\Z_3)^8$.  
Choose $m'=(2,1,3,3,2,2,1,1)$ as a positive 
integer lift of $\chi$. Then $N'=15$, and $F_{m'}(\A')$ 
factors through $Y$; thus, there is $2$-torsion in the 
first homology of $F_{m'}(\A')$.   
Explicit calculation shows that, in fact, 
$H_1(F_{m'}(\A'),\Z)=\Z^7\oplus \Z_2\oplus \Z_2 $ 
and $\Delta_1^{\overline{\F_2}}(t)=(t-1)^7(t^2+t+1)$.
\end{example}

This  leads to a couple of natural questions (see 
also \cite{CDS03, DeS12}). 

\begin{problem}
\label{prob:tors mf}
Is there a hyperplane arrangement $\A$ (without multiplicities) 
such that $H_1(F(\A),\Z)$ has non-trivial torsion?
\end{problem}

\begin{problem}
\label{prob:tors comb}
Given a hyperplane arrangement $\A$, is the torsion in  
$H_*(F(\A),\Z)$ determined by the intersection lattice $L(\A)$?
\end{problem}  

\subsection{A polarization construction}
\label{subsec:polar}

Given a multi-arrangement $(\A,m)$, we define in \cite{DeS12} 
a new hyperplane arrangement, $\B=\A\| m$, 
called the {\em polarization}\/ of the multi-arrangement.  
The rank of $\B$ equals $\rank \A+\abs{\set{H\in\A\colon m_H\geq 2}}$, 
while the number of hyperplanes in $\B$ equals $N=\sum_{H\in\A}m_H$.  

A crucial point of this construction is the connection 
between the Milnor fiber of the (simple) arrangement $\B$ 
and the Milnor fiber of the multi-arrangement $(\A,m)$:  
the pullback of the cover $F(\B)\to U(\B)$ 
along the canonical inclusion $U(\A) \to U(\B) $ 
is equivalent to the cover $F_m(\A)\to U(\A)$. 
Using this fact, together with Theorem \ref{thm:milnor multi}, 
the following result is proved in \cite{DeS12}.

\begin{theorem}[\cite{DeS12}]
\label{thm:polar tors}
Suppose $\A$ admits a pointed multinet, with distinguished 
hyperplane $H$ and multiplicity $m$.  Let $p$ be a prime 
dividing $m_H$. 
There is then a choice of multiplicities $m'$ on the deletion 
$\A' =\A\setminus \{H\}$ such that $H_q(F(\B),\Z)$ has 
$p$-torsion, where $\B=\A' \| m'$ and 
$q=1+\abs{\set{K\in \A':  m'_K\ge 3}}$.
\end{theorem}

\begin{example}
\label{ex:polar delb3}
Let $\A'$ be the deleted $\operatorname{B}_3$ 
arrangement from Examples \ref{ex:deleted B3} 
and \ref{ex:mf deleted B3}.  Then the choice of multiplicities 
$m'=(8,1,3,3,5,5,1,1)$ produces an arrangement 
$\B=\A'\| m'$ of $27$ hyperplanes in $\C^8$, such that 
$H_6(F(\B),\Z)$ has $2$-torsion of rank $108$. 
\end{example}

\subsection{Characteristic varieties of the Milnor fiber}
\label{subsec:cvmf}

Very little is known about the homology with coefficients in 
rank $1$ local systems of the Milnor fiber of a (multi-)\linebreak 
arrangement $(\A,m)$.  Since $\fma$ is a smooth, quasi-projective 
variety, deep theorems of Arapura \cite{Ar} and Budur--Wang \cite{BW} 
guarantee that the characteristic varieties $\VV^q_s(\fma, \C)$ 
are unions of torsion-translated subtori.  Let us analyze in 
more detail these varieties, following the approach 
of Dimca and Papadima from \cite{DP11}.  

As in \S\ref{subsec:mf cover}, 
let $\pi=\pi_m(\A)\colon F_m(\A)\to U(\A)$ be the 
restriction of the Hopf fibration to the Milnor fiber, 
and let $h\colon F_m(\A)\to F_m(\A)$ be the monodromy 
of the Milnor fibration.  Since $\pi$ is a finite, regular cover, 
Proposition \ref{prop:jumpcov} guarantees that  
$\pi^*(\V^q_s(U(\A), \k))\subseteq \V^q_s(F_m(\A), \k)$ and  
$\pi^*(\RR^q_s(U(\A), \k))\subseteq \RR^q_s(F_m(\A), \k)$, 
for all $q\ge 0$ and $s\ge 1$.  

If $h_*$ acts as the identity on $H_1(F_m(\A),\k)$, more can be said.  
Indeed, Corollary \ref{cor:trivnom} implies that 
$\pi^*\colon \RR^1_s(U(\A), \k)\to \RR^1_s(F_m(\A), \k)$ is 
an isomorphism, for all $s\ge 1$. Furthermore, by Theorem \ref{thm:tcone}, 
the map $\pi^*\colon H^1(U(\A),\C^*) \to H^1(F_m(\A),\C^*)$ restricts to  
an isomorphism between the components through the identity 
of $\V_s(\A)$ and $\V^1_s(F_m(\A),\C)$. 

Let us restrict now our attention to the usual Milnor fiber, 
$F(\A)$, and its degree $1$ characteristic varieties, 
$\V_s(F(\A))=\V^1_s(F(\A),\C)$. In general, the inclusion 
$\pi^*\colon \V_s(\A) \inj \V_s(F(\A))$ is strict. 
For instance, suppose $\A$ admits a non-trivial, 
reduced multinet $\mathcal{M}$, and let $T_\mathcal{M}$ 
be the corresponding component of $\V_1(\A)$. It is 
then shown in \cite{DP11} that $\V_1(F(\A))$ has a component 
passing through the identity and containing $\pi^*(T_\mathcal{M})$ 
as a proper subset. 

\begin{example}
\label{ex:cv mf braid}
Let $\A$ be the braid arrangement from Figure \ref{fig:braid}.  
Recall from Examples \ref{ex:cv braid} and \ref{ex:mf braid arr} 
that $\VV_1(\A)$ has four local components, 
$T_1,\dots, T_4$, corresponding to the triple points, 
and an essential component $T$, corresponding 
to a $(3,2)$-net. By the above discussion, 
the characteristic variety $\VV_1(F(\A))\subset (\C^*)^7$ 
has $2$-dimensional components $\pi^*(T_1), \dots ,\pi^*(T_4)$, 
as well as  a component $W$ strictly containing $\pi^*(T)$.  
As noted in \cite{DP11}, $W$ is a $4$-dimensional
algebraic subtorus. Direct computation shows that 
$\VV_1(F(\A))$ has no other irreducible components. 
\end{example}

\begin{example}
\label{ex:hesse mf cv}
Let $\A$ be the Hessian arrangement from Example \ref{ex:hesse}.  
We know that $\VV_1(\A)$ has $10$ components passing through 
the origin, all $3$-dimensional:  $9$ of those are local components,  
while the last one, $T$, corresponds to the $(4,3)$-net 
from Figure~\ref{fig:hessian}.  As noted in \cite{DP11}, the 
component $W$ of $\VV_1(F(\A))$ containing $\pi^*(T)$ 
has dimension $9$.  In particular, this shows that $b_1(F(\A))\ge 15$.
\end{example}

\begin{problem}
\label{quest:mf cv}
Find a more precise description of the characteristic varieties 
$\VV_s(F(\A))$, and, more generally, $\VV^1_s(F_m(\A),\k)$. 
\end{problem}

\subsection{The formality problem}
\label{subsec:formal mf}

The following question was raised in \cite{PS-formal}, 
in a more general context:  Is the Milnor fiber of a 
hyperplane arrangement always formal?   
Of course, if $\rank(\A)=2$, then $F(\A)$ has the homotopy 
type of a wedge of circles, and so it is formal.  
In general, though, the answer is no, as illustrated 
by the following example from \cite{Zu}. 

\begin{example}
\label{ex:zuber}
Let $\A$ be the Ceva($3$) arrangement from Example \ref{ex:ceva3},  
and let $T$ be the $2$-dimensional, essential component of $\VV_1(\A)$ 
described in Example \ref{ex:ceva bis}. The pullback $\pi^*(T)$ 
is a $4$-dimensional subtorus of $H^1(F(\A),\C^*)=(\C^*)^{12}$, 
of the form $\exp(L)$, for some linear subspace $L\subset H^1(F(\A),\C)$. 
Using the mixed Hodge structure on the cohomology of the 
Milnor fiber, Zuber \cite{Zu} shows that $L$ 
cannot possibly be a component of the resonance variety $\RR_1(F(\A))$.  
Hence, the tangent cone formula from Theorem \ref{thm:tcone} is 
violated, and thus, the space $F(\A)$ cannot be $1$-formal. 
\end{example}

This example raises several questions, all interrelated. 

\begin{problem}
\label{quest:mf tc}
Find a concrete description of the resonance varieties 
$\RR_s(F(\A))$, and, more generally, $\RR^1_s(F_m(\A),\k)$. 
For which arrangements $\A$ does the tangent cone formula 
$\TC_1(\VV_1(F(\A)))= \RR_1(F(\A))$ hold?
\end{problem}

\begin{problem}
\label{quest:mf nonf}
Give a purely topological explanation for the non-formality of 
the Milnor fibers of arrangements such as the Ceva($3$) 
arrangement.
\end{problem}

Such an explanation could involve either an explicit 
computation of the resonance varieties of Milnor fibers 
(as in Problem \ref{quest:mf tc}), or the study of 
triple Massey products in the cohomology of the 
Milnor fiber, or perhaps some completely new method.

\begin{problem}
\label{quest:mf formal}
Given a multi-arrangement $(\A,m)$, determine whether 
the Milnor fiber $F_m(\A)$ is formal or not.  Does this formality 
property depend only on the underlying arrangement $\A$, 
or also on the multiplicity vector $m$?
\end{problem}

\section{The boundary manifold of an arrangement}
\label{sect:boundary}

The boundary manifold of a hyperplane arrangement $\A$ 
in $\C^{d+1}$ is the boundary of a regular neighborhood in 
$\CP^{d}$ of the union of the projective hyperplanes 
comprising $\P(\A)$. 
In this section, we survey a number of known results (primarily 
from  \cite{CS06, CS08}) regarding 
the cohomology ring, fundamental group, jump loci, and 
formality properties of boundary manifolds of arrangements.

\subsection{The boundary manifold}
\label{subsec:bdry nbhd}
Let $\A$ be a (central) arrangement of hyperplanes in 
$\C^{d+1}$ ($d\ge 1$) with union $V=\bigcup_{H\in \A} H$ 
and complement $M= \C^{d+1}\setminus V$. 
Likewise, let $\P(\A)=\set{\P(H) \mid H\in \A}$ 
be the corresponding arrangement in $\CP^{d}$, 
with union $W=\P(V)$ and complement $U=\P(M)$.  
A regular neighborhood $\nu(W)$ of the algebraic hypersurface 
$W\subset \CP^{d}$ may be constructed as follows. 

Let $\phi\colon\CP^{d} \to \R$ be the smooth function 
defined by $\phi([z]) = \abs{Q(z)}^2 / \norm{z}^{2n}$, where 
$Q$ is a defining polynomial for the arrangement, 
and $n=\abs{\A}$. Then, for sufficiently small $\delta>0$, 
the preimage $\phi^{-1} ([0,\delta])$ is a closed, regular 
neighborhood of $W$.   (Since $Q$ is homogeneous, 
we may simply take $\delta=1$.) 
Alternatively, one may triangulate 
$\CP^{d}$ with $W$ as a subcomplex, and take $\nu(W)$ 
to be the closed star of $W$ in the second barycentric 
subdivision.  

As shown by Durfee \cite{Durfee}, these constructions yield 
isotopic neighborhoods, independent of the choices made.  
Evidently, $\nu(W)$ is a compact, orientable, smooth 
manifold with boundary, of dimension $2d$; 
moreover, $\nu(W)$ deform-retracts onto $W$. 

The  {\em exterior}\/ of the projectivized arrangement, 
denoted $\Uc$, is the complement in $\CP^{d}$ 
of the  open regular neighborhood $\nuc(W)$. 
Clearly, $\Uc$ is a compact, connected, orientable, smooth 
$2d$-manifold with boundary. Moreover, $U$ deform-retracts 
onto $\Uc$, and thus $\Uc\simeq U$.

\begin{definition}
\label{def:bdry mfd}
The \emph{boundary manifold}\/ of a hyperplane arrangement 
$\A$ in $\C^{d+1}$ is the common boundary 
\begin{equation*}
\label{eq:bdu}
\bdU= \partial \nu(W)
\end{equation*}
of the exterior $\Uc$ and the regular neighborhood 
$\nu(W)$ defined above.
\end{definition}

Note that $\bdU$ is a compact, orientable, smooth  
manifold (without boundary), of dimension $2d-1$.   
The inclusion map $j\colon \bdU \to \Uc$ is a  
$(d-1)$-equivalence, see \cite[Proposition 2.31]{Di92}; 
in particular, $\pi_i(\bdU) \cong \pi_i(U)$ for $i<d-1$.

If $d=1$, then $\P(\A)$ consists of $n$ distinct 
points in $\CP^1$; thus, the boundary manifold $\bdU$ 
consists of $n$ disjoint small circles around those points, 
and there is not much else to say.
Consequently, we will assume from now on that 
$d \ge 2$, in which case $\bdU$ is connected.  

Let us illustrate these definitions with a couple of examples, 
extracted from \cite{CS06}. 

\begin{example}
\label{ex:bdry pencil}
Let $\A$ be a pencil of $n$ hyperplanes in $\C^{d+1}$, 
defined by the polynomial $Q=z_1^{n}-z_2^{n}$. 
If $n=1$, then $\Uc=D^{2d}$ and $\bdU=S^{2d-1}$. 
Otherwise, $U$ may be realized as the complement of 
$n-1$ parallel hyperplanes in $\C^{d}$.  
In this case, 
\[
\Uc =(D^2\setminus \{\text{$n$ disjoint 
open disks}\})\times D^{2(n-1)},
\] 
and thus $\bdU$ is diffeomorphic to the connected 
sum $\sharp^{n-1} S^{1}\times S^{2(d-1)}$.  
\end{example}  

\begin{example}
\label{ex:boundary near-pencil}
Let $\A$ be a near-pencil of $n$ planes in $\C^{3}$, 
defined by the polynomial $Q=z_1(z_2^{n-1}-z_3^{n-1})$. 
In this case, $\bdU=S^1\times\Sigma_{n-2}$, 
where $\Sigma_g=\sharp^{g} S^1\times S^1$ 
denotes the orientable surface of genus $g$ 
(see also Example \ref{ex:bd mf near-pencil}).
\end{example}  

\subsection{Homology groups and cup products}
\label{subsec:homology bdry}

As shown in \cite{CS06}, the long exact sequence 
of the pair $(\Uc,\bdU)$ breaks into split, short exact sequences, 
\begin{equation} 
\label{eq:ubdu}
\xymatrixcolsep{16pt}
\xymatrix{0 \ar[r]& H_{q+1}(\Uc,\bdU,\Z)\ar[r]
& H_q(\bdU,\Z) \ar^{j_*}[r]& H_q(\Uc,\Z) \ar[r]& 0}.
\end{equation}

By Lefschetz duality, 
$H_{q+1}(\Uc,\bdU,\Z)\cong  H^{2d -q -1}(\Uc,\Z)$. 
Since $\Uc\simeq U$, and since the homology groups of 
$U$ are torsion-free, we obtain a direct sum decomposition
\begin{equation} 
\label{eq:hom bdry}
H_q(\bdU,\Z) \cong H_q(U,\Z) \oplus H_{2d-q-1}(U,\Z). 
\end{equation}
Hence, $\Poin(\bdU,t) = \Poin(U,t) + 
t^{2d-1}\cdot \Poin(U,t^{-1})$.

The cohomology groups of $\bdU$ admit a decomposition 
similar to the one from \eqref{eq:hom bdry}. 
To describe the cup-product structure in $H^*(\bdU,\Z)$, 
we first need to review some notions. Let $A$ be a graded, 
finite-dimensional algebra over $\Z$.  Assume 
that $A$ is graded-commutative, of finite type 
(i.e., each graded piece $A^q$ is a finitely 
generated $\Z$-module), and connected 
(i.e., $A^0=\Z$).   
Then, the $\Z$-dual $\check{A} = \Hom_{\Z}(A,\Z)$ 
is an $A$-bimodule, with left and right 
multiplication given by $(a\cdot f) (b)= f(ba)$ and 
$(f \cdot a) (b) =f(ab)$, respectively.  Moreover, 
this bimodule structure is compatible with the gradings:   
if $a\in A^{q}$ and $f\in \check{A}^{p}$, then both 
$af$ and $fa$ belong to $\check{A}^{p-q}$. 

\begin{theorem}[\cite{CS06}]
\label{theo:coho double}
Let $\A$ be an arrangement in $\C^{d+1}$, 
let $A=H^*(U,\Z)$ be the cohomology ring of the projectivized 
complement, and let $\widehat{A}=H^*(\bdU,\Z)$ be 
the cohomology ring of the boundary manifold. Then 
$\widehat{A}=A\oplus \check{A}$, with 
multiplication given by $(a,f)\cdot (b,g) = (ab,ag+fb)$, 
and grading $\widehat{A}^{q}=A^{q} \oplus \check{A}^{2d-q-1}$. 
\end{theorem}

Consequently, the cohomology ring of the boundary manifold 
depends only on the intersection lattice of the arrangement.
In the case when $d=2$, the structure of this 
ring can be described in more concrete terms, as follows.  

Write $A=A^0 \oplus A^1 \oplus A^2$, and fix ordered bases, 
$\{\alpha_1,\dots ,\alpha_{n}\}$ for $A^1$ and  
$\{\beta_1,\dots , \beta_{m}\}$ for $A^2$.  
The multiplication map,  
$\mu\colon A^1 \otimes A^1 \to A^2$, is then given by 
$\mu(\alpha_i , \alpha_j) =
\sum_{k=1}^{m}\mu_{ijk}\,  \beta_k$, 
for some integer coefficients  
satisfying $\mu_{jik}=-\mu_{ijk}$. 
Now write  
$\widehat{A}= A^0 \oplus (A^1 \oplus \check{A}^2) \oplus 
(A^2 \oplus \check{A}^1) \oplus \check{A}^0$, 
and pick dual bases 
$\{\alpha'_1,\dots, \alpha'_n\}$  for $\check{A}^{1}$ and 
$\{\beta'_1,\dots , \beta'_m\}$  for $\check{A}^{2}$. 
The multiplication map  
$\hat{\mu}\colon \widehat{A}^1 \otimes \widehat{A}^1 \to \widehat{A}^2$ 
restricts to $\mu$ on $A^1 \otimes A^1$, vanishes on 
$\check{A}^2\otimes \check{A}^2$, while on $A^1\otimes \check{A}^2$,  
it is given by $\hat{\mu}(\alpha_j ,\beta'_k) =
\sum_{i=1}^{n} \mu_{ijk}\, \alpha'_i$. 
Finally, $\hat{\mu}(\alpha_i, \alpha_j')=\hat{\mu}(\beta_i ,\beta'_j)=
\delta_{ij} \omega$, where $\omega$ is the generator of $\check{A}^0$ 
dual to $1\in A^0$.

\subsection{Resonance varieties}
\label{subsec:res bdry}

The resonance varieties of the boundary manifold of an  
arrangement $\A$ in $\C^{d+1}$ were studied in detail 
in \cite{CS06, CS08}.   
If $d \ge 3$, then $H^1(\bdU,\C)=H^1(U,\C)$, and 
the resonance varieties of $\bdU$ can be expressed 
solely in terms of the resonance varieties of $U$:
\begin{equation}
\RR^{q}_{s}(\bdU)=
\begin{cases}
\RR^{q}_{s}(U)
& \text{if $q \le d-2$,} \\[3pt]
\bigcup_{i+j=s} \bigl(\RR^{d-1}_{i}(U) \cap
\RR^{d}_{j}(U)\bigr)
& \text{if $q=d-1$ or $q=d$,} \\[3pt]
\RR^{2d-k-1}_{s}(U)
& \text{if $q \ge d+1$.}
\end{cases}
\end{equation}

Now suppose $\A$ is an arrangement of $n$ planes in $\C^3$, 
and let us consider the resonance varieties $\RR^1_s(\bdU)$ 
of its boundary manifold.  For depth $s=1$, these varieties 
admit a particularly simple description:
\begin{equation}
\label{eq:r1 bdry}
\RR^1_1(\bdU)=\begin{cases}
\C^{n-1} &\text{if $\A$ is a pencil,}\\
\C^{2(n-2)} &\text{if $\A$ is a near-pencil,}\\    
H^1(\bdU,\C) &\text{otherwise.}
\end{cases}
\end{equation}

The higher-depth resonance varieties, though, can be much 
more complicated, as the following example from 
\cite[Corollary 6.11]{CS06} illustrates.   

\begin{example}
\label{ex:res lines}
Let $\A$ be an arrangement of $n+1$ planes in $\C^3$, in general 
position.   If $n\ge 4$ and $\binom{n}{2}<s<\binom{n}{2} + n-2$, 
then the resonance variety $\RR^1_s(\bdU)$ is the zero locus 
of the Pfaffians (of appropriate size) of a generic $n\times n$ 
skew-symmetric matrix. In particular, $\RR^1_s(\bdU)$ is a 
singular, irreducible variety.  
\end{example}

\subsection{Graph manifold structure}
\label{subsec:graph manifold}

As noted in \S\ref{subsec:bdry nbhd}, if $\A$ is an 
arrangement in $\C^{d+1}$, with $d>3$, then $\pi_1(\bdU)=\pi_1(U)$. 
Thus, from the point of view of the fundamental group of the 
boundary manifold, the most interesting dimension to study 
is $d=3$.  So assume for the rest of this section that 
$\A$ is an arrangement of planes in $\C^3$. In this case, 
the boundary manifold admits another interpretation, 
that arises in the work of Jiang and Yau \cite{JY93, JY98}, 
Hironaka \cite{Hi01}, and Westlund \cite{We}. 
 
Let $\P(\A)=\{\ell_1,\dots ,\ell_n\}$ be the projectivized 
line arrangement in $\CP^2$. Suppose $\P(\A)$ has $r$ 
non-transverse intersection points, i.e., points 
$p_J=\bigcap_{j \in J} \ell_j$, of multiplicity $\abs{J}\ge 3$. 
Blowing up $\CP^2$ at each of these points, we obtain an 
arrangement $\widetilde{\A}=\{L_1,\dots ,L_{n+r}\}$ in  
the rational surface 
$\widetilde{\CP}^2 \cong \CP^2 \# r \overline{\CP}^2$, 
consisting of the proper transforms of the lines of $\P(\A)$, 
together with the exceptional lines arising from the blow-ups.  

This construction realizes the boundary manifold $\bdU$ 
as a graph manifold, in the sense of Waldhausen.  The 
underlying graph $\G$ has vertex set $\sV$, with 
vertices in one-to-one correspondence with the lines 
of $\widetilde{\A}$: the vertex corresponding to $\ell_i$ is 
labeled $v_i$, while the vertex corresponding to $p_J$ 
is labeled $v_J$.  The edge set $\sE$ consists of  
an edge $e_{i,j}$ from $v_i$ to $v_j$ if $i<j$ and the 
corresponding lines $\ell_i$ and $\ell_j$ are transverse, 
and an edge $e_{J,i}$ from $v_J$ to $v_i$ if 
$p_J\in \ell_i$.  Each vertex gets assigned a weight, 
equal to the self-intersection number of the 
corresponding line in the blow-up:  
$v_i$ has weight 
$w_i=1-\abs{\set{J\mid p_J\in  \ell_i}}$, 
and $v_J$ has weight $w_J = -1$.  

\subsection{Fundamental group}
\label{subsec:pi1}

Applying a method due to Hirzebruch \cite{Hr95} to 
the graph manifold structure described above, 
Westlund \cite{We} obtained a presentation for 
the fundamental group of the boundary manifold 
of an arrangement in $\C^3$, as follows. 

Let $\Gamma$ be the weighted graph associated to 
$\P(\A)$, and choose an orientation on this graph.  Pick 
a maximal tree $\cT \subset \sE$, and list the remaining edges  
as $e_1,\dots, e_s$, where $s$ is the number of linearly 
independent cycles in $\Gamma$.   With these choices, 
the fundamental group of $\bdU(\A)$ has presentation
\begin{equation} 
\label{eq:pres}
\pi_1(\bdU)=
\left\langle
\begin{array}{l}%
 x_1,\dots,x_{n+r} \\[2pt]
y_1,\dots,y_s
\end{array}
\Bigg|
\begin{array}{ll}%
[x_i,x_j^{u_{ij}}], & (i,j)\in\sE \\ [2pt]
\prod_{j=1}^{n+r} x_j^{u_{ij}}, & 1\le i\le n
\end{array}
\right\rangle,
\end{equation}
where 
\[
u_{ij}=
\begin{cases}
w_i & \text{if $i=j$,}\\
y_k & \text{if $(i,j)$ is the $k$-th element of $\sE\setminus\cT$,}\\
y_k^{-1} & \text{if $(j,i)$ is the $k$-th element of $\sE\setminus\cT$,}\\
1 & \text{if $(i,j)$ or $(j,i)$ belongs to $\cT$,}\\
0 & \text{otherwise}.
\end{cases}
\]
Here $[a,b]=aba^{-1}b^{-1}$, $a^0=1$ is the identity element, 
and $a^b=b^{-1}ab$ for $b\neq 0$.  Note that if 
$i\neq j$ and $u_{ij}\neq 0$, then $u_{ji}=u_{ij}^{-1}$.

\begin{prop}[\cite{CS08}]
\label{prop:comm}
The presentation \eqref{eq:pres} may be simplified to a 
commu\-tator-relators presentation for $\pi_1(\bdU)$, 
with generators $x_1,\dots , x_{n-1}$ and 
$y_1,\dots ,y_s$ as above.
\end{prop}

Instead of reproducing 
the precise set of relations here, let us first 
make a remark, and then illustrate with a simple example.  

\begin{remark}
\label{rem:bdry map}
Recall from \eqref{eq:ubdu} that the inclusion map 
$j\colon \bdU \to \Uc$ induces a split surjection in 
first homology, leading to the direct sum decomposition  
$H_1(\bdU,\Z)\cong 
H_1(\Uc,\Z) \oplus H_2(\Uc,\bdU,\Z)$.
The group $H_1(\Uc,\Z)=H_1(U,\Z)$ is 
freely generated by homology classes 
$x_1, \dots , x_{n-1}$ 
corresponding to the meridians around the 
first $n-1$ lines.  On the other hand, the group 
$H_2(\Uc,\bdU,\Z)$ is isomorphic to $H_1(\Gamma, \Z)$, 
and thus admits a basis $y_1,\dots ,y_s$ 
corresponding to disks in the exterior whose boundaries 
are the chosen cycles in the graph.  With respect to these 
generating sets, then, the abelianization 
map $\ab\colon \pi_1(\bdU)\to H_1(\bdU,\Z)$ takes 
$x_i$ to $x_i$ and $y_i$ to $y_i$,   
while the induced homomorphism $j_*\colon H_1(\bdU, \Z) 
\to H_1(\Uc,\Z)$ takes  $x_i$ to $x_i$ and $y_i$ to $0$.
\end{remark}

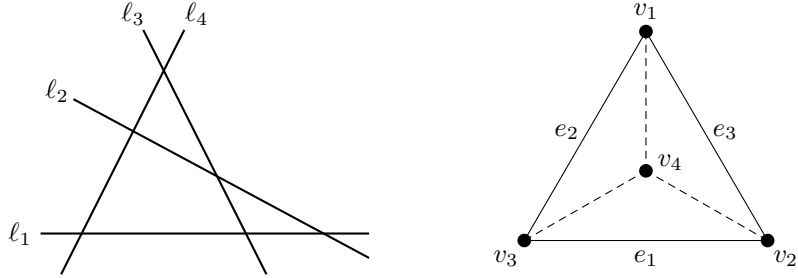
\begin{figure}
\centering
\begin{tikzpicture}[scale=0.54]
\draw[style=thick] (-0.5,3) -- (2.5,-3);
\draw[style=thick]  (0.5,3) -- (-2.5,-3);
\draw[style=thick] (-3,-2) -- (5,-2);
\draw[style=thick]  (-2.2,1.3) -- (5,-2.6);
\node at (-3.5,-2) {$\ell_1$};
\node at (-2.6,1.5) {$\ell_2$};
\node at (-0.75,3.4) {$\ell_3$};
\node at (0.8,3.4) {$\ell_4$};
\end{tikzpicture}
\hspace*{0.5in}
\begin{tikzpicture}[scale=0.8]
\draw (0,0) -- (2,3.4641) -- (4,0) -- cycle;
\draw[densely dashed] (0,0) -- (2,1.1547);
\draw[densely dashed] (2,3.4641) -- (2,1.1547);
\draw[densely dashed] (4,0) -- (2,1.1547);
\fill(0,0) circle (3.2pt);
\fill(2,3.4641) circle (3.2pt);
\fill(4,0) circle (3.2pt);
\fill(2,1.1547) circle (3.2pt);
\node at (-0.3,-0.3) {$v_3$};
\node at (2,3.8) {$v_1$};
\node at (4.3,-0.3) {$v_2$};
\node at (2.4,1.3) {$v_4$};
\node at (2,-0.3) {$e_1$};
\node at (0.7,1.8) {$e_2$};
\node at (3.3,1.8) {$e_3$};
\end{tikzpicture}
\caption{A general position arrangement and 
its associated graph}
\label{fig:4genlines}
\end{figure}

\begin{example} 
\label{ex:general position}
Let $\A$ be an arrangement of $4$ planes in $\C^3$ 
in general position.  In this case, $\G$ is the complete 
graph $K_4$.  Using the maximal tree indicated by 
dashed edges in Figure \ref{fig:4genlines}, we obtain 
the following presentation for the fundamental 
group of the boundary manifold of $\A$:
\begin{equation*} 
\label{eq:gen4}
\pi_1(\bdU)= 
\left\langle
\begin{array}{l}%
\\[-8.5pt]
x_1,x_2,x_3 \\[2.5pt]
y_{1},y_{2},y_{3}
\end{array}
\Bigg|
\begin{array}{ll}%
x_1x_2x_3=x_1x_{2}^{y_{1}^{}} x_{3}^{y_{2}^{}}
=
x_2 x_1^{y_{3}^{-1}} 
x_{3}^{y_{1}^{}} 
=
x_3 x_1^{y_{2}^{-1}} x_{2}^{y_{1}^{-1}}
 \\ [3pt]
[x_1^{},x_2^{y_{3}}] 
=
[x_1^{},x_3^{y_{2}}] 
=
[x_2^{},x_3^{y_{1}}]=1
\end{array}
\right\rangle.
\end{equation*}
\end{example}

\subsection{Alexander polynomial and characteristic varieties}
\label{subsec:cv bdry}

The next result expresses the Alexander polynomial 
of the boundary manifold in terms of the underlying 
graph structure.

\begin{theorem}[\cite{CS08}]
\label{thm:alex poly arr}
Let $\A$ be an essential arrangement of planes in  
$\C^3$, and let $\Gamma$ be the associated graph.  
Then the Alexander polynomial of the boundary 
manifold $\bdU$ is given by 
\[
\Delta_{\bdU} = \prod_{v \in \sV(\G)} (t_v-1)^{d_v-2}, 
\]
where $d_v$ denotes the degree of the vertex $v$, 
and $t_v=\prod_{i\in v} t_i$. 
\end{theorem}

Note that $\Delta_{\bdU}(1)=0$.  In view of Proposition \ref{prop:v1 3m}, 
we obtain the following decomposition of the first characteristic 
variety of $\bdU$ into irreducible components:

\begin{equation}
\label{eq:v1 bdry}
\VV^1_1(\bdU) = \bigcup_{v \in \sV(\G)\, :\, d_v\ge 3} 
\set{t_v-1=0}.  
\end{equation}
In particular, every component of $\VV^1_1(\bdU)$ is a codimension $1$ 
algebraic subtorus of the character torus of $\bdU$. 

\subsection{Formality}
\label{subsec:formal bdry}

As we saw in Examples \ref{ex:bdry pencil} and 
\ref{ex:boundary near-pencil}, for simple arrangements 
such as pencils or near-pencils, the boundary manifold 
is built out of spheres by successive product and
connected sum operations, and thus is formal.  
On the other hand, consider the following, 
equally simple, example.  

\begin{example}
\label{ex:tc lines} 
Let $\A$ be an arrangement of $5$ planes in $\C^3$, in general 
position.   From the discussion in Example \ref{ex:res lines}, it 
follows that $\RR^1_7(\bdU)$ is isomorphic  to 
$\{z\in \C^{6} \mid z_1z_6 -z_2z_5+z_3z_4=0\}$, 
an irreducible quadric with an isolated 
singular point at $0$.  By Theorem \ref{thm:tcone}, then, 
the manifold $\bdU$ is not $1$-formal.

Alternatively, formula \eqref{eq:v1 bdry} implies that 
$\VV^1_s(\bdU)\subseteq \{1\}$, for all $s\ge 1$. Thus, 
$\TC_1(\VV^1_7(\bdU))\ne \RR^1_7(\bdU)$, showing 
again that $\bdU$ is not $1$-formal. 
\end{example}

The general situation was elucidated in \cite[Theorem 9.7]{CS08}.

\begin{theorem}[\cite{CS08}] 
\label{thm:bdry}
Let $\A$ be an arrangement of planes in $\C^3$, and let $\bdU$ 
be the corresponding boundary manifold. The following are equivalent:
\begin{enumerate}
\item  \label{f1} The manifold $\bdU$ is formal. 
\item  \label{f2} The group $\pi_1(\bdU)$ is $1$-formal.
\item  \label{f3} $\TC_1(\VV^1_1(\bdU))=\RR^1_1(\bdU)$. 
\item  \label{f4} $\A$ is either a pencil or a near-pencil. 
\item  \label{f5} $\bdU$ is either $\sharp^n S^1\times S^2$ or 
$S^1\times \Sigma_{n-1}$, where $n=\abs{\A}-1$.
\end{enumerate}
\end{theorem}

\begin{proof}
Implications \eqref{f1} $\Rightarrow$ \eqref{f2} $\Rightarrow$ 
\eqref{f3} are true in general (see Appendix \ref{sect:formal}), 
implication \eqref{f4} $\Rightarrow$ \eqref{f5} was 
discussed in Examples \ref{ex:bdry pencil} and 
\ref{ex:boundary near-pencil}, while 
\eqref{f5} $\Rightarrow$ \eqref{f1} is obvious. 

To prove \eqref{f3} $\Rightarrow$ \eqref{f4}, suppose 
$\A$ is neither a pencil, nor a near-pencil.  Then, by 
formula  \eqref{eq:r1 bdry},  $\RR^1_1(\bdU)=H^1(\bdU,\C)$.
On the other hand, formula \eqref{eq:v1 bdry} shows 
that $\VV^1_1(\bdU)$ is a union of codimension $1$ 
subtori in $H^1(\bdU,\C^*)$.  Thus, 
$\TC_1(\VV^1_1(\bdU))\ne \RR^1_1(\bdU)$, and we are done.
\end{proof}
\section{The boundary of the Milnor fiber of an arrangement}
\label{sect:bdry mf}

In this section, we study the boundary of the closed Milnor fiber 
of a hyperplane arrangement, and the way it relates to the 
boundary manifold of the arrangement. 

\subsection{The closed Milnor fiber}
\label{subsec:bdF}

As usual, let $\A$ be a (central) arrangement of hyperplanes 
in $\C^{d+1}$ ($d\ge 1$), with defining polynomial 
$Q(\A)=\prod_{H\in \A} f_H$.    
Let $V=V(\A)$ be the union  of the hyperplanes in $\A$, and let 
$M=M(\A)$ be its complement.  

Now let $m$ be a (primitive) multiplicity vector for $\A$, 
and let $\qam=\prod_{H\in \A} f_H^{m_H}$.
As mentioned in \S\ref{sect:mf arr}, the  polynomial map 
$Q_m=\qam\colon \C^{d+1}\to \C$ restricts to a smooth 
fibration, $Q_m\colon M \to \C^*$, with fiber 
$F=\fma$. 

Intersecting the global Milnor fiber with  
a ball in $\C^{d+1}$ of large enough radius, we obtain a 
compact, smooth, orientable $2d$-dimensional 
manifold with boundary,
\begin{equation}
\label{eq:fc}
\Fc_m(\A) = \fma\cap D^{2(d+1)}, 
\end{equation}
which we call the {\em closed Milnor fiber}\/ of the multi-arrangement.  
Clearly, $F_m(\A)$ deform-retracts onto $\Fc_m(\A)$; hence, 
$\Fc_m(\A)\simeq F_m(\A)$.

The {\em boundary of the Milnor fiber}\/ of $(\A,m)$ is   
the compact, smooth, orientable, $(2d-1)$-dimensional manifold
\begin{equation}
\label{eq:bdfma}
\bdF_m(\A) = \fma\cap S^{2d+1}.
\end{equation}

As usual, we will abbreviate $\Fc(\A)$ and $\bdF(\A)$ 
when all multiplicities $m_H$ are equal to $1$, and will 
drop $\A$ from the notation when the arrangement is understood. 
As noted in \cite[Proposition 2.4]{Di92}, the pair $(\Fc_m, \bdF_m)$ 
is $(d-1)$-connected. In particular, if $d\ge 2$, the boundary 
of the Milnor fiber is connected, and the inclusion-induced 
homomorphism $\pi_1(\bdF_m)\to \pi_1(\Fc_m)$ is surjective. 

\subsection{The local Milnor fibration}
\label{subsec:alt mf}

Returning to the general situation, let $F_m=\fma$ be the 
global Milnor fiber of a multi-arrangement.  Recall that  
the monodromy of the fibration $F_m\to M\to \C^*$ is the 
diffeomorphism $h\colon F_m\to F_m$ given by 
$h(z)=e^{2\pi \ii/N} z$, where $N=\sum_{H\in \A} m_H$.  

Clearly, the map $h$ restricts to a diffeomorphism 
$h\colon \Fc_m\to \Fc_m$.  Let $Y_m(\A)$ be the mapping 
torus of this diffeomorphism; we then have a smooth fibration 
\begin{equation}
\label{eq:milfib}
\xymatrix{\Fc_m \ar[r] &Y_m(\A) \ar[r] &S^1}.
\end{equation}

Now let $K=V\cap S^{2d+1}$ be the link of the singularity, 
and let $\nu(K)$ be a closed, regular neighborhood of $K$ 
inside the sphere. Then, as shown by Milnor in \cite{Mi}, 
there is a smooth fibration 
\begin{equation}
\label{eq:milfib2}
\xymatrixcolsep{18pt}
\xymatrix{\Fc_m \ar[r] & S^{2d+1}\setminus \nuc(K) 
\ar^(.62){Q_m/\abs{Q_m}}[rr] && S^1}.
\end{equation}
Furthermore, the local Milnor fibration \eqref{eq:milfib2} is 
equivalent, through a fiber-preserving diffeomorphism, to the 
fibration \eqref{eq:milfib}.

\begin{figure}
\centering
\begin{tikzpicture}[baseline=(current bounding box.center),scale=0.7]  
\clip (0,0) circle (3);
\draw[ style=thick, color=blue ] (-3,0) -- (3,0);
\draw[ style=thick, color=dkgreen] (0,-3) -- (0,3);
\draw[ ] (0,0) circle (2.4);
\draw[style=thick,color=orange, scale=1, domain=0.1:3,smooth,variable=\x] 
plot ({ \x },{1/\x});
\draw[style=thick,color=orange, scale=1, domain=-3:-0.1,smooth,variable=\x] 
plot ({ \x },{1/\x});
\end{tikzpicture}
\hspace*{0.3in}
\begin{tikzpicture}[baseline=(current bounding box.center),scale=0.72]  
\begin{scope}[even odd rule]
   \clip (0,0) circle(2.4) (1.8,0) circle (2.4);
   \fill[orange!40] (0,0) circle (2.4) (1.8,0) circle(2.4);
\end{scope}
\draw[dkgreen, very thick] (0,0)circle(2.4);
\draw[dkgreen!20!white, very thick] (-63:2.4) arc (-63:-73:2.4);
\draw[blue, very thick] (1.8,0)circle(2.4);
\draw[blue!20!white, very thick] (1.8,0)+(107:2.4) arc (107:117:2.4);
\draw[dkgreen, very thick] (63:2.4) arc (63:73:2.4); 
\end{tikzpicture}
\caption{Local Milnor fibration and closed Milnor fiber for $Q(\A)=xy$} 
\label{fig:mf local} 
\end{figure}
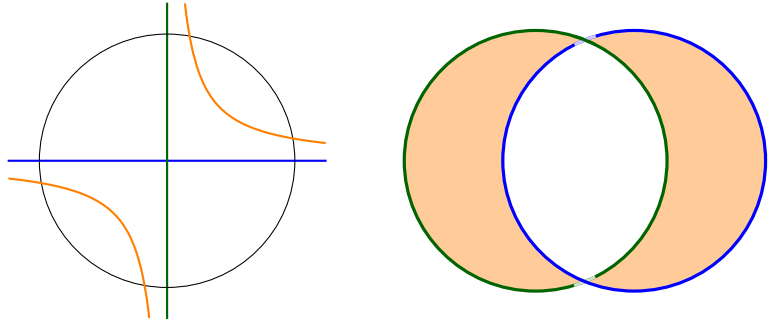

\begin{example}
\label{ex:seifert}
Let $\A$ be the arrangement in $\C^2$ defined by the polynomial 
$Q(\A)=z_1^n-z_2^n$.  Then $\Fc(\A)$ is a Seifert surface for the 
$n$-component Hopf link. This is a compact, connected, 
orientable surface of genus $\binom{n-1}{2}$, with $n$ open disks 
removed;  its boundary, $\bdF(\A)$, consists of $n$ disjoint circles.   

For $n=2$, the surface is the twisted cylinder depicted in 
Figure \ref{fig:mf local}, and the monodromy is a Dehn 
twist about the core of the cylinder.  (Although this diffeomorphism 
is isotopic to the identity, no such isotopy can be the identity on 
the two boundary circles.) 
\end{example}

\subsection{The double of the Milnor fiber}
\label{subsec:double milnor}

There are several ways in which the closed Milnor fiber 
and its boundary are related to the boundary manifold 
of an arrangement.  The next proposition (whose proof is 
adapted from the discussion in \cite[\S 2.8]{CS06}) details one 
such relationship, in the special case when the defining 
polynomial splits off a linear factor not involving the other 
variables. 

First recall a basic notion: the {\em double}\/ of a 
manifold $W$ with non-empty boundary $\partial W$ 
is the (closed) manifold $\partial(W\times [0,1])=
W\cup_{\partial W} W$. 

\begin{prop}
\label{prop:double mf}
Let $\A$ be a hyperplane arrangement in $\C^{d+1}$,   
defined by a polynomial of the form 
$Q(z_0,\dots, z_d)=z_0 Q_0(z_1,\dots, z_d)$.   
Let $\A_0$ be the arrangement in $\C^d$ defined by 
$Q_0$,  and let $\Fc_0 \to Y_0 \to S^1$ be the 
corresponding Milnor fibration. 
Then, the boundary manifold of $\A$ is diffeomorphic 
to the double of $Y_0$; moreover, $\bdU(\A)$ fibers over 
the circle, with fiber the double of $\Fc_0$.
\end{prop}

\begin{proof}
Let $V_0$ be the subvariety of $\C^d$ defined by $Q_0$. 
Its projective closure, $\overline{V}_0$, is the zero-set of $Q$; 
moreover, $\CP^{d}\setminus \overline{V}_0=\C^{d}\setminus V_0$.  
Forming the union of a regular neighborhood $\nu(V_0)$ 
with a tubular neighborhood of the hyperplane at infinity, , 
and rounding off corners, we obtain a regular neighborhood 
$\nu(\overline{V}_0)$.  Clearly,  $\CP^{d} \setminus \nuc(\overline{V}_0)$ 
is diffeomorphic to $D^{2d} \setminus (D^{2d}\cap \nuc(V_0))$. 
Hence, 
\begin{equation}
\label{eq:dec}
\bdU= \left( S^{2d-1} \setminus (S^{2d-1} \cap \nu(V_0))\right)\cup 
\left(D^{2d} \cap \partial \nu(V_0)\right)  .
\end{equation}

By the discussion from \S\ref{subsec:alt mf}, each of the two 
sides in the above decomposition is diffeomorphic to $Y_0$, 
and the gluing is done along their common boundary, $\partial Y_0$.  
Thus, $\bdU$ is the double of $Y_0$. The last assertion 
follows at once. 
\end{proof}

\begin{example} 
\label{ex:boolean 2}
Let $\B_n$ be the Boolean arrangement in $\C^{n}$.  
From Example \ref{ex:mf boolean}, we know that 
$F(\B_n)=(\C^*)^{n-1}$. Hence, 
$\Fc=(S^1\times [0,1])^{n-1}\cong T^{n-1}\times D^{n-1}$, 
and so $\bdF=T^{n-1}\times S^{n-2}$ 
(compare with \cite[\S 3, Exercise 1.15]{Di92}).  

Now note that $Q(\B_n)=z_0Q(\B_{n-1})$.  
We also know from 
Example \ref{ex:mf boolean} that the bundle 
$\Fc(\B_{n-1}) \to Y(\B_{n-1}) \to S^1$ is trivial.  
Applying Proposition \ref{prop:double mf}, then, 
shows that $\bdU = T^{n-1}\times  S^{n-2}$. 
\end{example}  

\begin{example}
\label{ex:bd mf near-pencil}
Let $\A$ be a near-pencil of $n$ planes in $\C^{3}$,   
defined by the polynomial $Q=z_0Q_0$, where 
$Q_0=z_1^{n-1}-z_2^{n-1}$. 
Then $Y_0$ admits a fibration over the circle 
(different from the Milnor fibration), whose fiber is 
$D^2$ with $n-1$ open disks removed, and whose 
monodromy is a Dehn twist about the boundary of $D^2$.  
Hence, $\bdU=S^1\times\Sigma_{n-2}$.  
As noted in \cite[Example 19.10.7]{NS}, we also 
have that $\bdF=S^1\times\Sigma_{n-2}$.  
\end{example} 

\subsection{The boundary of the Milnor fiber as a cyclic cover}
\label{subsec:bdFU}

We now describe a construction that realizes the boundary 
of the Milnor fiber associated to a hyperplane arrangement $\A$ 
as a regular, finite cyclic cover of the boundary manifold 
of the arrangement $\A$.   As usual, set $n=\abs{\A}$.

Recall that the Hopf fibration 
$\pi \colon \C^{d+1} \setminus \set{0} \to \CP^{d}$ 
restricts to a (trivializable) bundle map, 
$\pi\colon M\to U$, with fiber $\C^*$, 
which in turn restricts to a regular $\Z_n$-cover, 
$\pi\colon F\to U$.  

\begin{lemma}
\label{lem:bd cover}
The map $\pi \colon \C^{d+1} \setminus \set{0} \to \CP^{d}$ 
restricts to regular, cyclic $n$-fold covers, $\pi\colon \Fc \to \Uc$ 
and $\pi\colon \bdF \to \bdU$. 
\end{lemma}

\begin{proof}
Let $Q=Q(\A)$.   Recall that 
\begin{align}
\Fc&=\{z\in \C^{d+1} \mid \norm{z}\le 1 \text{ and } Q(z)=1\}, \text{ and} \\
\Uc&=\{[z] \in \CP^{2d} \mid \abs{Q(z)}^2 / \norm{z}^{2n}\ge 1\}.\notag
\end{align}

Thus, the map $\pi$ restricts to a map of pairs,
$\pi\colon (\Fc, \bdF) \to (\Uc,\bdU)$.  Hence, the 
cover $\pi\colon F\to U$ restricts to covers 
$\pi\colon \Fc \to \Uc$ 
and $\pi\colon \bdF \to \bdU$, and we are done. 
\end{proof}

Note that the inclusion $\Fc \to F$ is a fiber-preserving 
homotopy equivalence.  
Summarizing, we have a commuting ladder
\begin{equation}
\label{eq:cd}
\xymatrix{
\Z_n \ar@{=}[r] \ar[d] & \Z_n \ar@{=}[r] \ar[d] & \Z_n \ar[r] \ar[d] &
\C^* \ar@{=}[r] \ar[d] & \C^* \ar[d] \\
\bdF \ar^{\pi}[d] \ar[r]&  \Fc \ar^{\pi}[d] \ar^{\simeq}[r]&  F \ar^{\pi}[d] \ar[r]&  
M \ar[r] \ar^{\pi}[d] & \C^{d+1} \setminus\{0\} \ar^{\pi}[d] \\
\bdU  \ar[r]& \Uc  \ar^{\simeq}[r]& U  \ar@{=}[r]&  U \ar[r] & \CP^{d}
}
\end{equation}
where the horizontal arrows are inclusions, and the 
maps denoted by $\pi$ are principal bundles with 
fiber either $\C^*$ or $\Z_n$.

Assume now that $d=2$, and let $\P(\A)=\{\ell_1,\dots, \ell_n\}$ 
be the associated arrangement of lines in $\CP^2$. 
By Proposition \ref{prop:comm}, 
the group $\pi_1(\bdU)$ has generators $x_1,\dots , x_{n-1}$ 
corresponding to the meridians around the first $n-1$ lines, 
and generators $y_1,\dots ,y_s$ corresponding 
to the cycles in the associated graph $\Gamma$. 
 
\begin{prop}
\label{prop:bdf cover}
The $\Z_n$-cover $\pi\colon \bdF \to \bdU$ is classified by the 
homomorphism $\pi_1(\bdU)\surj \Z_n$ given by 
$x_{i} \mapsto 1$ and $y_{i} \mapsto 0$.  
\end{prop}

\begin{proof}
By Lemma \ref{lem:bd cover}, the cover $\pi\colon \bdF \to \bdU$ 
is the pullback along the inclusion $j\colon \bdU \to \Uc$ of the 
cover $\pi\colon \Fc \to \Uc$. By Theorem \ref{thm:mf}, this latter 
cover is classified by the homomorphism $\pi_1(\Uc)\surj \Z_n$, 
$x_i \mapsto 1$.   By Remark \ref{rem:bdry map}, 
the homomorphism $j_*\colon H_1(\bdU,\Z)\to H_1(\Uc,\Z)$ 
is given by $x_{i} \mapsto x_{i}$ and $y_{i} \mapsto 0$. 
The desired conclusion follows. 
\end{proof}

\begin{example}  
\label{ex:bmf pencil}
Let $\A$ be a pencil of $n+1$ planes in $\C^3$. 
From Example \ref{ex:bdry pencil}, we know that  
$\bdU=\sharp^{n} S^{1}\times S^{2}$.  
Since $\bdF \to \bdU$  is a cover with $n+1$ 
sheets, we find that $\pi_1(\bdF)=F_{n^2}$.  
Hence, by standard $3$-manifold topology, 
$\bdF=\sharp^{n^2} S^{1}\times S^{2}$. 
(See also  \cite[Example 19.10.6]{NS}.)
\end{example} 

\subsection{The characteristic polynomial of the monodromy}
\label{subsec:char poly bdF}

A detailed study of the boundary of the Milnor fiber of a non-isolated 
surface singularity was done by N\'{e}methi and Szil\'{a}rd in \cite{NS}.
When applied to arrangements in $\C^3$, their work yields the following 
result. 

\begin{theorem}[\cite{NS}]  
\label{thm:charpoly}
Let $\A$ be an arrangement of $n$ planes in $\C^3$.  
The characteristic polynomial of the algebraic monodromy 
acting on $H_1(\bdF,\C)$ is given by
\[
\Delta(t)=
\prod_{X\in L_2(\A)} (t-1) (t^{\gcd (\mu(X)+1,n)} -1)^{\mu(X)-1}.
\]
\end{theorem}

In particular, the first Betti number of $\bdF$ is determined solely 
by the M\"obius function of $L(\A)$:
\begin{equation}
\label{eq:b1delmf}
b_1(\bdF)= \sum_{X\in L_2(\A)}  (1+(\mu(X)-1)\gcd (\mu(X)+1,n)).
\end{equation}

This shows that the first Betti number of the boundary of the 
Milnor fiber is a much less subtle invariant than $b_1(F)$, in 
that $b_1(\bdF)$ depends only on the number and type of 
multiple points of $\P(\A)$, but not on their relative position. 
The next example illustrates this phenomenon.

\begin{example}
\label{ex:pappus bis}
Let $\A_1$ and $\A_2$ be the two arrangements from 
Example \ref{ex:pappus}.  Recall that the characteristic 
polynomials of the monodromy operators acting on 
the first homology of the Milnor fibers of these arrangements are 
$\Delta_1(t)=(t-1)^8(t^2+t+1)$ and $\Delta_2(t)=(t-1)^8$, 
respectively. On the other hand, the monodromy operators 
acting on the first homology of the boundaries of the 
respective Milnor fibers share the same characteristic 
polynomial, namely, 
$\Delta(t)=(t-1)^{27}(t^2+t+1)^9$. 
\end{example}

Nevertheless, there are pairs of arrangements for 
which the Milnor fibers have the same first Betti number, 
but the boundaries of the Milnor fibers have different first 
Betti numbers. Let us illustrate this claim with a simple example.

\begin{example}
\label{ex:4planes}
Let $\A_1$ be an arrangement of $4$ planes in general 
position and let $\A_2$ be a near-pencil of $4$ planes.  
Let $\Fc_1$ and $\Fc_2$ be the respective closed Milnor fibers.  
The monodromy operators acting on $H_1(\Fc_1, \C)$ and 
$H_1(\Fc_2, \C)$  have characteristic polynomial $\Delta(t)=(t-1)^3$, 
yet the monodromy operators acting on $H_1(\bdF_1, \C)$ and 
$H_1(\bdF_2, \C)$  have characteristic polynomials equal to 
$\Delta_1(t)=(t-1)^{6}$ and $\Delta_2(t)=(t-1)^{5}$, respectively. 
\end{example}

\subsection{The formality question}
\label{subsec:formal bd mf}

Next, we determine which Milnor fibers of plane arrangements 
have formal boundaries.

\begin{prop}
\label{prop:bdf formal}
Let $\A$ be an arrangement of planes in $\C^3$, and let $\bdF$ 
be the boundary of its Milnor fiber. The following are equivalent:
\begin{enumerate}
\item  \label{bf1} The manifold $\bdF$ is formal. 
\item  \label{bf2} $\A$ is either a pencil or a near-pencil. 
\item  \label{bf3} $\bdF$ is either $\sharp^{n^2} S^1\times S^2$ or 
$S^1\times \Sigma_{n-1}$, where $n=\abs{\A}-1$.
\end{enumerate}
\end{prop}

\begin{proof}
Implication \eqref{bf2} $\Rightarrow$ \eqref{bf3} was 
discussed in Examples \ref{ex:bmf pencil} and 
\ref{ex:bd mf near-pencil}, while implication 
\eqref{bf3} $\Rightarrow$ \eqref{bf1} is obvious. 

To prove \eqref{bf1} $\Rightarrow$ \eqref{bf2}, suppose 
$\A$ is neither a pencil, nor a near-pencil.  By 
Theorem \ref{thm:bdry}, then, the boundary 
manifold $\bdU$ is not formal.  
On the other hand, we know from Lemma \ref{lem:bd cover} 
that $\bdF$ is a regular, cyclic $n$-fold cover of $\bdU$. 
Hence, by Lemma \ref{lem:kfc}, the manifold $\bdF$ is not formal. 
\end{proof}

\subsection{Further considerations}
\label{subsec:h1 bdF}
In \cite[\S 24]{NS}, N\'{e}methi and A.~Szil\'{a}rd list a 
number of open problems regarding the boundary of the 
Milnor fiber of a non-isolated surface singularity.  
Problem 24.4.19 on the list reads as follows.

\begin{problem}
\label{quest:h1bdf}
Find a nice formula for the torsion of $H_1(\bdF,\Z)$. 
\end{problem}

The next example shows that torsion can indeed 
occur in the first homology of $\bdF$ (compare 
with \cite[Example 19.10.9]{NS}). 

\begin{example}  
\label{ex:gen 4}
Let $\A$ be a general position arrangement of $4$ planes 
in $\C^3$. As noted above, the characteristic polynomial 
of the algebraic monodromy acting on $H_1(\bdF,\C)$ is 
given by $\Delta(t)=(t-1)^{6}$.  Direct computation shows 
that, in fact, $H_1(\bdF, \Z)=\Z^{6} \oplus \Z_4$. 
\end{example}

For a generic arrangement of $n$ planes in $\C^3$, 
we expect that 
\begin{equation}
\label{eq:h1bdf gen}
H_1(\bdF, \Z)=\Z^{n(n-1)/2} \oplus \Z_{n}^{(n-2)(n-3)/2}.
\end{equation}
For an arbitrary arrangement $\A$, it would be interesting 
to see whether all the torsion in $H_1(\bdF(\A), \Z)$ 
consists of $\Z_n$-summands, where $n=\abs{\A}$. 

The next problem summarizes  Problems 24.2.1 
and 24.2.2 from \cite{NS}. 

\begin{problem}
\label{quest:coho bdf}
Determine the cohomology ring and the resonance varieties of 
the boundary of the Milnor fiber of an arrangement. 
\end{problem}

Let $\A$ be an arrangement.  
From Lemma \ref{lem:bd cover}, we know 
that the Hopf fibration restricts to a 
a regular, finite cover,  $\pi\colon \bdF \to \bdU$. 
Proposition \ref{prop:jumpcov}, then, insures that  
$\pi^*(\V^q_s(\bdU, \k))\subseteq \V^q_s(\bdF, \k)$ and  
$\pi^*(\RR^q_s(\bdU, \k))\subseteq \RR^q_s(\bdF, \k)$, 
for all $q\ge 0$ and $s\ge 1$.  

If $\A$ is a near-pencil of planes in $\C^3$, then all these inclusions are, 
in fact, equalities. In general, though, these inclusions are strict. 
For instance, if $\A$ is a pencil of $n+1$ planes in $\C^3$,  
and $n\ge 2$, then $\RR^1_1(\bdU, \k)=H^1(\bdU,\k)=\k^n$, 
whereas  $\RR^1_1(\bdF, \k)=H^1(\bdF,\k)=\k^{n^2}$. 

\appendix

\section{Cohomology jumping loci} 
\label{sect:cvs}

\subsection{Characteristic varieties}
\label{subsec:char var}
Let $X$ be a connected CW-complex.  Without loss of 
generality, we may assume $X$ has a single $0$-cell.  
Let $G=\pi_1(X,x_0)$ be the fundamental group of $X$, 
based at this $0$-cell.  We will assume throughout that $X$ 
has finitely many $1$-cells.   Of course, these cells must be 
attached at the unique $0$-cell $x_0$.  By cellular approximation, 
the based homotopy classes of the $1$-cells generate the 
fundamental group; hence, $G$ is a finitely generated group.
  
Now let $\k$ be an algebraically closed field, 
and let $\widehat{G}=\Hom(G,\k^*)=H^1(X,\k^*)$ be the affine 
algebraic group of $\k$-valued, multiplicative characters on $G$.  
The {\em characteristic varieties}\/ of $X$ over  
$\k$ are the jumping loci for homology with coefficients 
in $\k$-valued, rank-$1$ local systems on $X$:
\begin{equation}
\label{eq:cvs}
\V^q_s(X,\k)=\{\rho\in\Hom(G,\k^*) \mid  \dim_{\k} H_q(X,\k_\rho)\ge s\}.
\end{equation}

As long as $X$ has finite $q$-skeleton, 
these loci are Zariski closed subsets of $\widehat{G}$;   
moreover, we have a descending filtration 
\begin{equation}
\label{eq:filt}
\widehat{G}=\V^q_0(X,\k)\supseteq \V^q_1(X,\k)\supseteq\cdots
\supseteq\V^q_s(X,\k)\supseteq\cdots.
\end{equation}

As shown in \cite{PS-plms}, the depth-$1$ characteristic 
varieties $\V^q(X,\k)=\V^q_1(X,\k)$ behave well under 
direct products. More precisely, suppose $X_1$ and $X_2$ 
are two connected, finite-type CW-complexes, with fundamental 
groups $G_1$ and $G_2$.   Identify the character 
group of $\pi_1(X_1\times X_2)$ with 
$\widehat{G_1}\times \widehat{G_2}$.
Then, for all $q\ge 0$,  
\begin{equation}
\label{eq:cvprod}
\V^q(X_1\times X_2,\k)=\bigcup_{i=0}^{q} 
\V^i(X_1,\k)\times \V^{q-i}(X_2,\k).
\end{equation}

\subsection{Resonance varieties}
\label{subsec:rv}
 
Let $A=H^* (X,\k)$ be the cohomology algebra of $X$. 
If $H_1(X,\Z)$ has $2$-torsion, assume that $\ch(\k)\ne 2$.  
Then, for each $a\in A^1$, we have $a^2=0$, by graded-commutativity 
of the cup product. Thus, left-multiplication by $a$ defines a 
cochain complex,
\begin{equation}
\label{eq:aomoto}
\xymatrix{(A , \cdot a)\colon  \ 
A^0\ar^(.66){a}[r] & A^1\ar^{a}[r] & A^2  \ar[r]& \cdots},
\end{equation}

The jump loci for the cohomology of this complex define a 
natural filtration of the affine space $A^1=H^1(X,\k)$.  
The {\em resonance varieties}\/ of $X$ are the sets 
\begin{equation}
\label{eq:rvs}
\RR^q_s(X,\k)=\{a \in A^1 \mid \dim_{\k} 
H^q(A,a) \ge  s\}. 
\end{equation}
  
If $X$ has finite $q$-skeleton, the sets $\RR^q_s(X,\k)$ 
form a descending filtration of $A^1$ 
by homogeneous, Zariski closed subsets of $A^1$.  
Note that, if $A^q=0$, then $\RR^q_s(X,\k)=\emptyset$, 
for all $s>0$. In degree $0$, we have $\RR^0_1(X,\k)= \{ 0\}$,
and $\RR^0_s(X,\k)= \emptyset$, for $s>1$. 

The resonance varieties respect field extensions:  
if $\k\subseteq \mathbb{K}$, then
$\RR^q_s(X,\k)=\RR^q_s(X,\mathbb{K}) \cap H^1(X,\k)$.
Furthermore, as noted in \cite{PS-plms}, 
the depth-$1$ resonance varieties  
$\RR^q(X,\k)=\RR^q_1(X,\k)$ behave well under 
direct products:
\begin{equation}
\label{eq:resprod}
\RR^q(X_1\times X_2,\k)=\bigcup_{i=0}^{q} 
\RR^i(X_1,\k)\times \RR^{q-i}(X_2,\k).
\end{equation} 

\subsection{Jump loci in degree $1$}
\label{subsec:fox}

Given a finitely generated group  $G$,  
we may define its degree $1$ characteristic 
and resonance varieties as those of a 
presentation $2$-complex for the group.  It is readily 
checked that this definition does not depend on a 
choice of presentation for $G$.

If the group admits a finite presentation,  
$G=\langle x_1,\dots ,x_{n}\mid r_1,\dots ,r_m\rangle$,  
we may compute the sets $\V^1_s(G,\k)$ and 
$\RR^1_s(G,\k)$ directly from the presentation, 
by means of the Fox calculus.  The algorithm goes 
as follows.

Let $F_{n}$ be the free group with generators 
$x_1,\dots,x_{n}$,  and let $\epsilon\colon \k F_n\to \k$  
be the augmentation map, given by $\epsilon(x_i)=1$.   
For each $1\le j \le n$, there is a linear operator 
$\partial_j=\partial/\partial x_j \colon \k{F_{n}}\to \k{F_{n}}$,
known as the $j$-th Fox derivative, uniquely defined  
by the following rules: 
$\partial_j (1)=0$,
$\partial_j (x_i)=\delta_{ij}$, and
$\partial_j (uv)=\partial_j (u) \epsilon(v)+
u\partial_j (v)$.  

Now let $\alpha \colon \k{F_n}\to \k{G_{\ab}}$ be the ring 
morphism obtained by composing the abelianization 
map $\ab\colon G\to G_{\ab}$ with the presentation  
homomorphism $\phi\colon F_n\to G$, and extending 
linearly to group rings. We then define 
the Alexander matrix of the given presentation  
as the $m$ by $n$ matrix $\Phi=\Phi_G$ with entries 
\begin{equation}
\label{eq:phi}
\Phi_{ij}=\alpha(\partial_j r_i)
\end{equation}
in the group algebra $\k{G_{\ab}}$.
Note that $\k{G_{\ab}}$ is the coordinate ring of the 
algebraic group $\widehat{G}=\Hom(G, \k^*)$.  
The variety $\VV^1_s(G,\k)$, then, 
is the zero locus of the codimension~$s$ minors 
of $\Phi$, at least away from the trivial character, 
see \cite{Hi97, PS-plms}.  

A somewhat similar interpretation of the resonance varieties 
$\RR^1_s(G,\k)$ was given in \cite{MS00}, at least in the 
case when $G$ admits a commutator-relators presentation.  
More precisely, if all the relators $r_i$ belong to 
$[F_n,F_n]$, we may define the linearized Alexander matrix, 
$\Phi^{\operatorname{lin}}=\Phi^{\operatorname{lin}}_G$, 
as the $m$ by $n$ matrix with entries 
\begin{equation}
\label{eq:philin}
\Phi^{\operatorname{lin}}_{ij} = 
\sum_{k=1}^{n}\epsilon(\partial_k\partial_j r_i) y_k
\end{equation}
in the polynomial ring $S=\k[y_1,\dots, y_n]$.  
Of course, $S$ may be viewed as the coordinate ring of the 
affine space $H^1(G, \k)=\k^n$.  The variety $\RR^1_s(G,\k)$, 
then, is the zero locus of the codimension~$s$ minors 
of $\Phi^{\operatorname{lin}}$.  For instance,  if 
$n>1$, then $\RR^1_s(F_n,\k)=\k^n$, for all $s<n$.

\subsection{A naturality property}
\label{subset:natural}

Every group homomorphism $\varphi\colon G\to Q$ 
induces a morphism between character groups, 
$\hat\varphi\colon \widehat{Q} \to \widehat{G}$, 
given by $\hat\varphi (\rho)(g)=\varphi(\rho(g))$.     
 Likewise, $\varphi$ induces a homomorphism 
in cohomology, $\varphi^* \colon H^*(Q,\k)\to H^*(G,\Q)$.  
Clearly, if $\varphi$ is surjective, then both $\hat\varphi$ 
and $\varphi^1$ are injective.

The next proposition describes a nice functoriality 
property enjoyed by the characteristic and resonance 
varieties of groups.  The proposition extends results from 
\cite{PS-mathann, Su12}; the first part of the proof is modeled 
on the proof of \cite[Lemma 2.13]{Su12}, while the 
second part of the proof is modeled on the proof 
of \cite[Lemma 5.1]{PS-mathann}.

\begin{prop}
\label{prop:epi cv}
Let $G$ be a finitely generated group, and let $\varphi\colon G \surj Q$ 
be a surjective homomorphism.  Then, for each ground field $\k$, and 
each $s\ge 1$, the following hold.
\begin{enumerate}
\item \label{epi1}
The induced morphism 
between character groups, $\hat{\varphi} \colon 
\widehat{Q} \inj \widehat{G}$, restricts to an embedding 
$\V^1_s(Q,\k) \inj \V^1_s(G,\k)$. 
\item \label{epi2}
The induced morphism 
between cohomology groups, $\varphi^* \colon 
H^1(Q,\k)\inj H^1(G,\k)$, restricts to an embedding 
$\RR^1_s(Q,\k) \inj \RR^1_s(G,\k)$.
\end{enumerate}
\end{prop}

\begin{proof}
For the first part, let $\rho\colon Q\to \k^{*}$ be a character. 
The $5$-term exact sequence associated to the extension 
$1\to K \to G \to Q \to 1$ and the $\k{G}$-module 
$M=\k_{\rho \circ \varphi}$ ends in $H_1(G,M) \to H_1(Q,M_K)\to 0$, 
where $M_K=M/\{g m- m\mid  m\in M, g \in K\}$ is 
the module of coinvariants under the action of $K$. 
Clearly, $M_K$ is isomorphic as a $\k{Q}$-module to  
$\k_{\rho}$.  Hence, $\dim_{\k} H_1(G,\k_{\rho \circ \varphi})$ 
is bounded below by $\dim_{\k} H_1(Q,\k_{\rho})$.
Thus, if $\rho\in \VV^1_s(Q,\k)$, then $\hat\varphi(\rho) \in \VV^1_s(G,\k)$, 
and we are done.

For the second part, consider the commuting diagram
\begin{equation}
\label{eq:cup cd}
\xymatrixcolsep{30pt}
\xymatrix{
H^1(Q,\k) \wedge H^1(Q,\k)  \ar[r]^(.6){\mu_Q} 
\ar@{^{(}->}[d]^{\varphi^1\wedge \varphi^1} 
& H^2(Q,\k)  \ar[d]^{\varphi^2} \\
H^1(G,\k) \wedge H^1(G,\k)  \ar[r]^(.6){\mu_G} & H^2(G,\k) 
},
\end{equation}
where $\mu_Q$ and $\mu_G$ are the respective cup-product maps. 
Let $a$ be a nonzero element in $A^1=H^1(Q,\k)$.  By definition, 
$a$ belongs to $\RR^1_s(Q,\k)$ if there exist linearly 
independent elements $b_1,\dots, b_s\in A^1$ such that 
$a\wedge b_i\ne 0$ in $A^1\wedge A^1$, yet $\mu_Q(a,b_i)=0$.  
Clearly, $\varphi^1(b_1),\dots, \varphi^1(b_s)$ are linearly independent 
in $H^1(G,\k)$  and $\varphi^1(a)\wedge \varphi^1(b_i) \ne 0$, by injectivity 
of $\varphi^1$.  Moreover, $\mu_G(\varphi^1(a)\wedge \varphi^1(b_i))=0$, 
by commutativity of the diagram.  Hence, $\varphi^1(a)$ 
belongs to $\RR^1_s(G,\k)$. 
\end{proof}

\subsection{Alexander polynomial}
\label{subsec:alex}
As before, let 
$G=\langle x_1,\dots ,x_{n}\mid r_1,\dots ,r_m\rangle$
be a finitely presented group. 
For simplicity, we will assume $G_{\ab}$ is torsion-free 
(otherwise, we need to mod out its torsion subgroup), 
and will fix the coefficient field $\k=\C$. 

The {\em Alexander polynomial}\/ of the group $G$, denoted 
$\Delta_G$, is the greatest common divisor of the minors of 
size $n-1$ of the Alexander matrix $\Phi_G$.  As such, it is 
an element in the ring of Laurent polynomials 
$\C{G}_{\ab}$, well-defined up to units.

The Alexander polynomial defines a hypersurface, $V(\Delta_G)$, 
in the complex algebraic torus  $\widehat{G}$.  
As shown in \cite[Corollary 3.2]{DPS-imrn}, this hypersurface can be 
recovered from the characteristic variety $\VV_1(G)=\VV^1_1(G,\C)$.  
More precisely, either $\Delta_G=0$, or 
\begin{equation}
\label{eq:cva}
\cv (G)\setminus\set{1}= V(\Delta_G) \setminus\set{1}, 
\end{equation}
where $\cv (G)$ denotes the union of all codimension-one 
irreducible components of $\VV_1(G)$.
If $G$ is a $3$-manifold group, more can be said.

\begin{prop}[\cite{DPS-imrn}]
\label{prop:v1 3m}
Let $M$ be a compact, connected, orientable $3$-manifold 
without boundary.   Let $G=\pi_1(M)$, and suppose $G_{\ab}$ 
is torsion-free.  Then 
\begin{equation}
\label{eq:cva bis}
\VV_1(G)  \setminus\set{1} = V(\Delta_G) \setminus\set{1}.
\end{equation}
\end{prop}

In other words, at least away from the origin, the characteristic 
variety $\V_1(M)$ is the hypersurface defined by the 
Alexander polynomial $\Delta_M=\Delta_G$.

\section{Finite, regular abelian covers}
\label{sect:covers}

\subsection{Regular covers}
\label{subsec:regcov}

As before, let $X$ be a connected CW-complex with finite 
$1$-skeleton, and basepoint at the unique $0$-cell, $x_0$.   
Consider a covering map, $p\colon Y\to X$, with connected 
total space $Y$.  The cell structure on $X$ lifts to a cell 
structure on $Y$, in such a way that $p$ is a cellular map. 
Fix a basepoint $y_0\in p^{-1}(x_0)$; the induced homomorphism, 
$p_{\sharp}\colon \pi_1(Y,y_0)\to \pi_1(X,x_0)$, is injective.  
The assignment $p\leadsto  \im(p_{\sharp})$, then, 
establishes a one-to-one to correspondence between  
basepoint-preserving equivalence classes of connected 
covers of $X$ and subgroups of $\pi_1(X,x_0)$.

In this context, a special role is played by the {\em regular}\/ 
covers of our space $X$, that is, those covers 
$p\colon (Y,y_0)\to (X,x_0)$ for which $\im(p_{\sharp})$ 
is a normal subgroup of $G=\pi_1(X,x_0)$.  For such 
a cover, let $A=G/\im(p_{\sharp})$ be the quotient group,   
and let $\chi \colon G \to A$ be the canonical projection. 
We say that $A$ is the group of deck transformations, 
and $\chi$ is a classifying homomorphism for the cover $p$.
Conversely, if $\chi\colon G \surj A$ is an epimorphism to a 
(necessarily finitely generated) group $A$, there is a regular
cover $X^{\chi}\to X$, whose classifying homomorphism is $\chi$.

Now let $f\colon X'\to X$ be a map, and let $Y'\to X'$ be 
the cover obtained by pulling back the cover $X^{\chi}\to X$ 
along $f$. Then $Y'\to X'$ is a regular $A$-cover, classified 
by the homomorphism $\chi'=\chi\circ f_{\sharp}$.

\subsection{Homology of finite abelian covers}
\label{subsec:sak}
The next theorem records a formula for the homology 
groups $H_q(X^{\chi}, \k)$, in the case when the group 
of deck-transforma\-tions $A$ is finite, 
and $\k$ is an algebraically closed field of characteristic  
not dividing the order of $A$.  In the case $q=1$, this 
formula is well-known, and due to Libgober \cite{Li92}, 
Sakuma \cite{Sa} and E.~Hironaka~\cite{Hi97} 
for  $\k=\C$, and to Matei--Suciu~\cite{MS02} for 
other fields $\k$.  We provide a self-contained proof, 
following \cite[Theorem 2.5]{DeS12}.

Let $\chi\colon G\surj A$ be an epimorphism 
to a finite abelian group $A$, 
and let $\widehat{\chi}\colon \widehat{A}\to \widehat{G}$ 
be the induced morphism between character groups. 
Then $\widehat{\chi}$ is injective, and its image is
(non-canonically) isomorphic to $A$.  Furthermore, if 
$\rho\colon G\to \k^*$ 
is a character belonging to $\im(\widehat{\chi})$, there is 
a unique character $\iota_\rho\colon A\to \k^*$ such that 
$\rho=\iota_\rho\circ\chi$.

\begin{theorem}
\label{thm:betti cover}
Let $X^{\chi}\to X$ be the regular cover defined by an epimorphism 
$\chi$ from $G=\pi_1(X,x_0)$ to a finite abelian group $A$. 
Let $\k$ be an algebraically closed field of characteristic  
not dividing the order of $A$.  Then, for each $q\ge 0$,   
\begin{equation}
\label{eq:hcov}
\dim_\k H_q(X^{\chi},\k)=\sum_{s\ge 1}\abs{\im(\widehat{\chi}) 
\cap  \V^q_s(X,\k)}.
\end{equation}
\end{theorem}

\begin{proof}
The epimorphism $\chi$ puts a left $\k[G]$-module 
structure on the group algebra $\k[A]$.  By Shapiro's Lemma, 
$H_q(X^{\chi},\k)$ is isomorphic, as a right $\k[A]$-module, 
to $H_q(X,\k[A])$.
By our assumption on $\k$, the algebra $\k[A]$ is 
completely reducible, with one-dimensional, irreducible 
representations parametrized by $\widehat{A}$.   
As  a left $\k[G]$-module, each such representation 
is isomorphic to the image of 
$\widehat{\chi}\colon \widehat{A}\to\widehat{G}$. 
Thus, 
\begin{equation}
\label{eq:hqchi}
H_q(X, \k[A])\cong\bigoplus_{\rho\in\im(\widehat{\chi})} 
H_q(X,\k_\rho).
\end{equation}

By the remark above, $H_q(X,\k_\rho)$ is isomorphic, 
as a $\k[A]$-module, to $(\k_{\iota_\rho})^{\oplus b}$ 
where $b=\dim_\k H_q(X,\k_\rho)$.  
By definition, $\rho \in \V^q_s(X,\k)$ if and only if 
$\dim_{\k} H_q(X,\k_{\rho}) \ge s$.  Putting things together, 
we obtain an isomorphism of $\k[A]$-modules, 
\begin{equation}
\label{eq:decomp}
H_q(X^{\chi},\k) \cong \bigoplus_{s\ge 1}
\bigoplus_{\rho\in \im(\widehat{\chi}) 
\cap  \V^q_s(X,\k)} \k_{\iota_\rho}.  
\end{equation}
Taking dimensions on both sides completes the proof.
\end{proof}

\subsection{The characteristic polynomial of the algebraic monodromy}
\label{subsec:charpoly}

We now specialize to the case when $X^{\chi}\to X$ is a 
regular cover, classified by an epimorphism  
$\chi\colon \pi_1(X,x_0)\to A$ to a finite cyclic group $A$. 
Choose a generator $\alpha$ of $A$.   
Let $h=h_\alpha\colon X^{\chi}\to X^{\chi}$ be  
the corresponding monodromy automorphism, and let 
$h_*\colon H_q(X^{\chi},\k)\to H_q(X^{\chi},\k)$ be 
the induced map in homology.  As in \cite[\S 2.4]{DeS12}, 
then, we obtain the following application of Theorem \ref{thm:betti cover}.

\begin{theorem}
\label{thm:alg mono}
Assume $\ch(\k)\nmid \abs{A}$. Then, the characteristic polynomial of the 
algebraic monodromy, $\Delta^\k_{\chi,q}(t) = \det (t\cdot \id - h_*)$, 
is given by
\begin{equation}
\label{eq:delta chi}
\Delta_{\chi,q}^\k(t) =  \prod_{s\ge 1} 
\prod_{\rho\in \im(\widehat{\chi}) 
\cap  \V^q_s(X,\k)}  (t-\iota_{\rho}(\alpha)). 
\end{equation}
\end{theorem}

\begin{proof}
From the hypothesis, the $\k[A]$-module 
$H_q(X^{\chi},\k)=H_q(X,\k[A])$ is completely reducible.   
Therefore, the automorphism $h_*$ is diagonalizable.  
Furthermore, the eigenvalues of $h_*$, counted with 
multiplicity, are indexed by the irreducible $\k[A]$-modules 
appearing in decomposition \eqref{eq:decomp}, 
and we are done.
\end{proof}

In degree $q=1$, and for $\k=\C$, the polynomial 
$\Delta_{\chi}(t)=\Delta_{\chi,1}^\C(t)$ can be related to 
the Alexander polynomial of the group $G=\pi_1(X)$, as follows:
\begin{equation}
\label{eq:del chi}
\Delta_{\chi}(t) = (t-1)^c\Delta_G(t^{\chi_1}, \dots , t^{\chi_n}),
\end{equation}
where $c$ is an integer depending only on $G$.

\subsection{Jump loci of finite covers}
\label{subsec:jumps}

The next proposition and its corollary were proved in \cite{DP11} 
in the case when $\k$ has characteristic $0$.  For completeness, 
we include a proof, valid in a slightly more general context. 

\begin{prop}
\label{prop:jumpcov}
Let $A$ be a finite group, and let $p\colon Y\to X$ 
be a regular $A$-cover.  Suppose that $\ch(\k)\nmid \abs{A}$, 
and $\ch(\k)\ne 2$ if $H_1(X,\Z)$ has $2$-torsion. Then
$p^*(\V^q_s(X, \k))\subseteq \V^q_s(Y, \k)$ and  
$p^*(\RR^q_s(X, \k))\subseteq \RR^q_s(Y, \k)$, 
for all $q\ge 0$ and $s\ge 1$. 
\end{prop}

\begin{proof}
For the first claim, 
let $\LL$ be a $\k$-valued, rank-$1$ local system on $X$, and 
consider the Hochschild-Serre spectral sequence of the cover, 
$E^2_{ij}= H_i(A, H_j(Y, p^* \LL)) \Rightarrow H_{i+j}(X, \LL)$. 
Since $A$ is finite and $\ch(\k) \nmid \abs{A}$, we have that $E^2_{ij}=0$, 
for all $i>0$ and $j\ge 0$; thus, the spectral sequence collapses to 
an isomorphism, $H_{*}(X, \LL) \cong H_{*}(Y, p^*\LL)_A$.  
By duality, we obtain an injection, 
$p^{*}\colon H^{*}(X, \LL)\hookrightarrow H^{*}(Y, p^*\LL)$, 
and we are done. 

For the second claim, a standard transfer argument, using 
again the hypothesis on $\ch(\k)$, allows us to identify the 
induced algebra map in cohomology, 
$p^{*}\colon H^{*}(X, \k)\rightarrow H^{*}(Y, \k)$, with the inclusion 
$p^{*}\colon H^{*}(Y, \k)^{A} \hookrightarrow H^{*}(Y, \k)$. 
For a class $a \in H^{1}(Y, \k)^{A}$, the monodromy action of $A$  
on $H^{*}(Y, \k)$ gives rise to an action on the chain 
complex $(H^{*}(Y, \k), \cdot a)$, with fixed subcomplex 
$(H^{*}(X, \k), \cdot a)$.  We thus obtain an inclusion 
$H^*(H^{*}(X, \k), \cdot a) \hookrightarrow H^*(H^{*}(Y, \k), \cdot a)$, 
and we are done. 
\end{proof}

The proof of the second part of the proposition has an immediate 
corollary.

\begin{corollary}
\label{cor:trivnom}
With notation as above, suppose $A$ acts trivially on $H^1(Y,\k)$.  
Then $p^*\colon \RR^1_s(X, \k)\to \RR^1_s(Y, \k)$ is an isomorphism, 
for all $s\ge 1$. 
\end{corollary}

The assumption of the corollary is really necessary. For instance, 
if $X$ is a wedge of $n>1$ circles, then $Y$ is a wedge of 
$m=(n-1)\abs{A}+1$ circles; thus, if $\abs{A}>1$, then the map   
$p^*\colon \RR^1_1(X, \k)\to \RR^1_1(Y, \k)$ is 
a proper inclusion, sending $\k^n$ into $\k^m$.

\section{Formality}
\label{sect:formal}

\subsection{Formal spaces}
\label{subsec:formal}

Let $X$ be connected CW-complex with finite 
$1$-skeleton.  To such a space, Sullivan associated 
the commutative differential graded algebra (cdga)  
of polynomial differential forms on $X$ with coefficients 
in $\Q$, denoted $A_{\PL}(X, \Q)$.  

Let $H^*(X,\Q)$ be the rational cohomology algebra 
of $X$, endowed with the zero differential.  
The space $X$ is said to be {\em formal}\/ if there is a 
zig-zag of cdga morphisms connecting $A_{\PL}(X, \Q)$ 
to $H^*(X,\Q)$, with each such morphism inducing an 
isomorphism in cohomology. The space $X$ is merely 
{\em $k$-formal}\/  (for some $k\ge 1$) if each of these 
morphisms induces an isomorphism in degrees up to $k$, 
and a monomorphism in degree $k+1$.  If $X$ is a 
$k$-formal and $\dim X\le k+1$, then $X$ is formal. 

If $X$ is a smooth manifold, Sullivan's algebra may be 
replaced by the de~Rham algebra $\Omega^*_{\rm dR}(X)$ 
of smooth, differential forms on $X$.  Examples of formal 
spaces include suspensions, rational cohomology 
tori, surfaces, compact connected Lie groups, as well 
as their classifying spaces.   On the other hand, the 
only nilmanifolds which are formal are tori. Formality 
is preserved under wedges and products of spaces, 
and connected sums of manifolds.  

The $1$-minimality property of a space $X$ 
depends only on its fundamental group, $G=\pi_1(X,x_0)$.  
Alternatively, a finitely generated group $G$ is $1$-formal 
if and only if its Malcev Lie algebra (defined as the Lie algebra 
of primitive elements in the $I$-adic completion of the 
group-algebra $\Q[G]$) admits a quadratic 
presentation.   Examples of $1$-formal groups include free 
groups and free abelian groups of finite rank, surface groups, 
and groups with first Betti number equal to $0$ or $1$.  
The $1$-formality property is preserved under free products 
and direct products. 

A classical obstruction to formality is provided by the higher-order 
Massey products.  In particular, if $X$ is $1$-formal, and 
$\alpha_1,\alpha_2,\alpha_3$ are elements in $H^1(X,\Q)$ 
such that $\alpha_1 \alpha_2 =\alpha_2\alpha_3=0$, 
then the Massey triple product $\angl{\alpha_1,\alpha_2,\alpha_3}$ 
must vanish, modulo indeterminacy. 

One can make an analogous definition of $\k$-formality over an 
arbitrary field $\k$.  Again, an obstruction for $\k$-formality is 
provided by the Massey products in $H^*(X,\k)$.

\subsection{Formality in finite covers}
\label{subsec:formal funct}

Suppose $f\colon X\to Y$ is a map such that $f^*\colon H^*(Y,\Q) \to 
H^*(X,\Q)$ is an isomorphism up to degree $k$, and a monomorphism 
in degree $k+1$.  Then $X$ is $k$-formal if and only if $Y$ is $k$-formal. 

\begin{lemma}
\label{lem:kfc}
Let $p\colon Y\to X$ be a finite, regular cover.  If $Y$ is $k$-formal, 
then $X$ is also $k$-formal.  
\end{lemma}

\begin{proof} 
Let $\Gamma$ be the group of deck-transformations, so that 
$X=Y/\Gamma$.  As shown in \cite[Remark 3.30(2)]{FOT}, 
if $Y$ is formal, then $X$ is also formal.  The same argument 
applies to $k$-formality. 
\end{proof}

A particular case is worth mentioning. Suppose $G$ is a finitely 
generated group, and $H \triangleleft  G$ is a finite-index, normal  
subgroup.  If $H$ is $1$-formal, then $G$ is also $1$-formal. 
As the next example shows, the converse does not hold.

\begin{example}
\label{ex:non1f}
As is well-known, the Heisenberg group 
$H=\langle x,y \mid \text{$[x,y]$ central} \rangle$ 
admits non-trivial triple Massey products in $H^1(H,\Q)$; 
thus, $H$ is not $1$-formal.   
On the other hand, the semi-direct product  $G=H\rtimes \Z_2$ 
defined by the involution $x\mapsto x^{-1}$, $y\mapsto y^{-1}$ 
has $b_1(G)=0$, and so $G$ is $1$-formal.  
\end{example}

Nevertheless, the converse to Lemma \ref{lem:kfc} holds 
under an additional, rather restrictive condition. 

\begin{lemma}[\cite{DP11}]
\label{lem:kformal covers}
Let $p\colon Y \to X$ be a finite, regular cover, and 
suppose the group of deck-transformations acts trivially 
on $H^i(Y, \Q)$, for all $i\le k$. 
Then $Y$ is $k$-formal if and only if $X$ is $k$-formal.
\end{lemma}
 
\subsection{The tangent cone formula}
\label{subsec:formal tc}
 
Let $G=\pi_1(X,x_0)$. The homomorphism $\C\to \C^{*}$, 
$z\mapsto e^z$ induces a homomorphism  
$\exp\colon \Hom(G,\C) \to \Hom(G, \C^{*})=\widehat{G}$.  
Given a subvariety of $W\subseteq \widehat{G}$, 
define its {\em exponential tangent cone}\/ at the origin to be the set  
\begin{equation}
\label{eq:tau1}
\tau_1(W)= \{ z\in \Hom(G, \C) \mid \exp(\lambda z)\in W,\ 
\text{for all $\lambda \in \C$} \}.
\end{equation}

It turns out that $\tau_1(W)$ is a finite union of rationally 
defined linear subspaces (see \cite{DPS-duke}, and 
also \cite{Su12} for more details).
Moreover, $\tau_1(W)$ is included in $\TC_1(W)$, the 
usual tangent cone to $W$ at the origin.

Now fix an integer $s>0$, and consider the varieties 
$\VV^1_s(X) \subseteq \widehat{G}$ and 
$\RR^1_s(X) \subseteq H^1(X,\C)=\Hom(G,\C)$. 
By work of Libgober \cite{Li02}, the tangent 
cone to $\VV^1_s(X) $ at $1$ is included in $\RR^1_s(X)$.
Thus, we get a chain of inclusions, 
\begin{equation}
\label{eq:tc inc}
\tau_1(\VV^1_s(X))\subseteq  \TC_1(\VV^1_s(X))\subseteq \RR^1_s(X), 
\end{equation}
each of which is a proper inclusion, in general.  
The main connection between the formality property 
of a space and its cohomology jump loci is provided by the 
following theorem.  
 
\begin{theorem}[\cite{DPS-duke}]
\label{thm:tcone}
Let $X$ be a $1$-formal space.  For each $s>0$,
the following ``tangent cone formula" holds:
\begin{equation}
\label{eq:tc}
\tau_1(\VV^1_s(X))=\TC_1(\VV^1_s(X))=\RR^1_s(X).
\end{equation}
\end{theorem}

As a consequence, the irreducible components of 
$\RR^1_s(X)$ are all rationally defined subspaces, while 
the components of $\VV^1_s(X)$ passing through 
the origin are all algebraic subtori of the form $\exp(L)$, with $L$ 
running through the irreducible components of $\RR^1_s(X)$.

\newcommand{\arxiv}[1]
{\texttt{\href{http://arxiv.org/abs/#1}{arXiv:#1}}}

\newcommand{\arx}[1]
{\texttt{\href{http://arxiv.org/abs/#1}{arXiv:}}
\texttt{\href{http://arxiv.org/abs/#1}{#1}}}

\newcommand{\doi}[1]
{\texttt{\href{http://dx.doi.org/#1}{doi:#1}}}

\newcommand{\MR}[1]
{\href{http://www.ams.org/mathscinet-getitem?mr=#1}{MR#1}}

\newcommand{\MRh}[2]
{\href{http://www.ams.org/mathscinet-getitem?mr=#1}{MR#1 (#2)}}

\end{document}